\documentclass[a4paper,oneside]{amsart}
\usepackage{amssymb}
\usepackage{hyperref}
\usepackage{xcolor}
\usepackage{setspace}
\newtheorem{thm}{Theorem}[section]
\newtheorem{lem}[thm]{Lemma}
\newtheorem{cor}[thm]{Corollary}
\newtheorem{rem}[thm]{Remark}
 
\newtheorem{defn}[thm]{Definition}
\newtheorem{prop}[thm]{Proposition}
\def\Xint#1{\mathchoice
    {\XXint\displaystyle\textstyle{#1}}%
     {\XXint\textstyle\scriptstyle{#1}}%
     {\XXint\scriptstyle\scriptscriptstyle{#1}}%
     {\XXint\scriptstyle\scriptscriptstyle{#1}}%
	\!\int}
\def\XXint#1#2#3{{\setbox0=\hbox{$#1{#2#3}{\int}$}
	\vcenter{\hbox{$#2#3$}}\kern-.5\wd0}}
\newcommand{ \mint }{ \,\Xint- }

\newcommand{\Lra}{\Longrightarrow}
\newcommand{\br}{\mathbb{R}}

\newcommand{\sskip}{\smallskip}

\newcommand{\qqquad}{ \qquad \qquad}
\newcommand{\dotcup}{ \mathop{\dot{\cup}} }

\numberwithin{equation}{section}

\hypersetup{ colorlinks = true, urlcolor = blue, linkcolor = blue, citecolor = red }

\begin{document}
	
\title[linear elliptic systems from composite materials]{Gradient type estimates for linear elliptic systems from composite materials}
\everymath{\displaystyle}

\begin{abstract}
In this paper, we consider linear elliptic systems from composite materials where the coefficients depend on the shape and might have the discontinuity between the subregions.
We derive a function which is related to the gradient of the weak solutions and which is not only locally piecewise H\"{o}lder continuous but locally H\"{o}lder continuous.
The gradient of the weak solutions can be estimated by this derived function and we also prove local piecewise gradient H\"{o}lder continuity which was obtained by the previous results.
\end{abstract}

\author{Youchan Kim}
\email[Youchan Kim]{youchankim@uos.ac.kr}
\address[Youchan Kim]{Department of Mathematics, University of Seoul, Seoul 02504, Republic of Korea}

\author{Pilsoo Shin}
\email[Pilsoo Shin]{shinpilsoo.math@kgu.ac.kr}
\address[Pilsoo Shin]{Department of Mathematics, Kyonggi University, Suwon 16227, Republic of Korea}

\maketitle

\section{Introductions}

In this paper, we study linear elliptic systems from composite materials. 
First, we describe our model problem in this paper. For composite materials, the physical characteristics of the medium are divided into a finite number of components or subregions. So let $\Omega \subset \br^{n}$ $(n \geq 2)$ be a bounded domain and $\Omega_{1}, \cdots, \Omega_{l} \subset \Omega$ be the mutually disjoint subregions (of $\Omega)$ with $\Omega_{0} := \Omega \setminus (\Omega_{1} \cup \cdots \cup \Omega_{l})$. Here, the subregions $\Omega_{0}$, $\Omega_{1}$, $\cdots$, $\Omega_{l}$ represent each component of a composite material $\Omega$. Since the physical characteristics are regular in each component $\Omega_{0}, \cdots, \Omega_{l}$, we consider the following linear elliptic systems.

For $C^{1,\gamma}$-domains $\Omega_{1}$, $\cdots$, $\Omega_{l}$ and $\Omega$, let $u \in W^{1,2} \left( \Omega, \mathbb{R}^{N} \right)$ be a weak solution of
\begin{equation}\label{}
\partial_{\alpha} \left[ A_{ij}^{\alpha \beta}(x) \partial_{\beta} u^{j} \right] = \partial_{\alpha} F_{\alpha}^{i}
\quad \text{ in } \quad \Omega,
\end{equation}
for  $1 \leq \alpha, \beta \leq n$ and $1 \leq i, j \leq N$, where
\begin{equation}\label{} 
\lambda |\xi|^{2} 
\leq A_{ij}^{\alpha \beta }(x) \xi_{\alpha}^{i} \xi_{\beta}^{j} 
\qquad \text{and} \qquad
\big| A_{ij}^{\alpha \beta}(x) \big| \leq \Lambda 
\qquad
\big( x \in \Omega, \ \xi \in \br^{nN} \big),
\end{equation}
for some positive constant $\lambda$ and $\Lambda$. Because the physical characteristics are regular in each subregion $\Omega_{0}, \cdots, \Omega_{l} \subset \Omega$, we assume that $A_{ij}^{\alpha \beta}, \, F_{\alpha}^{i} \in C^{\mu}(\Omega_{k})$ $ \big( k \in \{ 0, \cdots, l \} \big)$ for any $ 1 \leq \alpha, \beta \leq n$ and $1 \leq i,j \leq N$. We remark that only the interior estimates will be obtained in this paper, and one may not impose any regularity condition on the boundary data.

The regularity theory related to composite materials is motivated by the numerical observation \cite{BIABSPLK1} that the gradient bound $|Du|$ is independent of the distance between the subdomains for certain homogeneous isotropic linear systems of elasticity. Bonnetier and Vogelius \cite{BEVM1} considered a geometric structure that two touching disk inside a disk to obtain a gradient boundedness of the weak solution. Then Li and Vogelius \cite{LYVM1} obtained global Lipschitz regularity and global piecewise gradient H\"{o}lder continuity for linear elliptic equations in general geometry, say mutually disjoint subdomains $\Omega_{1}$, $\cdots$, $\Omega_{l}$ inside the domain $\Omega$. Later, Li and Nirenberg \cite{LYNL1} extended \cite{LYVM1} by obtaining local Lipschitz regularity and local piecewise gradient H\"{o}lder continuity for linear elliptic systems. Here, gradient piecewise H\"{o}lder continuous means that $Du$ is H\"{o}lder continuous in each $\Omega_{k}$ for any $k \in \{ 0, \cdots, l \}$ and local gradient piecewise H\"{o}lder continuous means that $Du$ is H\"{o}lder continuous in each $\Omega_{k} \cap \Omega'$ for any $k \in \{ 0, \cdots, l \}$ and  $\Omega' \subset \subset \Omega$.

In \cite{LYNL1,LYVM1}, they obtained piecewise gradient H\"{o}lder continuity in the point-wise sense by using Schauder type approach. In this paper, we find a suitable function related to the gradient of the weak solution which is not only locally piecewise  H\"{o}lder continuous but locally H\"{o}lder continuous. (See Theorem \ref{Theorem_of_U}.)  Then by using that the coefficients are H\"{o}lder continuous in each component, one can also show that the gradient of the weak solution is piecewise H\"{o}lder continuous which was already obtained in \cite{LYNL1,LYVM1}. (See Corollary \ref{Corollary_of_Du}). The purpose for obtaining such a result is to derive a gradient type estimate which can be used for the open problem suggested by Li and Nirenberg in \cite{LYNL1} which is related to piecewise H\"{o}lder continuity of higher order derivatives for weak solutions to elliptic equations from composite materials. We will obtain the desired gradient H\"{o}lder type estimate by using the excess functional, say $\mint_{Q_{r}(z)} \big| g-(g)_{Q_{r}(z)} \big| \, dx$, which appears in Campanato type embeddings.

\smallskip

We introduce the notation in this paper. Let $y=(y^{1},y') \in \mathbb{R}^{n}$ be a typical point, and $r>0$ be a size. 
\begin{enumerate}
\item $Q_{r}'(y') = \left \{ x' = (x^{2}, \cdots , x^{n})  \in \mathbb{R}^{n-1} : \max_{2 \leq i \leq n} |x^{i}-y^{i}| < r \right \}$ is the open cube in $\mathbb{R}^{n-1}$ with center $y'$ and size $r$. Also we denote $Q_{r}' = Q_{r}'(0')$.
\item $Q_{r}(y)  =  \left \{ x \in \mathbb{R}^{n} : \max_{1 \leq i \leq n} |x^{i}-y^{i}| < r \right \} = (y^{1}-r,y^{1}+r) \times Q_{r}'(y')$ is the open cube in $\mathbb{R}^{n}$  with center $y$ and size $r$.  Also we denote $Q_{r} = Q_{r}(0)$.
\item For a function $g(x)$ in $\mathbb{R}^n$, 
$$
(g)_{U} = \mint_U g(x) \, dx = \frac{1}{|U|} \int_{U} g(x) \, dx,
$$
where $U$ is an open subset in $\mathbb{R}^n$ and $|U|$ is the $n$-dimensional Lebesgue measure of $U$.
\end{enumerate}

\smallskip

A typical composite material $\Omega \subset \mathbb{R}^{n}$ ($n \geq 2$) composed of $C^{1,\gamma}$-boundaries can be described as the following. Let $\Omega_{1} \subset \Omega$ be a connected component in $\Omega$ and $\Omega_{2}, \cdots, \Omega_{l} \subset \Omega$ be the components surrounded by $\Omega_{1}$. Without loss of generality, we may assume that $\Omega_{1} \cup \cdots \cup \Omega_{l}$ and $\Omega_{2} \cup \cdots \cup \Omega_{l}$ are open. Let $D_{2}, \cdots, D_{m} \subset \Omega$ be the disjoint open connected components of $\Omega_{2} \cup \cdots \cup \Omega_{l}$. Then for the open set  $D_{1}=\Omega_{1} \cup \cdots \cup \Omega_{l}$, we have that $\Omega_{1} = D_{1} \setminus (D_{2} \cup \cdots \cup D_{m})$. If $D_{1},D_{2},\cdots,D_{m}$ are $C^{1,\gamma}$-domains then we may say that the component(or the subregion) $\Omega_{1}$ is composed of $C^{1,\gamma}$-boundaries. For this geometry, one can prove that for a sufficiently small scale, there exists a coordinate system such that the boundary of the subregions become graphs, see for instance \cite{KYSP1}. For the composite geometry related to $C^{1,\gamma}$-domains, we also refer to \cite{LYNL1,LYVM1}.

With the description in the previous paragraph, we may assume that the cube $Q_{r}(z)$ can be divided into the components(of the composite material) or the subregions by using $C^{1,\gamma}$-graph functions $\left\{ \varphi_{k} : k \in K_{+} \right\}$. Here, $ K_{+} $ will be used to denote the index set of graph functions. The subregions in $Q_{r}(z)$ will be denoted as $Q_{r}^{k}(z)$ with a index set $K$. We remark that for the index of the components $K = \{ k_{-}, k_{-}+1, \cdots, k_{+} \}$, there is an one more element $\left\{ k_{+} +1 \right \}$ in the index set of the graph functions $K_{+} = K \dotcup \left\{ k_{+} +1 \right \}$. We also remark that there can be only one element in $K = \{ k_{-}, k_{-}+1, \cdots, k_{+} \}$. In Definition \ref{composite cube}, $\dotcup_{k \in K} U_{k} $ denotes disjoint union meaning that $\dotcup_{k \in K} U_{k} $ is the union of the sets $\{ U_{k} : k\in K \}$ and that $\{ U_{k} : k\in K \}$ are mutually disjoint.

\begin{defn}\label{composite cube}
We say that $\left( Q_{r}(z) , \left\{ \varphi_{k} : k \in K_{+} \right\} \right)$ is a composite cube if the graph functions $\varphi_{k} \in C^{1,\gamma} \left( Q_{r}'(z') \right)$ $(k \in K_{+})$ with $K = \left\{ k_{-}, k_{-}+1, \cdots, k_{+} \right\}$ and $K_{+} = K \dotcup \left\{ k_{+} +1 \right \}$ satisfy that
\begin{equation*}\label{}
\varphi_{k}(x')  \leq  \varphi_{k+1}(x')
\qqquad \left(  x \in Q_{r}'(z), \, k \in K \right),
\end{equation*}
and
\begin{equation*}\label{}
Q_{r}(z) = \dotcup_{k \in K} Q_{r}^{k}(z),
\end{equation*}
where $Q_{r}^{k}(y) : = \big\{ (x^{1},x') \in Q_{r}(y) : \varphi_{k}(x') < x^{1} \leq \varphi_{k+1}(x') \big \}$ $(k \in K)$.
\end{defn}

For the composite cube inside the cube, we use the following natural definition.

\begin{defn}
For the composite cube $\left( Q_{r}(z) , \left\{ \varphi_{k} : k \in K_{+} \right\} \right)$, we denote
\begin{equation*}\label{} 
Q_{\rho}^{k}(y) : = \big\{ (x^{1},x') \in Q_{\rho}(y) : \varphi_{k}(x') < x^{1} \leq \varphi_{k+1}(x') \big \}
\qquad \qquad ( k \in K ),
\end{equation*}
for any $Q_{\rho}(y) \subset Q_{r}(z)$.
\end{defn}

\begin{rem}
If $\left( Q_{r}(z) , \left\{ \varphi_{k} : k \in K_{+} \right\} \right)$ is a composite cube then for any $Q_{\rho}(y) \subset Q_{r}(z)$, $\left( Q_{\rho}(y) , \left\{ \varphi_{k} : k \in K_{+} \right\} \right)$ is also a composite cube. Moreover, we have that $\inf_{Q_{r}'(z')} \left| \varphi_{k_{-}} \right| \geq r$ and $\inf_{Q_{r}'(z')} \left| \varphi_{k_{+}+1} \right| \geq r$ for $K = \left\{ k_{-}, k_{-}+1, \cdots, k_{+}, k_{+}+1 \right\}$.
\end{rem}

To state our main theorem, we define a vector-valued function $\pi' : Q_{r}(z) \to \br^{n-1}$ which is  naturally induced from our geometry.

\begin{defn}\label{derivative of flow}
For the composite cube $\left( Q_{r}(z) , \left\{ \varphi_{k} : k \in K \right\} \right)$, define $T_{k} : Q_{r}(z) \to [0,1]$ $ (k \in K)$ as
\begin{equation}\label{GS220} 
T_{k}(x^{1},x') = \frac{x^{1} - \varphi_{k}( x')}{\varphi_{k+1}(x') - \varphi_{k}(x')} 
~ \text{ in } ~ Q_{r}^{k}(z)
\qqquad (k \in K).
\end{equation}
Then `the derivative of the naturally induced flow' $\pi : Q_{r}(z) \to \br^{n}$ is defined as
\begin{equation}\label{GS250}
\pi = (-1,\pi') = ( -1, \pi_{2}, \cdots, \pi_{n} )
\end{equation}
where
\begin{equation}\begin{aligned}\label{GS260}
\pi_{\alpha}(x) = D_{\alpha}\varphi_{k+1}( x') \cdot T_{k}(x)
+ D_{\alpha}\varphi_{k}( x')  \cdot [1-T_{k}(x)]
~ \text{ in } ~ Q_{r}^{k}(z),
\end{aligned}\end{equation}
for any $k \in K$ and  $\alpha \in \{ 2, \cdots, n \}$. 
\end{defn}

\begin{rem}
At a boundary point of subregions, say $x = \left(\varphi_{k}(x'), x' \right) \in Q_{r}(z)$ $(k \in K)$, we have that $\pi'(x) = D_{x'}\varphi_{k}(x')$. Moreover, from the point $\left(\varphi_{k}(x'), x' \right)$ to the point $\left(\varphi_{k+1}(x'), x' \right)$, the value of $\pi'$ changes linearly from $D_{x'}\varphi_{k}(x')$ to $D_{x'}\varphi_{k+1}(x')$. So the derivative of the naturally induced flow remains same for any subset $Q_{\rho}(y) \subset Q_{r}(z)$.
\end{rem}

\begin{rem}
In Definition \ref{derivative of flow}, the reason for using the phrase `the derivative of the naturally induced flow' is that for the flow $\psi : Q_{r-\epsilon}(z) \times (-\epsilon,\epsilon) \to \br^{n}$ defined as 
\begin{equation*}\label{}
\psi(x,t') 
= [ \varphi_{k+1} (x'+ t') -\varphi_{k+1}( x') ] \cdot T_{k}(x)
+ [ \varphi_{k} ( x'+ t') - \varphi_{k}( x') ] \cdot [ 1-T_{k}(x) ],
\end{equation*}
we have that $\pi'(x) = \partial_{t} \psi(x,t') \big|_{t=0}$. In \cite{KY1}, the concept of the (time) derivatives of the flow will be used to solve the open problem suggested by Li, Nirenberg and Vogelius in \cite{LYNL1,LYVM1}. But there are some technical difficulties related to the discontinuities and we need some additional argument for proving piecewise smoothness for piecewise smooth coefficients, see the introduction of \cite{KY1}.
\end{rem}

We assume that $A_{ij}^{\alpha \beta} : \br^{n} \to \br^{Nn} \times \br^{Nn} $  $\left(  1 \leq i,j \leq N, ~ 1 \leq \alpha, \beta \leq n \right)$ satisfy that
\begin{equation}\label{ell1}
\lambda |\zeta|^{2} \leq \sum_{ 1 \leq i,j \leq N} \sum_{1 \leq \alpha, \beta \leq n} A_{ij}^{\alpha \beta}(x) \zeta_{\alpha}^{i} \zeta_{\beta}^{j} 
\qquad \left( x \in \br^{n}, ~ \zeta \in \br^{Nn} \right).
\end{equation}
and
\begin{equation}\label{ell2}
\left| A_{ij}^{\alpha \beta} (x) \right|
\leq \Lambda 
\qquad \qquad
\left( x \in \br^{n}, ~ 1 \leq i,j \leq N, ~ 1 \leq \alpha, \beta \leq n \right).
\end{equation}

We now state the main results. In Theorem \ref{Theorem_of_U}, we focus on the cube $Q_{3R}$ and assume that the minimum of the absolute value of the graph functions are smaller than $4R$. If the the minimum of the absolute value of the graph functions such as $\varphi_{k_{-}}$ and $\varphi_{k_{+}+1}$ are greater than $4R$, then one can choose a new  graph functions $\varphi_{k_{-}}$ and $\varphi_{k_{+}+1}$ in the way that the regions of the composite cube  remain same. For the simplicity of the notation, we set
\begin{equation}\label{mu}
\mu = \frac{\gamma}{2(\gamma+1)} \in \left( 0, \frac{1}{4} \right].
\end{equation}
We also explain the constant $\mu = \frac{\gamma}{2(\gamma+1)}$ in the condition $A_{ij}^{\alpha \beta}, F_{\alpha}^{i} \in C^{\mu} \left( Q_{2R}^{k} \right)$. The main tool for handling the composite domain or the composite cube is the estimate on the non-crossing boundaries of the regions or the non-crossing graph functions, (see for instance \cite[Section 5]{LYVM1} or \cite[Section 4]{LYNL1}) which comes from that the boundary of the components in the composite materials does not cross each other. In the estimate, we lose some regularity (see Lemma \ref{GSS300}) and the main equation behaves like elliptic equations with $C^{\mu}$-H\"{o}lder continuous coefficients.  

\begin{thm}\label{Theorem_of_U}
For the composite cube $\left( Q_{3R} , \left\{ \varphi_{k} : k \in K_{+} \right\} \right)$, assume that
\begin{equation}\label{minimum}
\inf_{x' \in Q'_{3R}} |\varphi_{k}| < 4R, \quad  \left( k \in K_{+} \right) . 
\end{equation}
Also for $\mu$ in \eqref{mu}, assume  \eqref{ell1}, \eqref{ell2} and that 
\begin{equation*}\label{}
A_{ij}^{\alpha \beta}, F_{\alpha}^{i} \in C^{\mu} \left( Q_{2R}^{k} \right),
\end{equation*}
for any $1 \leq \alpha, \beta \leq n$, $ 1 \leq i,j \leq N$ and $k \in K$. Let $u$ be a weak solution of
\begin{equation*}\label{}
D_{\alpha} \left[ A_{ij}^{\alpha \beta} D_{\beta}u^{j} \right] = D_{\alpha} F^{i}_{\alpha} \quad \text{ in } \quad Q_{r}(z),
\end{equation*}
and define $U : Q_{2R} \to \br^{Nn}$ as $U = \left( U^{1}, \cdots, U^{N} \right)^{T}$ and
\begin{equation*}\label{} 
U^{i} = \left( \sum_{1 \leq \alpha \leq n } \pi_{\alpha} \left[ \left( \sum_{1 \leq j \leq N} \sum_{1 \leq \beta \leq n } A_{ij}^{\alpha \beta} D_{\beta} u^{j} \right) - F_{\alpha}^{i} \right], D_{x'}u^{i} + \pi' \, D_{1} u^{i} \right),
\end{equation*}
for any $1 \leq i \leq N$, where $\pi : Q_{2R} \to \br^{n}$ is defined in \ref{derivative of flow}. Then we have that
\begin{equation}\begin{aligned}\label{Theorem_of_U_Estimate1}
& \frac{1}{\rho^{2\mu}}  \mint_{Q_{\rho}(z) } \left| U  - (U)_{Q_{\rho}(z) } \right|^{2} \, dx  \\
& \quad \leq  \frac{c}{ r^{ 2\mu }}
\mint_{Q_{r}(z) } \left| U - (U)_{Q_{r}(z) } \right|^{2} \, dx \\
& \qquad +  \frac{c}{ R^{2\mu} } \left[   \mint_{Q_{r}(z) } |U|^{2} \, dx 
+   \| F \|_{L^{\infty}(Q_{r}(z))}^{2} + R^{2\mu} \sup_{ k \in K } [ F ]_{C^{\mu}(Q_{r}^{k}(z))}^{2}  \right],
\end{aligned}\end{equation}
and
\begin{equation}\label{Theorem_of_U_Estimate2} 
\mint_{Q_{\rho}(z) } \left| U  \right|^{2} \, dx
\leq c \left( \mint_{Q_{r}(z) } \left| U  \right|^{2} \, dx
+ \| F \|_{L^{\infty}(Q_{2R})}^{2} 
+ r^{2\mu} \sup_{ k \in K } [ F ]_{C^{\mu}(Q_{r}^{k}(z))}^{2}     \right),
\end{equation}
for any $z \in Q_{R}$ and $0 < \rho \leq r \leq R$.  
Here, the constant $c$ depends on the terms $n$, $N$, $\lambda$, $\Lambda$, $\sup_{k \in K_{+}} \| D_{x'} \varphi_{k} \|_{L^{\infty}(Q_{3R}')}$, $R^{\gamma} \sup_{k \in K_{+}} [ D_{x'}\varphi_{k} ]_{C^{\gamma}(Q_{3R}')}$, $ R^{\mu} \sup_{k \in K} \left[ A_{ij}^{\alpha \beta} \right]_{C^{\mu} \left( Q_{2R}^{k} \right)}$ and the number of elements in the set $K$. 
\end{thm}

We will later prove that $\pi' \in C^{2\mu}(Q_{2R})$. We will compare $Du$ and $U$ later in Lemma \ref{HODS600} and Lemma \ref{HODS700}. So by using Theorem \ref{Theorem_of_U}, one can obtain Lipschitz regularity and piecewise gradient H\"{o}lder estimates which was already obtained in \cite{LYNL1,LYVM1}.

\begin{cor}\label{Corollary_of_Du}
For the composite cube $\left( Q_{3R} , \left\{ \varphi_{k} : k \in K_{+} \right\} \right)$, assume that
\begin{equation*}\label{}
\inf_{x' \in Q'_{3R}} |\varphi_{k}| < 4R, \quad   \left( k \in K_{+} \right) .
\end{equation*}
Also for $\mu$ in \eqref{mu}, assume  \eqref{ell1}, \eqref{ell2} and that 
\begin{equation*}\label{}
A_{ij}^{\alpha \beta}, F_{\alpha}^{i} \in C^{\mu} \left( Q_{2R}^{k} \right),
\end{equation*}
for any $1 \leq \alpha, \beta \leq n$, $ 1 \leq i,j \leq N$ and $k \in K$. Let $u$ be a weak solution of
\begin{equation*}\label{}
D_{\alpha} \left[ A_{ij}^{\alpha \beta} D_{\beta}u^{j} \right] = D_{\alpha} F^{i}_{\alpha} \quad \text{ in } \quad Q_{2R}.
\end{equation*}
Then we have that $Du \in L^{\infty}(Q_{R})$ and $Du \in C^{\gamma} \left( Q_{R}^{l} \right)$ for any $l \in K$ with the estimates
\begin{equation*}\label{} 
\| Du \|_{L^{\infty}(Q_{R})}^{2}
\leq c \left( \mint_{ Q_{2R} } \left| Du  \right|^{2} \, dx
+\| F \|_{L^{\infty}(Q_{2R}^{k})}^{2} 
+ R^{2\mu} \sup_{ k \in K } [ F ]_{C^{\mu}(Q_{2R}^{k})}^{2}     \right).
\end{equation*}
and 
\begin{equation*}\label{}
[Du]_{ C^{\mu} \left(Q_{R}^{l} \right) }^{2} 
 \leq  \frac{c}{R^{2\mu}} \left( \mint_{Q_{2R} } |Du|^{2} \, dx 
+ \| F \|_{L^{\infty}(Q_{2R})}^{2} + R^{2\mu} \sup_{ k \in K } [ F ]_{C^{\gamma}(Q_{2R}^{k}    )}^{2}  \right),
\end{equation*}
for any $l \in K$. Here, the constant $c$ depends on $n$, $N$, $\lambda$, $\Lambda$, $\sup_{k \in K_{+}} \| D_{x'} \varphi_{k} \|_{L^{\infty}(Q_{3R}')}$, $R^{\gamma} \sup_{k \in K_{+}} [ D_{x'}\varphi_{k} ]_{C^{\gamma}(Q_{3R}')}$, $ R^{\mu} \sup_{k \in K} \left[ A_{ij}^{\alpha \beta} \right]_{C^{\mu} \left( Q_{2R}^{k} \right)}$ and the number of elements in the set $K$. 
\end{cor}

We refer to Calder\'{o}n-Zygmund type estimate for linear equations \cite{JYKY1,JYKY3,HFLDWL1,UK1} and $p$-Laplace type equations \cite{JYKY2,ZC1}. Also there is an another direction about for elliptic equation from composite materials which is the blow up phenomenon for such as two almost touching fibres have the extreme ($0$ or $\infty$) conductivities, see \cite{AHKHLM1,KHLMYK1}.

\smallskip

For the sake of the convenience, unless specified, we employ the letter $c \geq 1$ throughout this paper to denote any constants that can be explicitly computed in terms of the constants $n$, $N$, $\lambda$, $\Lambda$, $\sup_{k \in K_{+}} \| D_{x'} \varphi_{k} \|_{L^{\infty}(Q_{3R}')}$, $R^{\gamma} \sup_{k \in K_{+}} [ D_{x'}\varphi_{k} ]_{C^{\gamma}(Q_{3R}')}$, $ R^{\mu} \sup_{k \in K} \left[ A_{ij}^{\alpha \beta} \right]_{C^{\mu} \left( Q_{2R}^{k} \right)}$ $(1 \leq \alpha, \beta \leq n, \ 1 \leq i,j \leq n)$ and $|K|$ the number of elements in the set $K$. Thus the exact value denoted by $c$ may change from line to line in a given computation.

\section{Estimates on the derivative of the naturally induced flow}

In this section we will prove that the derivative of the naturally induced flow $\pi$ in \eqref{GS250} is locally H\"{o}lder continuous in the cube.

\subsection{Decay estimate for the graph functions}

To handle two non-crossing graph functions in $\left\{ \varphi_{k} : k \in K_{+} \right\} $, we  use following result, which naturally holds from our geometric settings (also see \cite[Section 5]{LYVM1} or \cite[Section 4]{LYNL1}).

\begin{lem}\label{GSS300}
Suppose that $\varphi_{k}, \varphi_{l} : C^{1,\gamma}
(Q_{r+\rho}') \to \mathbb{R}$ satisfy that 
\begin{equation*}
[ D_{x'}\varphi_{k} ]_{C^{\gamma}(Q_{r+\rho}')}, [ D_{x'}\varphi_{l} ]_{C^{\gamma}(Q_{r+\rho}')} \leq c_{1},
\end{equation*}
\begin{equation*}
\| \varphi_{k} \|_{L^{\infty}(Q_{r+\rho}')}, \| \varphi_{l} \|_{L^{\infty}(Q_{r+\rho}')} \leq c_{2}, 
\end{equation*}
and
\begin{equation*}
\varphi_{k} \leq \varphi_{l}
\text{ in } Q_{r+\rho}'.
\end{equation*}
Then we have that
\begin{equation}\label{GS330}
|D_{x'}\varphi_{l}(x') - D_{x'}\varphi_{k}(x')|
\leq 3\rho^{-1} \left( \rho^{1+\gamma} c_{1} + 2c_{2} \right)^{\frac{1}{\gamma+1}} [\varphi_{l}(x') - \varphi_{k}(x')]^{\frac{\gamma}{\gamma+1}} 
\end{equation}
for any $x' \in Q_{r}'$.
\end{lem}

\begin{proof}
Fix $x '\in Q_{r}'$. Choose $y' \in Q_{r+\rho}'$ with
\begin{equation}\label{GS340}
y' = x' -  \frac{ D_{x'}\varphi_{l}(x') - D_{x'}\varphi_{k}(x') }{ |D_{x'}\varphi_{l}(x') - D_{x'}\varphi_{k}(x')| } 
\left( \frac{ [ \varphi_{l}(x') - \varphi_{k}(x')]}{ \rho^{1+\gamma} c_{1} + 2c_{2} } \right)^{\frac{1}{\gamma+1}} \rho .
\end{equation}
Then by Taylor expansion of $\varphi_{l}-\varphi_{k}$ with respect to $x'$,
\begin{equation*}\begin{aligned}\label{}
& [\varphi_{l}(x') - \varphi_{k}(x')] - [\varphi_{l}(y') - \varphi_{k}(y')] \\
& \quad \geq [D_{x'}\varphi_{l}(x') - D_{x'}\varphi_{k}(x')] \cdot (x'-y')  \\
& \qquad - \left([D_{x'}\varphi_{l}]_{C^{\gamma}(Q_{r+\rho}')} + [D_{x'}\varphi_{k}]_{C^{\gamma}(Q_{r+\rho}')} \right) |x'-y'|^{1+\gamma}.
\end{aligned}\end{equation*}
Since $\varphi_{l}(y') \geq \varphi_{k}(y')$, we find from \eqref{GS340} that
\begin{equation*}\begin{aligned}
&  \varphi_{l}(x') - \varphi_{k}(x') \\
& \quad \geq \rho  |D_{x'}\varphi_{l}(x') - D_{x'}\varphi_{k}(x')|
 \left( \frac{ [ \varphi_{l}(x') - \varphi_{k}(x')]}{ \rho^{1+\gamma} c_{1} + 2c_{2} } \right)^{\frac{1}{\gamma+1}} \\
& \qquad -  \frac{ \rho^{\gamma+1} \left([D_{x'}\varphi_{l}]_{C^{\gamma}(Q_{r+\rho}')} + [D_{x'}\varphi_{k}]_{C^{\gamma}(Q_{r+\rho}')} \right) [\varphi_{l}(x') - \varphi_{k}(x')] }{ \rho^{\gamma+1} c_{1} + 2c_{2} }.
\end{aligned}\end{equation*}
So with that $[ D_{x'}\varphi_{k} ]_{C^{\gamma}(Q_{r+\rho}')}, [ D_{x'}\varphi_{l} ]_{C^{\gamma}(Q_{r+\rho}')} \leq c_{1}$, 
 we absorb the last term in the right-hand side to the left-hand side to find that
\begin{equation*}\label{}
3 [\varphi_{l}(x') - \varphi_{k}(x')]
\geq \rho |D_{x'}\varphi_{0}(x') - D_{x'}\varphi_{1}(x')|
 \left( \frac{ [ \varphi_{l}(x') - \varphi_{k}(x')]}{ \rho^{1+\gamma} c_{1} + 2c_{2} } \right)^{\frac{1}{\gamma+1}},
\end{equation*}
and so the lemma holds.
\end{proof}

\begin{rem}\label{kappa_remark}
With the assumption in the main theorems, the estimate \eqref{GS330} in Lemma \ref{GSS300} has scaling invariance. By taking $r = 2R$ and $\rho=R$, we use the condition \eqref{minimum} to find that $c_{2} = \sup_{k \in K_{+}} \| \varphi_{k} \|_{L^{\infty}(Q_{3R}')} \leq 4R + 2nR  \sup_{k \in K_{+}} \| D_{x'}\varphi_{k} \|_{L^{\infty}(Q_{3R}')}$. Thus with $c_{1}= \sup_{k \in K_{+}} [D_{x'}\varphi_{k}]_{C^{\gamma}(Q_{3R}')}$,
\begin{equation*}\begin{aligned}\label{}
& \left| D_{x'}\varphi_{l}(x') - D_{x'}\varphi_{k}(x') \right| \\
& \ \leq 6 \left[ R^{\gamma}  \sup_{k \in K_{+}} [D_{x'}\varphi_{k}]_{C^{\gamma}(Q_{3R}')} + 1 + n  \sup_{k \in K_{+}} \| D_{x'}\varphi_{k} \|_{L^{\infty}(Q_{3R}')} \right]^{\frac{1}{\gamma+1}} \left[ \frac{ \varphi_{l}(x') - \varphi_{k}(x') }{R} \right]^{\frac{\gamma}{\gamma+1}} 
\end{aligned}\end{equation*}
for any $x' \in Q_{2R}'$ and $k,l \in K_{+}$.
\end{rem}

In the main theorem, we obtain the estimate with respect to the cube $Q_{3R}$. But for proving the main theorem, we localize the problem and derive the estimates with respect to the cube $Q_{r}(z) \subset Q_{2R}$. So from now on, we will assume that 
\begin{equation*}\label{}
|D_{x'}\varphi_{l}(x') - D_{x'}\varphi_{k}(x')|
\leq \kappa R^{-2\mu} |\varphi_{l}(x') - \varphi_{k}(x')|^{2\mu} 
\qquad  \left( x' \in Q_{r}'(z), \ l,k \in K_{+}  \right),
\end{equation*}
for some constant $\kappa>0$, where this constant will be chosen as
\begin{equation*}\begin{aligned}
\kappa &=  18n \left( 1 + R^{\gamma}  \sup_{k \in K_{+}} [D_{x'}\varphi_{k}]_{C^{\gamma}(Q_{3R}')} +   \sup_{k \in K_{+}} \| D_{x'}\varphi_{k} \|_{L^{\infty}(Q_{3R}')} \right) \\ &
\qquad  + R^{2\mu} \sup_{k \in K} \left[ A_{ij}^{\alpha \beta} \right]^2_{C^{\mu} \left( Q_{2R}^{k} \right)},
\end{aligned}\end{equation*}
in the proof of the main theorem. Here, we also remark that 
\begin{equation*}
R^{2\mu} \left[ D_{x'}\varphi_{k} \right]_{C^{2\mu}(Q_{3R}')} \leq  (3nR)^{\gamma} [D_{x'}\varphi_{k}]_{C^{\gamma}(Q_{3R}')}
\leq \kappa
\qquad \left( k \in K_{+} \right).
\end{equation*}

\subsection{Estimate on \texorpdfstring{$(\pi_{2}, \cdots, \pi_{n})$}{pi}} 

With Lemma \ref{GSS300}, we obtain H\"{o}lder estimates related to $\pi_{i}$ in Lemma \ref{GSS600}. The results in this section is obtained with respect to the origin, but the origin can be changed to an arbitrary point in $\br^{n}$ by using translation.

\sskip

For the composite cube $\left( Q_{2r}, \left\{ \varphi_{k} : k \in K_{+} \right\} \right)$, let $\pi$ be  the derivative of the naturally induced flow $\pi : Q_{2r} \to \br^{n}$ in Definition \ref{derivative of flow}. Then
\begin{equation}\label{GS410}
\pi = (-1,\pi') = ( -1, \pi_{2}, \cdots, \pi_{n} )
\end{equation}
where
\begin{equation}\begin{aligned}\label{GS420}
\pi_{\alpha}(x) = D_{\alpha}\varphi_{k+1}( x') \cdot T_{k}(x)
+ D_{\alpha}\varphi_{k}( x')  \cdot [1-T_{k}(x)]
\quad \text{ in } \quad Q_{2r}^{k},
\end{aligned}\end{equation}
for any $k \in K$ and  $\alpha \in \{ 2, \cdots, n \}$. In view of Lemma \ref{GSS300}, we assume that
\begin{equation}\label{GS400}
|D_{x'}\varphi_{l}(x') - D_{x'}\varphi_{k}(x')|
\leq \kappa \left( \frac{ |\varphi_{l}(x') - \varphi_{k}(x')| }{R} \right)^{2\mu} 
\quad  \left( x' \in Q_{2r}, ~ k,l \in K_{+} \right),
\end{equation}
and
\begin{equation}\label{GS405}
|D_{x'}\varphi_{k}(x') - D_{x'}\varphi_{k}(y')|
\leq \kappa  \left( \frac{ |x'-y'| }{R} \right)^{2\mu}
\qqquad  \left( x',y' \in Q_{2r}', ~ k \in K_{+} \right),
\end{equation}
for some constant $\kappa>0$ and $R>0$. Also recall from \eqref{mu} that $2\mu = \frac{\gamma}{\gamma+1}$.

\sskip

For this subsection, we employ the letter $c \geq 1$ to denote any constants that can be explicitly computed in terms such as $n$,  $\kappa$, $\sup_{k \in K_{+}} \| D_{x'} \varphi_{k} \|_{L^{\infty}(Q_{2r}')}$  and the number of elements in the set $K$.

\sskip

We first handle the case when the two points belong to a same region.

\begin{lem}\label{GSS400}
Under the assumption \eqref{GS400} and \eqref{GS405}, we have that
\begin{equation*}
\big| \pi'(y) - \pi'(z) \big| \leq c \kappa \left( \frac{ |y-z| }{R} \right)^{ 2\mu }
\qquad \left( y,z \in \overline{ Q_{2r}^{k} } \cap Q_{2r} \right),
\end{equation*}
for any $k \in K$.
\end{lem}

\begin{proof}
Let $y = (y^{1},y') \in \overline{ Q_{2r}^{k} } \cap Q_{2r} $ and $z = (z^{1},z') \in \overline{ Q_{2r}^{k} } \cap Q_{2r} $. Then 
\begin{equation}\label{GS430} 
\varphi_{k}(y') \leq y^{1} \leq \varphi_{k+1}(y')
\qquad \text{and} \qquad
\varphi_{k}(z') \leq z^{1} \leq \varphi_{k+1}(z').
\end{equation}
To prove the lemma, we take
\begin{equation}\label{GS440}
w = \frac{y^{1} - \varphi_{k}(y' )}{\varphi_{k+1}(y' ) - \varphi_{k}(y' )}[\varphi_{k+1}(z' ) - \varphi_{k}(z' )]  + \varphi_{k}(z' )
\in \mathbb{R}.
\end{equation}
Then we claim that
\begin{equation}\label{GS450}
w \in [\varphi_{k}(z'), \varphi_{k+1}(z')]
\qquad \text{and} \qquad
|(w,z' ) - y| + |(w,z' ) - z| \leq c |y - z|^{\gamma}.
\end{equation}
By a direct calculation using \eqref{GS440}, we have from \eqref{GS430} that
$$ w - \varphi_{k}(z' ) = \frac{y^{1} - \varphi_{k}(y' )}{\varphi_{k+1}(y' ) - \varphi_{k}(y' )}  [\varphi_{k+1}(z' ) - \varphi_{k}(z' )] \geq 0$$
and
$$ \varphi_{k+1}(z' ) - w 
= \frac{ \varphi_{k+1}(y' ) - y^{1}}{\varphi_{k+1}(y' ) - \varphi_{k}(y' )}  [\varphi_{k+1}(z' ) - \varphi_{k}(z' )]  \geq 0$$
which implies that
\begin{equation}\label{GS480}
w \in [\varphi_{k}(z' ), \varphi_{k+1}(z' )].
\end{equation}
To prove the second inequality of $\eqref{GS450}$, recall from \eqref{GS440} and \eqref{GS430} that
\begin{equation*}
w = \frac{y^{1} - \varphi_{k}(y' )}{\varphi_{k+1}(y' ) - \varphi_{k}(y' )} \left( [\varphi_{k+1}(z' ) - \varphi_{k+1}(y' ) ] - [ \varphi_{k}(z' )- \varphi_{k}(y' )] \right)
 + \varphi_{k}(z' ) - \varphi_{k}(y' ) + y^{1},
\end{equation*}
and 
\begin{equation*}
\left| \frac{y^{1} - \varphi_{k}(y' )}{\varphi_{k+1}(y' ) - \varphi_{k}(y' )} \right| \leq 1.
\end{equation*}
So we obtain that
\begin{equation}\label{GS510}
|w-z^{1}| \leq |\varphi_{k+1}(z' ) - \varphi_{k+1}(y' )| + 2|\varphi_{k}(z' )- \varphi_{k}(y' )| + |y^{1} - z^{1}|
\leq c |y - z|,
\end{equation}
and
\begin{equation}\label{GS520}
|w-y^{1}| \leq |\varphi_{k+1}(z' ) - \varphi_{k+1}(y' )| + 2|\varphi_{k}(z' )- \varphi_{k}(y' )| 
\leq c |y - z|.
\end{equation}
So the claim \eqref{GS450} holds from \eqref{GS480}, \eqref{GS510} and \eqref{GS520}.

\sskip

To estimate $\pi'(y) - \pi'(z)$, with \eqref{GS440}, we make the following computations that
\begin{equation*}\begin{aligned}\label{}
\frac{D_{x'}\varphi_{k+1}(z') [ z^{1} - \varphi_{k}(z')]}{ \varphi_{k+1}(z') - \varphi_{k}(z')} 
& = \frac{D_{x'}\varphi_{k+1}(z') [ z^{1} - w + w - \varphi_{k}(z')]}{ \varphi_{k+1}(z') - \varphi_{k}(z')} \\
& = \frac{D_{x'}\varphi_{k+1}(z')[z^{1} - w]}{\varphi_{k+1}(z') - \varphi_{k}(z')}
+ \frac{D_{x'}\varphi_{k+1}(z') [ y^{1} - \varphi_{k}(y')]}{ \varphi_{k+1}(y') - \varphi_{k}(y')},
\end{aligned}\end{equation*}
and
\begin{equation*}\begin{aligned}\label{}
\frac{D_{x'}\varphi_{k}(z')[\varphi_{k+1}(z') - z^{1}]}{\varphi_{k+1} (z') - \varphi_{k}(z')} 
& = \frac{D_{x'}\varphi_{k}(z')[\varphi_{k+1}(z') -  w + w - z^{1}]}{\varphi_{k+1} (z')  - \varphi_{k}(z')} \\
& = \frac{D_{x'}\varphi_{k}(z')[ \varphi_{k+1}(y') - y^{1}]}{\varphi_{k+1}(y') - \varphi_{k}(y')} 
+ \frac{D_{x'}\varphi_{k}(z') [ w- z^{1}]}{\varphi_{k+1}(z') - \varphi_{k}(z')}.
\end{aligned}\end{equation*}
By the definition of $\pi'(y)$ and $\pi'(z)$ in \eqref{GS260},
\begin{equation*}\begin{aligned}
\pi'(z) - \pi'(y) 
& = \left[ \frac{D_{x'}\varphi_{k+1}(z')[z^{1} - \varphi_{k}(z')]}{\varphi_{k+1}(z') - \varphi_{k}(z')} + \frac{D_{x'}\varphi_{k}(z')[\varphi_{k+1}(z') - z^{1}]}{\varphi_{k+1}(z') - \varphi_{k}(z')}  \right] \\
& \quad - \left[ \frac{D_{x'}\varphi_{k+1}(y')[y^{1} - \varphi_{k}(y')]}{\varphi_{k+1}(y') - \varphi_{k}(y')} + \frac{D_{x'}\varphi_{k}(y')[\varphi_{k+1}(y') - y^{1}]}{\varphi_{k+1}(y') - \varphi_{k}(y')}  \right].
\end{aligned}\end{equation*}
So it follows that
\begin{equation*}\begin{aligned}
\pi'(z)-\pi'(y) 
& =  \frac{[D_{x'}\varphi_{k+1}(z') - D_{x'}\varphi_{k}(z')][z^{1} - w]}{\varphi_{k+1}(z') - \varphi_{k}(z')} \\
&\quad +\frac{[D_{x'}\varphi_{k}(z') - D_{x'}\varphi_{k}(y')][\varphi_{k+1}(y') - y^{1}]}{\varphi_{k+1}(y') - \varphi_{k}(y')}  \\
& \quad +  \frac{[D_{x'}\varphi_{k+1}(z') - D_{x'}\varphi_{k+1}(y')][y^{1} - \varphi_{k}(y')]}{\varphi_{k+1}(y') - \varphi_{k}(y')} .
\end{aligned}\end{equation*}
It only remains to estimate the right-hand side of the above equality. By \eqref{GS400},
\begin{equation}\label{GS570} 
|D_{x'}\varphi_{k+1}(z') - D_{x'}\varphi_{k}(z')|
\leq \kappa \left( \frac{ |\varphi_{k+1}(z') - \varphi_{k}(z')| }{R} \right)^{2\mu}.
\end{equation}
From \eqref{GS430} and \eqref{GS450}, we have that 
\begin{equation}\label{GS580}
w, z^{1} \in [\varphi_{k}(z' ), \varphi_{k+1}(z' )],
\qquad \text{and} \qquad 
y^{1} \in [\varphi_{k}(y'),\varphi_{k+1}(y')].
\end{equation}
So  by \eqref{GS405}, \eqref{GS570} and \eqref{GS580}, we obtain that
\begin{equation*}\label{}
\big| \pi'(z)-\pi'(y)  \big|
\leq c\kappa \left[ \left( \frac{ |z^{1}-w| }{R} \right)^{2\mu} +  \left( \frac{ |y' - z'| }{R} \right)^{2\mu}  \right],
\end{equation*}
and the lemma follows from \eqref{GS510}.
\end{proof}

With Lemma \ref{GSS400}, we now obtain H\"{o}lder continuity of $\pi'$ in $Q_{2r}$.

\begin{lem}\label{GSS600}
Under the assumption \eqref{GS400} and \eqref{GS405}, we have that
\begin{equation*}
\big| \pi'(y) - \pi'(z) \big| \leq c \kappa \left( \frac{ |y-z| }{ R } \right)^{2\mu}
\qquad (y,z \in Q_{2r}).
\end{equation*}
\end{lem}

\begin{proof}
If $y,z \in Q_{2r}^{k}$ $(k \in K)$ then the lemma holds from Lemma \ref{GSS400}. So suppose that $y \in Q_{2r}^{k}$ and $z \in Q_{2r}^{l}$ with $k < l$ $(k,l \in K)$. Let $\alpha : [0,1] \to Q_{2r}$ be a line connecting $y$ and $z$. Then on the line $\alpha$, one can choose the points
\begin{flalign}
\nonumber w_{k+1}= \left( \varphi_{k+1}(w_{k+1}'), w_{k+1}' \right) & \in \overline{ Q_{2r}^{k}} \cap \overline{ Q_{2r}^{k+1}}, \\
\label{GS630} w_{k+2}=\left( \varphi_{k+2}(w_{k+2}'), w_{k+2}' \right) & \in \overline{ Q_{2r}^{k+1}} \cap \overline{ Q_{2r}^{k+2}}, \\
\nonumber \vdots & \\
\nonumber w_{l}=\left( \varphi_{l}(w_{l}'), w_{l}' \right) & \in \overline{ Q_{2r}^{l-1}} \cap \overline{ Q_{2r}^{l}}.
\end{flalign}
Let $y=w_{k}$ and $z=w_{l+1}$. Since $y \in Q_{2r}^{k}$ and $z \in Q_{2r}^{l}$, we find from \eqref{GS630} and  Lemma \ref{GSS400} that
\begin{equation*} 
\left| \pi'(w_{m}) - \pi' \left( w_{m+1} \right) \right|
\leq c \kappa \left( \frac{ \left| w_{m} -  w_{m+1}  \right| }{ R } \right)^{2\mu} \qquad \left( m=k,\cdots,l \right).
\end{equation*}
Since the points  $y=w_{k}$, $w_{k+1}$, $w_{k+1}$, $\cdots$, $w_{l}$ and $w_{l+1}=z$ are placed on the line connecting $y$ and $z$, we get that
\begin{equation*}
\left|  \pi' \left( w_{m} \right) - \pi'(w_{m+1})  \right|
\leq c \kappa \left( \frac{ |y-z| }{ R } \right)^{2\mu} \qquad 
\left( m=k,\cdots,l \right).
\end{equation*}
So the lemma follows by the triangle inequality.
\end{proof}

\section{Gradient estimates for reference equations}\label{GELE}

In this section, we obtain gradient estimates for when the coefficients are measurable in one variable. The results in this section is based on \cite{DH1} and \cite[Lemma 3.5]{DHKD1}, but we write this section for the convenience of the readers. For the extension to nonlinear problems, see \cite{BSKY1}. For this section, we employ the letter $c \geq 1$ to denote any constants that can be explicitly computed in terms such as $n$, $N$, $\lambda$ and $\Lambda$.

\smallskip 

Assume that
\begin{equation}\label{GELE1100} 
\lambda |\xi|^{2} 
\leq A_{ij}^{\alpha \beta }(x^{1}) \xi_{\alpha}^{i} \xi_{\beta}^{j} 
\qquad \text{and} \qquad
\left| A_{ij}^{\alpha \beta}(x^{1}) \right| \leq \Lambda 
\qquad
\left( x \in Q_{2}, \ \xi \in \br^{nN} \right),
\end{equation}
for some positive constant $\lambda$ and $\Lambda$.
Let $h \in W^{1,2} \left( Q_{2} ,\br^{N} \right)$ be a weak solution of
\begin{equation}\label{GELE1300}
D_{\alpha} \left[ A_{ij}^{\alpha \beta} (x^{1}) D_{\beta}h^{j} \right] = 0 
\quad \text{ in } \quad Q_{2}.
\end{equation}

We first have the following energy estimate in the following lemma.

\begin{lem}\label{GELES2000}
Under the assumption \eqref{GELE1100}, let $h$ be a weak solution of \eqref{GELE1300}. Then
\begin{equation*}\label{} 
\int_{Q_{\rho}}  | Dh |^{2}  \, dx
\leq \frac{ c }{ \rho^{2} } \int_{Q_{2\rho}}  
\big|  h - (h)_{Q_{2\rho}} \big|^{2}  \, dx,
\end{equation*}
for any $\rho \in (0,1]$. 
\end{lem}

\begin{proof}
Fix $\rho \in (0,1]$. In view of \eqref{GELE1300}, 
\begin{equation}\label{GELE2300}
D_{\alpha}\left[ A_{ij}^{\alpha \beta} (x^{1}) D_{\beta}\left( h^{j} - (h^{j})_{Q_{2\rho}} \right) \right] = 0 ~ \text{ in } ~ Q_{2\rho}.
\end{equation}
Choose a cut-off function $\phi \in C_{c}^{\infty}(Q_{2\rho})$ with
\begin{equation}\label{GELE2400} 
0 \leq \phi \leq 1,
\qquad 
|D\phi| \leq  c\rho^{-1}
\qquad \text{and} \qquad
\phi = 1 ~ \text{ in } ~ Q_{\rho}.
\end{equation}
We test \eqref{GELE2300} by $\left[ h^{i} - (h^{i})_{Q_{2\rho}} \right] \phi^{2}$ to find that
\begin{equation*}\begin{aligned}\label{} 
& \int_{Q_{2\rho}} \left \langle A_{ij}^{\alpha \beta} (x^{1}) D_{\beta} \left[ h^{j} - (h^{j})_{Q_{2\rho}} \right] , D_{\alpha}\left[ h^{i} - (h^{i})_{Q_{2\rho}} \right] \right \rangle \phi^{2} \, dx \\
& \quad = - \int_{Q_{2\rho}} \left \langle A_{ij}^{\alpha \beta} (x^{1}) D_{\beta} \left[ h^{j} - (h^{j})_{Q_{2\rho}} \right] , \left[ h^{i} - (h^{i})_{Q_{2\rho}} \right] 2\phi D_{\alpha} \phi \right \rangle  \, dx .
\end{aligned}\end{equation*}
Then by \eqref{GELE1100} and Young's inequality, we have
\begin{equation*}\label{} 
\int_{Q_{2\rho}}  
\big| Dh \big|^{2} \phi^{2} \, dx
\leq c \int_{Q_{2\rho}}  
\big|  h - (h)_{Q_{2\rho}} \big|^{2}  |D\phi|^{2} \, dx,
\end{equation*}
and the lemma follows from \eqref{GELE2400}.
\end{proof}

We have the following higher order estimate in the following lemma.

\begin{lem}\label{GELES3000}
Under the assumption \eqref{GELE1100}, let $h$ be a weak solution of \eqref{GELE1300}. Then for any integer $ p \geq 0 $, we have that
\begin{equation*}\label{} 
\int_{Q_{1}}  
\big| D_{x'}^{p}Dh \big|^{2}\, dx
\leq c(p)  \int_{Q_{2}}  
\big|  h  - (h)_{Q_{2}} \big|^{2} \, dx.
\end{equation*}
\end{lem}

\begin{proof}
If $p=0$ then the lemma holds from Lemma \ref{GELES2000}. So we assume that $p \geq 1$. 

Let $q \in \{ 0, \cdots, p \}$ and $\rho_{q} = 1 +  \frac{1}{q+1}$. Fix $(0,\xi') \in \mathbb{N}^{n}$ with $|(0,\xi')| = q$. Then we have that
\begin{equation}\label{GELE3300}
D_{\alpha} \Big( A_{ij}^{\alpha \beta} (x^{1}) D_{\beta} \left[ D_{\xi'} \left( h^{j} - (h^{j})_{Q_{2\rho}} \right) \right] \Big) = 0 \text{ in } Q_{2}.
\end{equation}
Choose a cut-off function $\phi \in C_{c}^{\infty}(Q_{\rho_{q}})$ with
\begin{equation}\label{GELE3400} 
0 \leq \phi \leq 1,
\qquad 
|D\phi| \leq  c(q)
\qquad \text{and} \qquad
\phi = 1 \text{ in } Q_{\rho_{q+1}}.
\end{equation}
We test \eqref{GELE3300} by $D_{\xi'}\left( h^{i} - (h^{i})_{Q_{2}} \right) \phi^{2}$ to find that
\begin{equation*}\begin{aligned}\label{} 
& \int_{Q_{\rho_{q}}} \big \langle A_{ij}^{\alpha \beta} (x^{1}) D_{\beta} \left[ D_{\xi'} \left( h^{j} - (h^{j})_{Q_{2}} \right)  \right] , D_{\alpha} \left[ D_{\xi'} \left( h^{i} - (h^{i})_{Q_{2}} \right) \right] \big \rangle \phi^{2} \, dx \\
& \quad = - \int_{Q_{\rho_{q}}} \big \langle A_{ij}^{\alpha \beta} (x^{1}) D_{\beta} \left[ D_{\xi'} \left( h^{i} - (h^{i})_{Q_{2}} \right) \right] , \left[ D_{\xi'}\left( h^{i} - (h^{i})_{Q_{2}} \right) \right] 2\phi D_{\alpha} \phi \big \rangle  \, dx .
\end{aligned}\end{equation*}
Then by \eqref{GELE1300} and Young's inequality,
\begin{equation*}\label{} 
\int_{Q_{\rho_{q}}}  
\big| DD_{\xi'}\left( h - (h)_{Q_{2\theta}} \right)  \big|^{2} \phi^{2} \, dx
\leq c \int_{Q_{\rho_{q}}}  
\big|  D_{\xi'} \left( h - (h)_{Q_{2\theta}} \right)  \big|^{2}  |D\phi|^{2} \, dx,
\end{equation*}
whence we have from \eqref{GELE3400} that 
\begin{equation*}\label{} 
\int_{Q_{\rho_{q+1}}}  
\big| DD_{\xi'}\left( h - (h)_{Q_{2}} \right) \big|^{2}\, dx
\leq c(q)\theta^{-2} \int_{Q_{\rho_{q}}}  
\big|  D_{\xi'}\left( h - (h)_{Q_{2}} \right) \big|^{2} \, dx.
\end{equation*}
Since $\xi'\in \mathbb{N}^{n-1}$ was arbitrary chosen, we find that
\begin{equation*}\label{} 
\int_{Q_{\rho_{q+1}}}  
\big| DD_{x'}^{q}\left( h - (h)_{Q_{2}} \right)  \big|^{2}\, dx
\leq c(q) \int_{Q_{\rho_{q}}}  
\big|  D_{x'}^{q} \left( h - (h)_{Q_{2}} \right) \big|^{2} \, dx.
\end{equation*}
Also since $q \in \{ 0, \cdots, p \}$ was arbitrary chosen, if we apply induction, we get that
\begin{equation*}\label{} 
\int_{Q_{1}}  
\big| DD_{x'}^{p}h \big|^{2}\, dx
\leq \int_{Q_{\rho_{p+1}}}  
\big| DD_{x'}^{p}h \big|^{2}\, dx
\leq c(p)   \int_{Q_{2}}  
\big|  h - (h)_{Q_{2}} \big|^{2} \, dx.
\end{equation*}
and this complete the proof.
\end{proof}

\begin{lem}\label{GELES5000}
Under the assumption \eqref{GELE1100}, let $h$ be a weak solution of \eqref{GELE1300}. Then
\begin{equation*}\label{}
|h(x)-h(y)|^{2}
\leq c |x-y|
\int_{Q_{2}} \big| h - (h)_{Q_{2}} \big|^{2} \, dx,
\end{equation*}
for any $x,y \in Q_{1}$.
\end{lem}

\begin{proof}
By using an approximation argument, we may assume that $h \in C^{1}(Q_{2})$. To use the Sobolev type embedding, we take an integer $p > \frac{n}{2}$.

Let $(x^{1},x'), (y^{1},y') \in Q_{1}$. We use the Sobolev embedding theorem in $x^1$-variable to find 
\begin{equation*}
\left| h(x^{1},x') - h(y^{1},x') \right|^{2} 
\leq c \left| x^{1}-y^{1} \right|  \int_{(-1,1)} |D_{1}h(z^{1},x')|^{2} \, dz^{1}.
\end{equation*}
Also for any fixed $z^{1} \in (-1,1)$, applying the Sobolev embedding theorem in $x'$-variable, we have
\begin{equation*} 
\left| D_{1}h(z^{1},x') \right|^{2}
\leq c \sum_{ 0 \leq q \leq p } \int_{Q_{1}'} \left| D_{x'}^{q}D_{1}h(z^{1},z') \right|^{2} \, dz'.
\end{equation*}
So it follows that
\begin{equation}\label{GELE5500} 
\left| h(x^{1},x') - h(y^{1},x') \right|^{2}
\leq c |x-y|
\sum_{ 0 \leq q \leq p } \int_{Q_{1}} \left| D_{x'}^{q}D_{1}h(z^{1},z') \right|^{2} \, dz^{1}dz'.
\end{equation}

On the other hand, by applying the Sobolev embedding theorem first in $x'$-variable and then in $x^1$-variable, we obtain that
\begin{equation*}\label{} 
\left| h(x^{1},x') - h(x^{1},y') \right| ^{2}
\leq c |x'-y'|
\sum_{ 1 \leq q \leq p } \int_{Q_{1}'} \left| D_{x'}^{q}h(x^{1},z') \right|^{2} \, dz', \end{equation*}
and 
\begin{equation*}\label{}  
\left| D_{x'}^{q}h(x^{1},z') \right|^{2}
\leq c \int_{(-1,1)}  \left| D_{x'}^{q}h(z^{1},z') \right|^{2}  + \left| D_{x'}^{q}D_{1}h(z^{1},z') \right|^{2} \, dz^{1}
\qquad (0 \leq q \leq p),
\end{equation*}
for any fixed $z^{1} \in (-1,1)$. So we have that
\begin{equation}\label{GELE5800} 
\left| h(x^{1},x') - h(x^{1},y') \right|^{2} 
\leq c |x-y|
\sum_{ 0 \leq q \leq p } \int_{Q_{1}} \left| D_{x'}^{q}Dh(z^{1},z') \right|^{2} \, dz^{1}dz'. 
\end{equation}

By combining \eqref{GELE5500} and \eqref{GELE5800}, we discover that
\begin{equation*}\label{} 
\left| h(x^{1},x') - h(y^{1},y') \right|^{2}
\leq c |x-y| \sum_{ 0 \leq q \leq p } \int_{Q_{1}} \left| D_{x'}^{q}Dh(z) \right|^{2} \, dz.
\end{equation*}
Now using Lemma \ref{GELES3000}, we finish the proof of the lemma.
\end{proof}

Next, we handle the inhomogeneous case. Under the assumptions \eqref{GELE1100}, let $w \in W^{1,2} \left( Q_{2}, \br^{N} \right)$ be a weak solution of
\begin{equation}\label{GELE6100}
D_{\alpha}\left[ A_{ij}^{\alpha \beta} (x^{1}) D_{\beta}w^{j} \right] = D_{1} F_{1}^{i}(x^{1}) 
\quad \text{ in } \quad Q_{2}.
\end{equation}
Set $W : Q_{2} \to \br^{Nn}$ where $W = \left( W^{1}, \cdots, W^{N} \right)^{T}$ and
\begin{equation}\label{GELE7100} 
W^{i} = \left( \left[ -  \sum_{1 \leq j \leq N} \sum_{1 \leq \beta \leq n}  A_{ij}^{1 \beta}(x^{1}) D_{\beta}w^{j} \right] + F_{1}^{i}(x^{1}), D_{x'}w^{i} \right)
\end{equation}
for any $1 \leq i \leq N$. Then $A_{ij}^{1 \beta} (x^{1}) D_{\beta}w^{j}  - F_{1}^{i}(x^{1})$ is weakly differentiable for any $i \in \{ 1, \cdots, N \}$ as in the following lemma.

\begin{lem}\label{GELES6000}
With \eqref{GELE1100}, let $w$ be a weak solution of \eqref{GELE6100}. Then
\begin{equation*}\label{} 
\left| DW_{1} \right|
\leq c \left| DD_{x'}w(x) \right|
\qquad \left( a.e.\ x \in Q_{2} \right).
\end{equation*}
\end{lem}

\begin{proof}
Since $F^{i}_1$ is a function of $x^1$-variable, $D_{\beta}w^{j}$ are weakly differentiable in $Q_2$ for the $x'$-variables. So by the definition of weak solution and \eqref{GELE6100}, $A_{ij}^{1 \beta}(x^{1}) D_{\beta}w^{j}- F_{1}^{i}(x^{1}) $ are weakly differentiable in the $x^1$-variable. Thus
\begin{equation*}\begin{aligned}\label{} 
D_{1}W_{1}^{i} = D_{1}\left[ A_{ij}^{1 \beta}(x^{1}) D_{\beta}w^{j}(x) - F_{1}^{i}(x^{1}) \right] 
& = - \sum_{ 2 \leq \alpha \leq n} D_{\alpha}\left[ A_{ij}^{1 \beta}(x^{1}) D_{\beta}w^{j}(x) \right]  \\
& = - \sum_{ 2 \leq \alpha \leq n} A_{ij}^{1 \beta}(x^{1}) D_{\alpha}D_{\beta}w^{j}(x)
\end{aligned}\end{equation*}
for a.e. $x \in Q_2$ and $1 \leq i \leq N$. So we find that
\begin{equation*}\begin{aligned}\label{} 
\left| D_{1}W_{1}^{i} \right|
& = \left| D\left[ A_{ij}^{1 \beta}(x^{1}) D_{\beta}w^{j}(x) - F_{1}^{i}(x^{1}) \right] \right| \\
& \leq \left| D_{1}\left[ A_{ij}^{1 \beta}(x^{1}) D_{\beta}w^{j}(x) - F_{1}^{i}(x^{1}) \right]  \right|
+ \left| D_{x'}\left[ A_{ij}^{1 \beta}(x^{1}) D_{\beta}w^{j}(x) - F_{1}^{i}(x^{1}) \right]  \right| \\
& \leq c |DD_{x'}w(x)|,
\end{aligned}\end{equation*}
for a.e. $x \in Q_2$ and $1 \leq i \leq N$.
\end{proof}

With Lemma \ref{GELES6000}, we obtain the following excess decay estimate.

\begin{lem}\label{GELES7000}
Under the assumption \eqref{GELE1100}, let $w$ be a weak solution of \eqref{GELE6100}. Then for $W$ in \eqref{GELE7100}, we have that
\begin{equation*}\label{}
\mint_{Q_{\rho}} \big| W - (W)_{Q_{\rho}} \big|^{2} \, dx
\leq c  \rho  
\mint_{Q_{2}} \big| W - (W)_{Q_{2}} \big|^{2} \, dx
\qquad \left( \rho \in (0,2] \right).
\end{equation*}
\end{lem}

\begin{proof}
By differentiating \eqref{GELE6100} with respect to $x^{m}$-variable $(m \in \{ 2, \cdots, n \})$,
\begin{equation}\label{GELE7400}
D_{ \alpha }\left[ A_{ij}^{\alpha \beta} (x^{1}) D_{\beta} \left( D_{m}w^{j} \right) \right] = 0 \quad \text{ in } \quad Q_{2}.
\end{equation}
So by applying Lemma \ref{GELES5000} to \eqref{GELE7400} and $D_{m}w$ instead of \eqref{GELE1300} and $h$ respectively,
\begin{equation}\label{GELE7500}
\mint_{Q_{\rho}} \big| D_{m}w - (D_{m}w)_{Q_{\rho}} \big|^{2} \, dx 
\leq c \rho \int_{Q_{2}} \big| D_{m}w - (D_{m}w)_{Q_{2}} \big|^{2} \, dx,
\end{equation}
for any $\rho \in (0,1]$. Also by applying Lemma \ref{GELES2000} to \eqref{GELE7400} and $D_{m}w$ instead of \eqref{GELE1300} and $h$ respectively,
\begin{equation}\label{GELE7600} 
\int_{Q_{\rho}}  | DD_{m}w |^{2}  \, dx
\leq \frac{ c }{ \rho^{2} } \int_{Q_{2\rho}}  
\big|  D_{m}w - (D_{m}w)_{Q_{2\rho}} \big|^{2}  \, dx,
\end{equation}
for any $\rho \in (0,1]$. Since $m \in \{ 2, \cdots, n \}$ was arbitrary chosen, we find from \eqref{GELE7500} and \eqref{GELE7600} that
\begin{equation}\label{GELE7700}
\mint_{Q_{\rho}} \big| D_{x'}w - (D_{x'}w)_{Q_{\rho}} \big|^{2} \, dx 
\leq c \rho
 \int_{Q_{2}} \big| D_{x'}w - (D_{x'}w)_{Q_{2}} \big|^{2} \, dx,
\end{equation}
and
\begin{equation}\label{GELE7800} 
\int_{Q_{\rho}}  | DD_{x'}w |^{2}  \, dx
\leq \frac{ c }{ \rho^{2} } \int_{Q_{2\rho}}  
\left|  D_{x'}w - (D_{x'}w)_{Q_{2\rho}} \right|^{2}  \, dx,
\end{equation}
for any $\rho \in (0,1]$.

With \eqref{GELE7100}, by Poincar\'{e}'s inequality and Lemma \ref{GELES6000},
\begin{equation*}\begin{aligned}\label{} 
\mint_{Q_{\rho}} \left| \left( W_{1} \right)  - \left( W_{1} \right)_{Q_{\rho}} \right|^{2} \, dx 
\leq c \rho^{2}  \mint_{Q_{\rho}} \left| D W_{1} \right|^{2}  dx 
\leq  c \rho^{2}  \mint_{Q_{\rho}} \big| DD_{x'}w \big|^{2} dx,
\end{aligned}\end{equation*}
for any $\rho \in (0,1/2]$. By \eqref{GELE7700} and \eqref{GELE7800},
\begin{equation*}
\rho^{2} \mint_{Q_{\rho}} \big| DD_{x'}w \big|^{2}  dx
\leq  \mint_{Q_{2\rho}} \big| D_{x'}w - (D_{x'}w)_{Q_{2\rho}} \big|^{2} \, dx
\leq c\rho \int_{Q_{2}} \big| D_{x'}w - (D_{x'}w)_{Q_{2}} \big|^{2} \, dx,
\end{equation*}
for any $\rho \in (0,1/2]$. Thus
\begin{equation*}\begin{aligned}\label{GELE8100} 
& \mint_{Q_{\rho}} \left| W_{1}  - \left( W_{1} \right)_{Q_{\rho}} \right|^{2} \, dx \\
& \quad \leq c\rho \int_{Q_{2}} \big| D_{x'}w - (D_{x'}w)_{Q_{2}} \big|^{2} \, dx,
\end{aligned}\end{equation*}
for any $\rho \in (0,1/2]$. So with \eqref{GELE7700}, we find from \eqref{GELE7100} and H\"{o}lder's inequality to find that
\begin{equation}\label{GELE8200} 
\mint_{Q_{\rho}} \big| W - (W)_{Q_{\rho}} \big|^{2} \, dx \leq c \rho  \mint_{Q_{2}} \big| W - (W)_{Q_{2}} \big|^{2} \, dx,
\end{equation}
for any $\rho \in (0,1/2]$. If $\rho \in (1/2,2]$ then one can check that
\begin{equation*}\begin{aligned}
\mint_{Q_{\rho}} \big| W - (W)_{Q_{\rho}} \big|^{2} \, dx 
& \leq 2 \mint_{Q_{\rho}} \big| W - (W)_{Q_{2}} \big|^{2} + \big| (W)_{Q_{2}} - (W)_{Q_{\rho}} \big|^{2} \, dx \\
& \leq c \mint_{Q_{2}} \big| W - (W)_{Q_{2}} \big|^{2} \, dx,
\end{aligned}\end{equation*}
which implies that
\begin{equation}\label{GELE8400} 
\mint_{Q_{\rho}} \big| W - (W)_{Q_{\rho}} \big|^{2} \, dx \leq c \rho  \mint_{Q_{2}} \big| W - (W)_{Q_{2}} \big|^{2} \, dx,
\end{equation}
for any $\rho \in (1/2,2]$. So we discover from \eqref{GELE8200} and \eqref{GELE8400} that the lemma holds.
\end{proof}

\section{Comparison of H\"{o}lder norm}

In Section \ref{GELE}, we derived the excess decay estimate with respect to the functional $W$ not the gradient of the weak solution $Dw$. With this estimate, we will obtain the excess decay estimate with respect to the functional $U$ in \eqref{HOD630} (which corresponds to $W$) not $Du$. So to obtain piecewise H\"{o}lder continuity of $Du$, we compare H\"{o}lder norm of $U$ and H\"{o}lder norm of $Du$ in this section.

\sskip

In the later sections we will consider composite cubes. So unlike $W$ in \eqref{GELE7100}, $U$ in \eqref{HOD630} depends on $\pi'$ in \eqref{GS250} which was naturally induced by our geometry. In fact, we will use Lemma \ref{PHNS400} in the later paper \cite{KY1}. So to minimize the condition of the results, we consider only one point for Lemma \ref{HODS600} and two points for Lemma \ref{HODS700}.

\sskip

We compare $U$ with $Du$ and $F$ in the following lemma. Later, we will take $\zeta_{\beta}^{j} = D_{\beta}u^{j}$.

\begin{lem}\label{HODS600}
Let $\pi = (-1, \pi_{2}, \cdots, \pi_{n}) \in \br^{n}$ and $\pi_{1} = -1$. For the constants 
\begin{equation}\label{}
A_{ij}^{\alpha \beta}, F_{\alpha}^{i}, \zeta_{\alpha}^{i} 
\qquad (1 \leq i,j \leq N, ~ 1 \leq \alpha, \beta \leq n),
\end{equation}
satisfying \eqref{ell1} and \eqref{ell2}, we define $U \in \br^{Nn}$ as 
\begin{equation}\label{HOD620}
U= \left( 
\begin{array}{ccc}
U_{1}^{1} & \cdots & U_{n}^{1} \\
\vdots & \ddots & \vdots \\
U_{1}^{N} & \cdots & U_{n}^{N}.
\end{array}\right)
\end{equation}
where
\begin{equation}\label{HOD630} 
U_{1}^{i}
= \sum_{1 \leq \alpha \leq n } \pi_{\alpha} \left[ \left( \sum_{1 \leq j \leq N} \sum_{1 \leq \beta \leq n } A_{ij}^{\alpha \beta} \zeta_{\beta}^{j}  \right)- F_{\alpha}^{i} \right]
\qqquad ( i =1, \cdots, N )
\end{equation}
and
\begin{equation}\label{HOD640} 
U_{\beta}^{i} = \zeta_{\beta}^{i} + \pi_{\beta} \, \zeta_{1}^{i}
\qqquad (i =1, \cdots, N, ~ \beta = 2, \cdots, n ).
\end{equation}
Then we have that
\begin{equation}\label{} 
|\zeta| \leq c \Big[ |U|  + |F| \Big] 
\end{equation}
where $c$ is a constants only depending on $n,N,\lambda$ and $\Lambda$.
\end{lem}

\begin{proof}
Since $\pi_{1} = -1$, it  follows from \eqref{HOD630} and \eqref{HOD640} that 
\begin{equation*}\label{} 
U_{1}^{i}
= \sum_{ 1 \leq  \alpha \leq n }  \pi_{ \alpha } \left[ \left( \sum_{1 \leq j \leq N} \sum_{2 \leq \beta \leq n }  A_{ij}^{\alpha \beta}  U_{\beta}^{j} \right) -  F_{\alpha}^{i} \right] - \sum_{1 \leq \alpha \leq n } \pi_{\alpha} \left[  \sum_{1 \leq j \leq N} \sum_{1 \leq \beta \leq n }  A_{ij}^{\alpha \beta} \pi_{\beta} \zeta_{1}^{j} \right],
\end{equation*}
for any $i  = 1, \cdots, N$, which implies that
\begin{equation}\label{HOD660} 
\sum_{1 \leq j \leq N}  \sum_{1 \leq \alpha,\beta \leq n  } A_{ij}^{\alpha \beta} \pi_{\alpha} \pi_{\beta}  \zeta_{1}^{j} 
= \sum_{ 1 \leq \alpha \leq n } \pi_{ \alpha } \left[ \left( \sum_{1 \leq j \leq N} \sum_{2 \leq \beta \leq n }  A_{ij}^{\alpha \beta}  U_{\beta}^{j} \right) - F_{\alpha}^{i} \right] - U_{1}^{i},
\end{equation}
for any $i  = 1, \cdots, N$.  Thus
\begin{equation*}\begin{aligned}\label{}
& \sum_{1 \leq i,J \leq N}   \sum_{1 \leq \alpha,\beta \leq n  }  A_{ij}^{\alpha \beta} \left(\pi_{\alpha} \zeta_{1}^{i}\right) \big(\pi_{\beta}   \zeta_{1}^{j}\big)  \\
& \quad = \sum_{1 \leq i \leq N} \zeta_{1}^{i} \left( \sum_{ 1 \leq \alpha \leq n } \pi_{ \alpha } \left[ \left( \sum_{1 \leq j \leq N} \sum_{2 \leq \beta \leq n }  A_{ij}^{\alpha \beta}  U_{\beta}^{j} \right) - F_{\alpha}^{i} \right] - U_{1}^{i} \right).
\end{aligned}\end{equation*}
So from \eqref{ell1}, \eqref{ell2} and that $ \pi_{1} = -1$, we obtain that
\begin{equation*}\label{}
\lambda |\pi|^{2} |\zeta_1|^{2} 
\leq |\pi| |\zeta_1|  \left( \sum_{ 1 \leq \alpha \leq n }  \left[ \left( \sum_{1 \leq j \leq N} \sum_{2 \leq \beta \leq n }  \left| A_{ij}^{\alpha \beta}  U_{\beta}^{j} \right| \right) + \left| F_{\alpha}^{i} \right| \right] + |U_{1}| \right)
\end{equation*}
whence
\begin{equation}\label{HOD670}
|\zeta_1| \leq |\pi||\zeta_1| \leq c \Big[ |U|  + |F| \Big].
\end{equation}
Moreover, from \eqref{HOD640} and \eqref{HOD670}, we discover that
\begin{equation*}
|\zeta_\beta|\leq |U|+|\pi||\zeta_1| \leq c \Big[ |U|  + |F| \Big]
\end{equation*} 
for any $\beta = 2,\cdots,n$. This and \eqref{HOD670} complete the proof.
\end{proof}

To obtain H\"{o}lder semi-norm $Du$ with H\"{o}lder semi-norm $U$ and $F$, we consider the two points set as a domain in the following lemma.

\begin{lem}\label{HODS700}
For fixed $x,y \in \br^{n}$ $( x \not  =y)$, let $\pi = (-1, \pi_{2}, \cdots, \pi_{n}) : \{ x, y \} \to \br^{n}$ and $\pi_{1} = -1$. For the functions
\begin{equation*}\label{}
A_{ij}^{\alpha \beta} : \{ x, y \}  \to \br^{Nn} \times \br^{Nn}, 
\qquad
F_{\alpha}^{i} : \{ x, y \}  \to \br^{Nn}
\qquad \text{and} \qquad
\zeta_{\alpha}^{i}  : \{ x, y \}  \to \br^{Nn},
\end{equation*}
satisfying \eqref{ell1} and \eqref{ell2} $(1 \leq \alpha, \beta \leq n, ~ 1 \leq i,j \leq N)$, we define $U : \{ x, y \}  \to \br^{Nn}$ as 
\begin{equation}\label{HOD720}
U= \left( 
\begin{array}{ccc}
U_{1}^{1} & \cdots & U_{n}^{1} \\
\vdots & \ddots & \vdots \\
U_{1}^{N} & \cdots & U_{n}^{N}.
\end{array}\right)
\end{equation}
where
\begin{equation}\label{HOD730} 
U_{1}^{i}
= \sum_{1 \leq \alpha \leq n } \pi_{\alpha} \left[ \left( \sum_{1 \leq j \leq N} \sum_{1 \leq \beta \leq n } A_{ij}^{\alpha \beta} \zeta_{\beta}^{j} \right) - F_{\alpha}^{i} \right]
\qqquad 
( i =1, \cdots, N )
\end{equation}
and
\begin{equation}\label{HOD740} 
U_{\beta}^{i} = \zeta_{\beta}^{i} + \pi_{\beta} \, \zeta_{\beta}^{i}
\qqquad (i =1, \cdots, N, ~ \beta = 2, \cdots, n ).
\end{equation}
Then we have that
\begin{equation*}\begin{aligned}\label{}
|\zeta(x) - \zeta(y)| \
& \leq c \Big[ |U(y) - U(x)| + |F(y) - F(x)| \Big] \\
& \quad +  c  \Big[ |U(x)| + |F(x)| \Big] \left[  \left| \pi(x) - \pi(y)\right| +  \left| \pi(x) - \pi(y)\right|^{2}  \right] \\
& \quad +  c  \Big[ |U(x)| + |F(x)| \Big]  \sum_{1 \leq i,j \leq N}   \sum_{1 \leq \alpha, \beta \leq n  }  \left| A_{ij}^{\alpha \beta}(y)  - A_{ij}^{\alpha \beta}(x) \right|,
\end{aligned}\end{equation*}
where the constant $c$ only depends on $n, \, N, \, \lambda$ and $\Lambda$. 
\end{lem}

\begin{proof}
In this proof, $c$ denotes a constant only depending on $n, \, N, \, \lambda$ and $\Lambda$.

Since $\pi_{1} = -1$, it  follows from \eqref{HOD630} and \eqref{HOD640} that 
\begin{equation*}\label{} 
U_{1}^{i}
= \sum_{ 1 \leq  \alpha \leq n }  \pi_{ \alpha } \left[ \left( \sum_{1 \leq j \leq N} \sum_{2 \leq \beta \leq n }  A_{ij}^{\alpha \beta}  U_{\beta}^{j} \right) -  F_{\alpha}^{i} \right] - \sum_{1 \leq \alpha \leq n } \pi_{\alpha} \left[   \sum_{1 \leq j \leq N} \sum_{1 \leq \beta \leq n }  A_{ij}^{\alpha \beta}  \pi_{\beta} \zeta_{1}^{j} \right],
\end{equation*}
for any $i  = 1, \cdots, N$, which implies that
\begin{equation}\begin{aligned}\label{}
\label{HOD780} 
& \sum_{1 \leq j \leq N}  \sum_{1 \leq \alpha,\beta \leq n  } A_{ij}^{\alpha \beta} \pi_{\alpha} \pi_{\beta}  \zeta_{1}^{j} \\
& \qquad = \sum_{ 1 \leq \alpha \leq n } \pi_{ \alpha } \left[ \left( \sum_{1 \leq j \leq N} \sum_{2 \leq \beta \leq n }  A_{ij}^{\alpha \beta}  U_{\beta}^{j} \right) - F_{\alpha}^{i} \right] - U_{1}^{i},
\end{aligned}\end{equation}
for any $i  = 1, \cdots, N$. One can directly check that
\begin{equation}\begin{aligned}\label{HOD790}
& \sum_{1 \leq j \leq N}  \left[ \sum_{1 \leq \alpha,\beta \leq n  } A_{ij}^{\alpha \beta}(y) \pi_{\alpha}(y) \pi_{\beta}(y) \right] \left[  \zeta_{1}^{j} (y) - \zeta_{1}^{j}(x) \right] \\
& \quad =  \sum_{1 \leq j \leq N}  \left[ \sum_{1 \leq \alpha,\beta \leq n  } A_{ij}^{\alpha \beta}(y) \pi_{\alpha}(y) \pi_{\beta}(y) \right]  \zeta_{1}^{j} (y) \\
& \qquad +   \sum_{1 \leq j \leq N} \sum_{1 \leq \alpha,\beta \leq n  } \left[  A_{ij}^{\alpha \beta}(x)  \pi_{\alpha}(x) \pi_{\beta}(x)- A_{ij}^{\alpha \beta}(y) \pi_{\alpha}(y) \pi_{\beta}(y)  \right] \zeta_{1}^{j}(x) \\
& \qquad -  \sum_{1 \leq j \leq N} \left[ \sum_{1 \leq \alpha,\beta \leq n  }  A_{ij}^{\alpha \beta}(x) \pi_{\alpha}(x) \pi_{\beta}(x) \right] \zeta_{1}^{j} (x),
\end{aligned}\end{equation}
for any $i  = 1, \cdots, N$. With \eqref{ell1}, we use the left-hand side of \eqref{HOD790} as follows : 
\begin{equation*}\begin{aligned}\label{}
& \lambda |\zeta_1(x) - \zeta_1(y)|^{2} |\pi(y)|^{2} \\
& \quad \leq \sum_{1 \leq i,j \leq N} \left[  \zeta_{1}^{i} (y) - \zeta_{1}^{i}(x) \right]  \left[ \sum_{1 \leq \alpha,\beta \leq N  }A_{ij}^{\alpha \beta}(y) \pi_{\alpha}(y) \pi_{\beta}(y) \right] \left[  \zeta_{1}^{j} (y) - \zeta_{1}^{j}(x) \right]  \\
& \quad \leq  |\zeta_1(x) - \zeta_1(y)| \sum_{1 \leq i \leq N}  \left| \sum_{1 \leq j \leq N}  \left[ \sum_{1 \leq \alpha,\beta \leq n  } A_{ij}^{\alpha \beta}(y) \pi_{\alpha}(y) \pi_{\beta}(y) \right] \left[  \zeta_{1}^{j} (y) - \zeta_{1}^{j}(x) \right] \right|,
\end{aligned}\end{equation*}
which implies that
\begin{equation}\begin{aligned}\label{HOD810}
& \lambda |\zeta_1(x) - \zeta_1(y)| |\pi(y)|^{2}  \\
& \quad \leq \sum_{1 \leq i \leq N} \left| \sum_{1 \leq j \leq N} \left[ \sum_{1 \leq \alpha,\beta \leq n  }  A_{ij}^{\alpha \beta}(y) \pi_{\alpha}(y) \pi_{\beta}(y) \right] \left[  \zeta_{1}^{j} (y) - \zeta_{1}^{j}(x) \right] \right|.
\end{aligned}\end{equation}
We next estimate the right-hand side of \eqref{HOD790}. With \eqref{HOD780}, one can prove that
\begin{equation*}\begin{aligned}\label{}
& \left| \sum_{1 \leq j \leq N}  \sum_{1 \leq \alpha,\beta \leq n  } \left[ A_{ij}^{\alpha \beta}(y) \pi_{\alpha}(y) \pi_{\beta}(y)   \zeta_{1}^{j}(y) - A_{ij}^{\alpha \beta}(x) \pi_{\alpha}(x) \pi_{\beta}(x)  \zeta_{1}^{j}(x) \right]  \right| \\
& \quad \leq c |\pi(y)| \Big[ |U(y) - U(x)| + |F(y) - F(x)| \Big] \\
& \qquad +  c  \Big[ |U(x)| + |F(x)| \Big]  \left| \pi(x) - \pi(y) \right|  \\
& \qquad +  c  \Big[ |U(x)| + |F(x)| \Big]   |\pi(y)| \sum_{1 \leq i,j \leq N}   \sum_{1 \leq \alpha, \beta \leq n  }  \left| A_{ij}^{\alpha \beta}(y)  - A_{ij}^{\alpha \beta}(x) \right|,
\end{aligned}\end{equation*}
for any $i  = 1, \cdots, N$. Since $\pi_{1}=-1$, one can check from \eqref{ell2} that
\begin{equation*}\begin{aligned}
& \left| \sum_{1 \leq j \leq N} \sum_{1 \leq \alpha,\beta \leq n  } \left[  A_{ij}^{\alpha \beta}(x)  \pi_{\alpha}(x) \pi_{\beta}(x)- A_{ij}^{\alpha \beta}(y) \pi_{\alpha}(y) \pi_{\beta}(y)  \right] \zeta_{1}^{j}(x) \right| \\
& \quad \leq c 
\sum_{1 \leq i,j \leq N}  \sum_{1 \leq \alpha,\beta \leq n  }  \left|  A_{ij}^{\alpha \beta}(x) - A_{ij}^{\alpha \beta}(y) \right||\pi(x)| |\pi(y)| \\
& \qquad + c \left| \pi(x) - \pi(y) \right|\Big[ |\pi(x)| + |\pi(y)| \Big]   |\zeta_1(x)|,
\end{aligned}\end{equation*}
for any $i  = 1, \cdots, N$. In view of Lemma \ref{HODS600}, we have that
\begin{equation}\label{HOD820}
|\zeta_1(x)| \leq |\pi(x)||\zeta_1(x)| \leq c(n,N,\lambda,\Lambda) \Big[ |U(x)| + |F(x)| \Big].
\end{equation}
By applying the above three estimates and that $\pi(x) = \pi(x) - \pi(y) + \pi(y)$ to \eqref{HOD790}, we get that
\begin{equation*}\begin{aligned}\label{}
& \left| \sum_{1 \leq j \leq N}  \left[ \sum_{1 \leq \alpha,\beta \leq n  } A_{ij}^{\alpha \beta}(y) \pi_{\alpha}(y) \pi_{\beta}(y) \right] \left[  \zeta_{1}^{j} (y) - \zeta_{1}^{j}(x) \right] \right|\\
& \quad \leq c |\pi(y)| \left[ |U(y) - U(x)| + |F(y) - F(x)| \right] \\
& \qquad +  c  \Big[ |U(x)| + |F(x)| \Big] \left[ |\pi(y)| \, \left| \pi(x) - \pi(y)\right| +  \left| \pi(x) - \pi(y)\right|^{2}  \right] \\
& \qquad +  c  \Big[ |U(x)| + |F(x)| \Big] \left[ |\pi(y)| \sum_{1 \leq i,j \leq N}   \sum_{1 \leq \alpha, \beta \leq n  }  \left| A_{ij}^{\alpha \beta}(y)  - A_{ij}^{\alpha \beta}(x) \right| \right],
\end{aligned}\end{equation*}
for any $i  = 1, \cdots, N$. Since $\pi(y)  = \big( -1,\pi_{2}(y), \cdots, \pi_{n}(y) \big)$, we find from \eqref{HOD810} that 
\begin{equation}\begin{aligned}\label{HOD830}
& |\pi(y)||\zeta_1(x) - \zeta_1(y)| \\
& \quad \leq c \Big[ |U(y) - U(x)| + |F(y) - F(x)| \Big] \\
& \qquad +  c  \Big[ |U(x)| + |F(x)| \Big] \left[  \left| \pi(x) - \pi(y)\right| +  \left| \pi(x) - \pi(y)\right|^{2}  \right] \\
& \qquad +  c  \Big[ |U(x)| + |F(x)| \Big]  \sum_{1 \leq i,j \leq N}   \sum_{1 \leq \alpha, \beta \leq n  }  \left| A_{ij}^{\alpha \beta}(y)  - A_{ij}^{\alpha \beta}(x) \right|.
\end{aligned}\end{equation}
Moreover, combining \eqref{HOD740}, \eqref{HOD820} and \eqref{HOD830}, we discover that
\begin{equation*}\begin{aligned}
|\zeta_\beta(y)-\zeta_\beta(x)|&\leq |U(y)-U(x)|+|\pi(y)-\pi(x)||\zeta_1(x)|+  |\pi(y)||\zeta_1(y)-\zeta_1(x)| \\ 
 & \leq c \Big[ |U(y) - U(x)| + |F(y) - F(x)| \Big] \\
& \quad +  c  \Big[ |U(x)| + |F(x)| \Big] \left[  \left| \pi(x) - \pi(y)\right| +  \left| \pi(x) - \pi(y)\right|^{2}  \right] \\
& \quad +  c  \Big[ |U(x)| + |F(x)| \Big]  \sum_{1 \leq i,j \leq N}   \sum_{1 \leq \alpha, \beta \leq n  }  \left| A_{ij}^{\alpha \beta}(y)  - A_{ij}^{\alpha \beta}(x) \right|.
\end{aligned}\end{equation*} 
for any $\beta = 2,\cdots,n$. With $|\pi(y)|\geq 1$, this and \eqref{HOD830} complete the proof.

\end{proof}

\section{Excess decay estimates}\label{GEW}

We obtain the desired excess decay estimates in this section. For the first subsection, we consider the case when $|\pi'|$ decays with respect to the size of the cube and the coefficients are piecewise constant. Then in the next subsection, we consider the case when $|\pi'|$ has no decay assumption and the coefficients are piecewise constant. For the last subsection, we handle the case when the coefficients are piecewise H\"{o}lder continuous with no decay assumption on $|\pi'|$ by using the perturbation argument.

\subsection{Piecewise constant coefficients with decay assumption}\label{PCCDA}

Choose a size $\tau \in (0,R]$. For a composite cube $\left( Q_{\tau}, \left\{ \varphi_{k} : k \in K_{+} \right\} \right)$, let $\pi$ be  the derivative of the naturally induced flow $\pi : Q_{\tau} \to \br^{n}$ in Definition \ref{derivative of flow}. Then
\begin{equation*}\label{}
\pi = (-1,\pi') = ( -1, \pi_{2}, \cdots, \pi_{n} )
\end{equation*}
where
\begin{equation*}\begin{aligned}\label{}
\pi_{\alpha}(x) = D_{\alpha}\varphi_{k+1}( x') \cdot T_{k}(x)
+ D_{\alpha}\varphi_{k}( x')  \cdot [1-T_{k}(x)]
\quad \text{ in } \quad Q_{\tau}^{k},
\end{aligned}\end{equation*}
for any $k \in K$ and  $\alpha \in \{ 2, \cdots, n \}$. For some universal constant $\nu \geq 1$ which will be determined later, we also assume an decay of $\pi'$ that
\begin{equation}\label{PCC230}
|\pi'(0)|=0
\qquad \text{and} \qquad
|\pi'| \leq \nu \left( \frac{ \rho }{ R } \right)^{2\mu}
\quad \text{in} \quad Q_{\rho},
\end{equation}
for any $ 0< \rho \leq \tau$.

\sskip

For this subsection, we employ the letter $c \geq 1$ to denote any constants that can be explicitly computed in terms such as $n$, $N$, $\lambda$, $\Lambda$, $\kappa$, $R^{\gamma} \sup_{k \in K_{+}} \left[ D_{x'}\varphi_{k} \right]_{C^{\gamma}(Q_{\tau}')}$ and the number of elements in the set $K$.

\sskip

We first handle the case when the coefficients are piecewise constant. For the constants $A_{ij,k}^{\alpha \beta}, F_{\alpha,k}^{i}$ $(1 \leq \alpha, \beta \leq n, \ 1 \leq i,j \leq N, ~ k \in K)$ satisfying that 
\begin{equation}\label{PCC240} 
\lambda |\xi|^{2} 
\leq A_{ij,k}^{\alpha \beta } \xi_{\alpha}^{i} \xi_{\beta}^{j} 
\qquad \text{and} \qquad
\left| A_{ij,k}^{\alpha \beta} \right| \leq \Lambda 
\qqquad
\left( \xi \in \br^{Nn} \right),
\end{equation}
we define $A_{ij}^{\alpha \beta}, F_{\alpha}^{i}$ $(1 \leq \alpha, \beta \leq n, \ 1 \leq i,j \leq N)$ as 
\begin{equation}\label{PCC250}
A_{ij}^{\alpha \beta}(x) 
= \sum_{k \in K} A_{ij,k}^{\alpha \beta} \chi_{Q_{\tau}^{k}},
\qquad \text{and} \qquad
F_{\alpha}^{i}(x) 
= \sum_{k \in K} F_{\alpha,k}^{i} \chi_{Q_{\tau}^{k}}.
\end{equation}
We remark that $A_{ij}^{\alpha \beta }$ and $F_{\alpha}^{i}$ $(1 \leq \alpha, \beta \leq n, \ 1 \leq i,j \leq N )$  are constant in each $Q_{\tau}^{k}$ $(k \in K)$. Then one can check from \eqref{PCC240} that 
\begin{equation}\label{PCC260} 
\lambda |\xi|^{2} 
\leq A_{ij}^{\alpha \beta }(x) \xi_{\alpha}^{i} \xi_{\beta}^{j} 
\qquad \text{and} \qquad
\big| A_{ij}^{\alpha \beta}(x) \big| \leq \Lambda 
\qquad \quad
\left( x \in Q_{\tau}, ~ \xi \in \br^{Nn} \right).
\end{equation} 

\smallskip

Let $w$ be a weak solution of
\begin{equation}\label{PCC270}
D_{\alpha} \left[ A_{ij}^{\alpha \beta}(x) D_{\beta}w^{j} \right]  =  D_{\alpha} F_{\alpha}^{i} ~ \text{ in } ~ Q_{\tau}.
\end{equation}
By using Gehring-Giaquinta-Modica type inequality, see for instance \cite[Theorem 5.6]{BSWL1}, one can prove that
\begin{equation}\begin{aligned}\label{PCC280}
\left( \mint_{ Q_{\frac{\theta}{2}} } |Dw|^{2+\sigma} \, dx \right)^{\frac{2}{2+\sigma}}
\leq c \left(  \mint_{Q_{\theta}} |Dw|^{2}  \, dx + \left \| F \right \|_{L^{\infty}(Q_{\theta})}^{2} \right)
\qquad \left( \theta \in (0,\tau] \right),
\end{aligned}\end{equation}
for some small universal constant $\sigma \in (0,1]$. We define  $W : Q_{\tau} \to \br^{Nn}$  as $W = \left( W^{1}, \cdots, W^{N} \right)^{T}$ where
\begin{equation}\label{PCC290}
W^{i} = \left( \left[ \left( - \sum_{1 \leq j \leq N} \sum_{1 \leq \beta \leq n} A_{ij}^{1 \beta} D_{\beta}w^{j} \right) + F_{1}^{i} \right], D_{x'}w^{i} \right),
\end{equation}
for any $1 \leq i \leq N$.

\smallskip

Fix $\theta \in (0,\tau] $. For each $k \in K_{+}$, one can choose $z_{k}' \in Q_{\theta}'$ so that
\begin{equation}\label{PCC313}
\varphi_{k}(z_{k}') \leq \varphi_{k+1}(z_{k+1}')
\qquad (k \in K),
\end{equation}
and
\begin{equation}\label{PCC310}
(\varphi_{k}(z_{k}'),z_{k}') \in Q_{ \theta }
\quad \text{if} \quad
Q_{\theta} \cap \left\{ (\varphi_{k}(x'),x') : x' \in Q_{\theta}' \right\} \not = \emptyset
\qquad (k \in K_{+}).
\end{equation}
Set
\begin{equation}\label{PCC315}
z_{k} = (\varphi_{k}(z_{k}'),z_{k}').
\end{equation}
Then by \eqref{PCC313} and \eqref{PCC310},
\begin{equation}\label{PCC320}
z_{k}^{1} \leq z_{k+1}^{1}
\qquad (k \in K).
\end{equation}
It follows from the definition of $\pi'$ in Definition \ref{derivative of flow}, \eqref{PCC230} and \eqref{PCC310} that
\begin{equation}\label{PCC330} 
z_{k} \in Q_{\theta} 
\quad \Longrightarrow \quad
|D_{x'}\varphi_{k}(z_{k}')| 
= | \pi'(\varphi_{k}(z_{k}'),z_{k}')| 
= |\pi'(z_{k})|
\leq \nu  \left( \frac{ \theta }{ R } \right)^{2\mu}.
\end{equation}
Since $D_{x'}\varphi_{k} \in C^{\gamma} \left( Q_{\theta}' \right)$ $(k \in K_{+})$ and $2\mu = \frac{\gamma}{\gamma+1}$ in \eqref{mu}, we discover that $R^{2\mu} \left[ D_{x'}\varphi_{k} \right]_{C^{2\mu}(Q_{\theta}')} \leq c R^{\gamma} \left[ D_{x'}\varphi_{k} \right]_{C^{\gamma}(Q_{\theta}')}$. So we find from \eqref{PCC310} and \eqref{PCC330} that
\begin{equation}\label{PCC340} 
Q_{ \theta } \cap \left \{ (\varphi_{k}(x'),x') : x' \in Q_{ \theta }' \right \} \not = \emptyset
\quad \Lra \quad
|D_{x'}\varphi_{k}|\leq \nu  \left( \frac{ \theta }{ R } \right)^{2\mu} ~  \text{ in } ~  Q_{ \theta }',
\end{equation}
for any $k \in K_{+}$.

We define $\bar{A}_{ij}^{\alpha \beta}, \bar{F}_{\alpha}^{i} $ $(1 \leq \alpha, \beta \leq n, \ 1 \leq i,j \leq N)$ as 
\begin{equation}\label{PCC350}
\bar{A}_{ij}^{\alpha \beta}(x^{1}) 
= \sum_{k \in K} A_{ij,k}^{\alpha \beta} \chi_{z_{k}^{1} < x^{1} \leq z_{k+1}^{1}}
~ \text{and} \
\bar{F}_{\alpha}^{i}(x^{1}) 
= \sum_{k \in K} F_{\alpha,k}^{i} \chi_{z_{k}^{1} < x^{1} \leq z_{k+1}^{1}}
~ \text{in} \  Q_{\theta}.
\end{equation}
Then one can check from \eqref{PCC240} that 
\begin{equation}\label{PCC360} 
\lambda |\xi|^{2} 
\leq \bar{A}_{ij}^{\alpha \beta }(x^1) \xi_{\alpha}^{i} \xi_{\beta}^{j} 
\quad \text{and} \quad
\big| \bar{A}_{ij}^{\alpha \beta}(x^1) \big| \leq \Lambda 
\qquad 
\left( x^1 \in (-\theta,\theta), ~ \xi \in \br^{Nn} \right).
\end{equation}
Also one can  check from \eqref{PCC310} and \eqref{PCC320} that
\begin{equation}\label{PCC370} 
\big\| \bar{F} \big\|_{L^{\infty}(Q_{\theta})}
\leq  \big\| F \big\|_{L^{\infty}(Q_{\theta})},
\end{equation}
by using that 
\begin{equation*}
Q_{\theta}^{k} = \emptyset 
\qquad \Lra \qquad 
\left\{ (x^{1},x') \in Q_{\theta} : z_{k}^{1} < x^{1} \leq  z_{k+1}^{1} \right\} = \emptyset.
\end{equation*}
With \eqref{PCC315} and \eqref{PCC340}, one can compare $A_{ij}^{\alpha \beta}$ and $F_{\alpha}^{i}(x) $ with $\bar{A}_{ij}^{\alpha \beta}$ and $\bar{F}_{\alpha}^{i}(x)$ respectively:
\begin{equation}\label{PCC380}
\left| \left\{ x \in Q_{\theta} : \bar{A}_{ij}^{\alpha \beta} \not = A_{ij}^{\alpha \beta} \right\} \right| 
+ \left| \left\{ x \in Q_{\theta} : \bar{F}_{\alpha}^{i} \not = F_{\alpha}^{i} \right\} \right|
\leq c  \nu  \left(\frac{\theta}{R}\right)^{2\mu} \theta^{n}.
\end{equation}

Let $h$ be the weak solution of
\begin{equation}\label{PCC410}\left\{\begin{array}{rcll}
D_{\alpha} \left[ \bar{A}_{ij}^{\alpha \beta}(x^{1}) D_{\beta}h^{j} \right] & = & D_{\alpha} \bar{F}_{\alpha}^{i}(x^{1}) & ~ \text{ in } ~ Q_{\theta},\\
h & = & w & ~ \text{ in } ~ \partial Q_{\theta}.
\end{array}\right.\end{equation}
Then set $H : Q_{\theta} \to \br^{Nn}$ as $H = \left( H^{1}, \cdots, H^{N} \right)^{T}$ where
\begin{equation}\label{PCC420} 
H^{i} = \left( \left[ \left( - \sum_{1 \leq j \leq N} \sum_{1 \leq \beta \leq n } \bar{A}_{ij}^{1 \beta} D_{\beta}h^{j} \right) + \bar{F}_{1}^{i} \right], D_{x'}h^{i} \right),
\end{equation}
for any $1 \leq i \leq N$.

\begin{lem}\label{PCCS500}
Suppose \eqref{PCC230}, \eqref{PCC240} and that $Dw \in L^{2+\sigma}(Q_{\theta})$ for some $\sigma \in (0,\infty]$. Then for $H$ in \eqref{PCC420}, we have that
\begin{equation}\label{PCC530} 
\mint_{Q_{\theta}} |W - H|^{2}  dx
\leq c  \nu^{\frac{\sigma}{2+\sigma}}  \left( \frac{ \theta }{ R } \right)^{\frac{ 2\mu \sigma }{2+\sigma} }  
\left[ \left( \mint_{Q_{\theta}} |Dw|^{2+\sigma}   dx \right)^{\frac{2}{2+\sigma}} + \left \| F \right \|_{L^{\infty} \left( Q_{\theta} \right)}^{2} \right].
\end{equation}
\end{lem}

\begin{proof}
We first estimate $Dw- Dh$. We test \eqref{PCC270} and \eqref{PCC410} by $w-h$ in $Q_{\theta}$ to find that 
\begin{equation*}\begin{aligned}\label{} 
& \mint_{Q_{\theta}} \left \langle \bar{A}_{ij}^{\alpha \beta} \left[ D_{\beta}w^{j} - D_{\beta} h^{j} \right] , D_{\alpha}w^{i} - D_{\alpha}h^{i} \right \rangle \, dx\\
& \quad =  \mint_{Q_{\theta}} \left \langle \left[ \bar{A}_{ij}^{\alpha \beta}  - A_{ij}^{\alpha \beta} \right]  D_{\beta}w^{j}, D_{\alpha}w^{i} - D_{\alpha}h^{i} \right \rangle + \left \langle F_{\alpha}^{i} - \bar{F}_{\alpha}^{i}, D_{\alpha}w^{i} - D_{\alpha}h^{i} \right \rangle \, dx.
\end{aligned}\end{equation*}
By Young's inequality, we obtain from \eqref{PCC360} that
\begin{equation}\label{PCC550} 
\mint_{Q_{\theta}} \big| Dw - Dh \big|^{2}   dx
\leq c \left[ \mint_{Q_{\theta}} \left| \bar{A}_{ij}^{\alpha \beta} - A_{ij}^{\alpha \beta} \right|^{2}|Dw|^{2} \, dx 
+  \mint_{Q_{\theta}} \big| \bar{F}_{\alpha}^{i} - F_{\alpha}^{i} \big|^{2} \ dx \right].
\end{equation}
By H\"{o}lder's inequality, we obtain that
\begin{equation*}
\mint_{Q_{\theta}} \left| \bar{A}_{ij}^{\alpha \beta} - A_{ij}^{\alpha \beta} \right|^{2}|Dw|^{2} \, dx
\leq \left[ \mint_{Q_{\theta}} \left| \bar{A}_{ij}^{\alpha \beta} - A_{ij}^{\alpha \beta} \right|^{\frac{2(2+\sigma)}{\sigma} } \right]^{\frac{\sigma}{2+\sigma}} \left[ \mint_{Q_{\theta}} |Dw|^{2+\sigma} \, dx \right]^{\frac{2}{2+\sigma}}.
\end{equation*}
So we find from \eqref{PCC360} and \eqref{PCC380} that
\begin{equation}\begin{aligned}\label{PCC570}
\mint_{Q_{\theta}} \left| \bar{A}_{ij}^{\alpha \beta} - A_{ij}^{\alpha \beta} \right|^{2}|Dw|^{2} \, dx 
& \leq c \nu^{\frac{\sigma}{2+\sigma}}  \left( \frac{ \theta }{ R } \right)^{\frac{ 2\mu \sigma }{2+\sigma} }  
\left( \mint_{Q_{\theta}} |Dw|^{2+\sigma} \, dx \right)^{\frac{2}{2+\sigma}}.
\end{aligned}\end{equation}
In view of \eqref{PCC380}, we obtain that
\begin{equation}\begin{aligned}\label{PCC580} 
\int_{Q_{\theta}} \left| \bar{F}_{\alpha}^{i} - F_{\alpha}^{i} \right|^{2} \, dx
& \leq \int_{ \left\{ x \in Q_{\theta} : F \not = \bar{F}  \right\}  } \left| \bar{F} - F \right|^{2} \, dx \\
& \leq \left|\left\{ x \in Q_{\theta} : F \not = \bar{F}  \right\}\right|^{\frac{\sigma}{2+\sigma}} 
\left( \int_{Q_{\theta}} \left| \bar{F}_{\alpha}^{i} - F_{\alpha}^{i} \right|^{2+\sigma} \right)^{\frac{2}{2+\sigma}} \\
& \leq c \nu^{\frac{\sigma}{2+\sigma}}  \left( \frac{ \theta }{ R } \right)^{\frac{ 2\mu \sigma }{2+\sigma} }   \theta^{n} \left[ \left\| \bar{F} \right\|_{L^{\infty}(Q_{\theta})}^{2} + \left\| F \right\|_{L^{\infty}(Q_{\theta})}^{2} \right].
\end{aligned}\end{equation}
With \eqref{PCC370}, it follows from \eqref{PCC550}, \eqref{PCC570} and \eqref{PCC580} that
\begin{equation}\begin{aligned}\label{PCC590} 
& \mint_{Q_{\theta}} |Dw - Dh|^{2} \, dx \\
& \qquad \leq c \nu^{\frac{\sigma}{2+\sigma}}  \left( \frac{ \theta }{ R } \right)^{\frac{ 2\mu \sigma }{2+\sigma} }  \left[   \left(  \mint_{Q_{\theta}} |Dw|^{2+\sigma} \, dx  \right)^\frac{2}{2+\sigma} +  \left\| F \right\|_{L^{\infty} \left( Q_{\theta} \right)}^{2}\right].
\end{aligned}\end{equation}
We obtain from \eqref{PCC570} that
\begin{equation}\begin{aligned}\label{PCC600}
\int_{Q_{\theta}} \left| \bar{A}_{ij}^{1 \beta} D_{\beta}w^{j} - A_{ij}^{1 \beta}  D_{\beta}w^{j} \right|^{2}  dx 
& \leq c \mint_{Q_{\theta}} \left| \bar{A}_{ij}^{\alpha \beta} - A_{ij}^{\alpha \beta} \right|^{2} |Dw|^{2} \, dx  \\
& \leq c \nu^{\frac{\sigma}{2+\sigma}}  \left( \frac{ \theta }{ R } \right)^{\frac{ 2\mu \sigma }{2+\sigma} }  \left( \mint_{Q_{\theta}} |Dw|^{2+\sigma} dx \right)^{\frac{2}{2+\sigma}}.
\end{aligned}\end{equation}
for any $1 \leq \beta \leq n$ and $1 \leq i,j \leq N$. With \eqref{PCC290} and \eqref{PCC420}, we have from \eqref{PCC590} and \eqref{PCC600} that
\begin{equation*}\begin{aligned}\label{} 
\mint_{Q_{\theta}} |W - H|^{2} \, dx
& \leq c \nu^{\frac{\sigma}{2+\sigma}}  \left( \frac{ \theta }{ R } \right)^{\frac{ 2\mu \sigma }{2+\sigma} }    \left[ \left( \mint_{Q_{\theta}} |Dw|^{2+\sigma}  \, dx \right)^{\frac{2}{2+\sigma}} + \left \| F \right \|_{L^{\infty}(Q_{\theta})}^{2} \right],
\end{aligned}\end{equation*}
where $H : Q_{\theta} \to \br^{N}$ is defined in \eqref{PCC420}.
\end{proof}

\begin{lem}\label{PCCS800}
Suppose \eqref{PCC230}, \eqref{PCC240} and that $Dw \in L^{2+\sigma}(Q_{\theta})$ for some $\sigma \in (0,\infty]$.  Let $w$ be the weak solution of \eqref{PCC270}. Then 
\begin{equation*}\begin{aligned}\label{}
& \mint_{Q_{\rho}} \left| W  - (W)_{Q_{\rho}} \right|^{2} \, dx \\
& \quad \leq c \left( \frac{\rho}{\theta} \right) \mint_{Q_{\theta}} \left| W - (W)_{Q_{\theta}} \right|^{2} \, dx \\
& \qquad + c \nu^{\frac{\sigma}{2+\sigma}}  \left( \frac{ \theta }{ R } \right)^{\frac{ 2\mu \sigma }{2+\sigma} }   \left( \frac{\theta}{\rho} \right)^{n}  \left[ \left( \mint_{Q_{\theta}} |Dw|^{2+\sigma}  \, dx \right)^{\frac{2}{2+\sigma}} + \left \| F \right \|_{L^{\infty}(Q_{\theta})}^{2} \right] ,
\end{aligned}\end{equation*}
for any $0 < \rho \leq \theta \leq \tau$.
\end{lem}

\begin{proof}
Let $h$ be the weak solution of \eqref{PCC410} and set $H$ as in \eqref{PCC420}. Since $\bar{F}_{\alpha}^{i}(x^{1})$ $(1 \leq \alpha \leq n, 1 \leq i \leq N)$ are independent of $x'$-variables, \eqref{PCC410} yields that
\begin{equation}\label{PCC840} 
D_{\alpha} \left[ \bar{A}_{ij}^{\alpha \beta}(x^{1}) D_{\beta}h^{j} \right]   = D_{1} \bar{F}_{1}^{i}(x^{1}) ~ \text{ in } ~ Q_{\theta}.
 \end{equation}
Apply Lemma \ref{GELES7000} to $h$ in \eqref{PCC840} and $H$ in \eqref{PCC420} instead of $w$ in  \eqref{GELE6100} and $W$ in \eqref{GELE7100} respectively. Then we have that
\begin{equation*}\label{}
\mint_{Q_{\rho}} \left| H- (H)_{Q_{\rho}} \right|^{2} \, dx
\leq c \Big( \frac{\rho}{\theta} \Big)
\mint_{Q_{\theta}} \left| H - (H)_{Q_{\theta}} \right|^{2} \, dx.
\end{equation*}
It follows from Lemma \ref{PCCS500} that
\begin{equation*}\begin{aligned}
& \mint_{Q_{\rho}} \left| W  - (W)_{Q_{\rho}} \right|^{2} \, dx  \\
& \quad \leq c  \left( \frac{\rho}{\theta} \right)
\mint_{Q_{\theta}} \left| W - (W)_{Q_{\theta}} \right|^{2} \, dx \\
& \qquad + c \nu^{\frac{\sigma}{2+\sigma}}  \left( \frac{ \theta }{ R } \right)^{\frac{ 2\mu \sigma }{2+\sigma} }   \left( \frac{\theta}{\rho} \right)^{n}    \left[ \left( \mint_{Q_{\theta}} |Dw|^{2+\sigma}  \, dx \right)^{\frac{2}{2+\sigma}} + \left \| F \right \|_{L^{\infty}(Q_{\theta})}^{2} \right].
\end{aligned}\end{equation*}
\end{proof}

\begin{lem}\label{PCCS900}
There exists a constant $\varepsilon \in (0,1]$ such that if $\nu \left( \frac{ \tau }{ R } \right)^{2\mu} \leq \varepsilon$ then
\begin{equation*}\label{}
\mint_{Q_{\rho}} |W|^{2} \, dx \leq c \left[ \mint_{Q_{\tau}} |W|^{2} \, dx +  \| F \|_{L^{\infty}(Q_{\tau})}^{2} \right],
\end{equation*}
for any $ 0 < \rho \leq \tau$. Here, $\varepsilon \in (0,1]$ depends only on the constants $n$, $N$, $\lambda$, $\Lambda$, $R^{\gamma} \sup_{k \in K_{+}} [D_{x'}\varphi_{k} ]_{C^{\gamma}(Q_{\tau}')}$ and the number of the element in the set $K$.
\end{lem}

\begin{proof}
The proof is similar to the paper \cite{DFMG1,KTMG1,KY2}, but we give the proof for the sake of the completeness. 

By applying Lemma \ref{PCCS800} to \eqref{PCC270} and applying Lemma \ref{HODS600} to $Dw$ and $W$ instead of $\zeta$ and $U$ with $\pi=(-1,0,\cdots,0)$, we find that 
\begin{equation}\begin{aligned}\label{PCC915}
& \mint_{Q_{\rho}} \left| W  - (W)_{Q_{\rho}} \right|^{2} \, dx  \\
& \quad \leq c  \left( \frac{\rho}{\theta} \right) \mint_{Q_{\theta}} \left| W - (W)_{Q_{\theta}} \right|^{2} \, dx \\
& \qquad + c \nu^{\frac{\sigma}{2+\sigma}}  \left( \frac{ \theta }{ R } \right)^{\frac{ 2\mu \sigma }{2+\sigma} }   \left( \frac{\theta}{\rho} \right)^{n} \left( \mint_{Q_{\theta}} |W|^{2} \, dx 
+ \| F \|_{L^{\infty}(Q_{\theta})}^{2} \right),
\end{aligned}\end{equation}
for any $0 < \rho \leq \theta \leq \tau$, which implies that
\begin{equation*}\begin{aligned}
& \left( \mint_{Q_{\rho}} \left| W  - (W)_{Q_{\rho}} \right|^{2} \, dx \right)^{\frac{1}{2}} \\
& \quad \leq c_{1} \left( \frac{\rho}{\theta} \right)^{\frac{1}{2}} \left( \mint_{Q_{\theta}} \left| W - (W)_{Q_{\theta}} \right|^{2} \, dx \right)^{\frac{1}{2}} \\
& \qquad + c_{1} \nu^\frac{ \sigma }{2(2+\sigma)} \left( \frac{ \theta }{ R } \right)^\frac{ \mu \sigma }{2+\sigma}\left( \frac{\theta}{\rho} \right)^\frac{n}{2} \left[\left( \mint_{Q_{\theta}} |W|^{2} \, dx\right)^\frac{1}{2}
+ \| F \|_{L^{\infty}(Q_{\tau})} \right].
\end{aligned}\end{equation*}
For a small constant $\delta \in (0,1)$ chosen to be later, let $\tau_{i} = \delta^{i} \tau$.
By letting $\rho=\tau_{i+1}$ and $\theta=\tau_{i}$,
\begin{equation*}\begin{aligned}
& \left( \mint_{Q_{\tau_{i+1}}} \left| W  - (W)_{Q_{\tau_{i+1}}} \right|^{2} \, dx \right)^{\frac{1}{2}} \\
& \quad \leq c_1 \delta^{\frac{1}{2}} \left( \mint_{Q_{ \tau_{i} }} \left| W - (W)_{Q_{ \tau_{i} }} \right|^{2} \, dx \right)^{\frac{1}{2}} \\
& \qquad + c_{1}\delta^{-\frac{n}{2}}  \nu^{\frac{\sigma}{2(2+\sigma)} } \left(\frac{\tau_{i}}{R}\right)^{\frac{\mu\sigma}{2+\sigma} }     \left[ \left(\mint_{Q_{  \tau_{i}  }} |W|^{2} \, dx \right)^\frac{1}{2}
+ \| F \|_{L^{\infty}(Q_{\tau})} \right],
\end{aligned}\end{equation*}
for any $i=0,1,2, \cdots$.  
Choose the universal constant  $\delta \in (0,1)$ so that $c_{1} \delta^{\frac{1}{2}} \leq \frac{1}{4}$. Since the constant $\sigma \in (0,1]$ chosen in \eqref{PCC280} is universal, select the universal constant $\varepsilon \in (0,1]$ so that
\begin{equation*}\label{}
\frac{ 20 c_{1} \delta^{-n} \varepsilon^{\frac{\sigma}{2(2+\sigma)}}   }{1-\delta^{ \frac{\mu \sigma}{2(2+\sigma) } }  } 
\leq 1,
\end{equation*}
which implies that
\begin{equation}\begin{aligned}\label{PCC930}
& 20 c_{1}  \nu^{\frac{\sigma}{2(2+\sigma)} } \left(\frac{\tau}{R}\right)^{\frac{\mu\sigma}{2+\sigma} }    \delta^{-n} \sum_{i=0}^{j}\delta_{}^{ \frac{\mu \sigma i}{ 2+\sigma } } \\
& \qquad \leq \frac{ 20 c_{1}  \delta^{-n} \nu^{\frac{\sigma}{2(2+\sigma)} } }{ 1-\delta^{ \frac{\mu \sigma}{ 2(2+\sigma) } } } \left(\frac{\tau}{R}\right)^{\frac{\mu\sigma}{2+\sigma} } 
\leq \frac{ 20  c_{1} \delta^{-n} \varepsilon^{\frac{\sigma}{2(2+\sigma)}}   }{1-\delta^{ \frac{\mu \sigma}{2(2+\sigma) } }  } 
\leq 1,
\end{aligned}\end{equation}
for any $j=0,1,2,\cdots$. Since 
\begin{equation*}
\left(\mint_{Q_{  \tau_{i}  }} |W|^{2} \, dx \right)^\frac{1}{2} \leq \left( \mint_{Q_{ \tau_{i} }} \left| W - (W)_{Q_{ \tau_{i} }} \right|^{2} \, dx \right)^{\frac{1}{2}} + \left| (W)_{Q_{ \tau_{i} }} \right|, 
\end{equation*}
we have that
\begin{equation*}\begin{aligned}
& \left( \mint_{Q_{\tau_{i+1}}} \left| W  - (W)_{Q_{\tau_{i+1}}} \right|^{2} \, dx \right)^{\frac{1}{2}} \\
& \quad \leq \frac{1}{2} \left( \mint_{Q_{ \tau_{i} }} \left| W - (W)_{Q_{ \tau_{i} }} \right|^{2} \, dx \right)^{\frac{1}{2}} \\
& \qquad + c_{1}\delta^{-\frac{n}{2}}  \nu^{\frac{\sigma}{2(2+\sigma)} } \left(\frac{\tau_{i}}{R}\right)^{\frac{\mu\sigma}{2+\sigma} }  \left(  \left| (W)_{Q_{ \tau_{i} }} \right|
+  \| F \|_{L^{\infty}(Q_{\tau})} \right),
\end{aligned}\end{equation*}
from \eqref{PCC930}. By summing up over $i=0,1,2,\cdots, j$, it follows that 
\begin{equation}\begin{aligned}\label{PCC950}
& \sum_{i=0}^{j+1} \left( \mint_{Q_{\tau_{i}}} \left| W  - (W)_{Q_{\tau_{i}}} \right|^{2} \, dx \right)^{\frac{1}{2}} \\
&\qquad \leq 4 \left( \mint_{Q_{\tau_{0}}} \left| W  - (W)_{Q_{\tau_{0}}} \right|^{2} \, dx \right)^{\frac{1}{2}} \\
& \qquad \qquad + 2 c_{1}\delta^{-\frac{n}{2}}  \nu^{\frac{\sigma}{2(2+\sigma)} } \left(\frac{\tau}{R}\right)^{\frac{\mu\sigma}{2+\sigma} } \sum_{i=0}^{j} \left[ \delta^{ \frac{\mu \sigma i}{ 2+\sigma } }  \left(  \left| (W)_{Q_{ \tau_{i} }} \right| 
+  \| F \|_{L^{\infty}(Q_{\tau})} \right) \right],
\end{aligned}\end{equation}
for any $j=0,1,2,\cdots$. 

We now claim that
\begin{equation}\label{PCC960}
\left| (W)_{Q_{ \tau_{j} }} \right|  \leq 10 \delta^{-n} \left[ \left( \mint_{Q_{\tau_{0}}} \left| W   \right|^{2} \, dx \right)^{\frac{1}{2}}  +  \| F \|_{L^{\infty}(Q_{\tau})} \right]
\quad (j=0,1,2,\cdots).
\end{equation}
We prove the claim \eqref{PCC960} by induction. First, we can easily check that \eqref{PCC960} holds when $j=0$. Next by an inductive assumption, suppose that \eqref{PCC960} holds for $0,1,\cdots,j$. From the inequality
\begin{equation*}\begin{aligned}
\left| (W)_{Q_{ \tau_{j+1} }} - (W)_{Q_{ \tau_{0} }} \right| 
& \leq 
\sum_{i=0}^{j} \mint_{Q_{\tau_{i+1}}} \left| W  - (W)_{Q_{\tau_{i}}} \right| \, dx \\
& \leq \delta^{-\frac{n}{2}}
\sum_{i=0}^{j+1} \left( \mint_{Q_{\tau_{i}}} \left| W  - (W)_{Q_{\tau_{i}}} \right|^{2} \, dx \right)^{\frac{1}{2}},
\end{aligned}\end{equation*}
and \eqref{PCC950}, we see that
\begin{equation*}\begin{aligned}
& \left| (W)_{Q_{ \tau_{j+1} }} - (W)_{Q_{ \tau_{0} }} \right| \\
& \quad \leq 4 \delta^{-\frac{n}{2}} \left( \mint_{Q_{\tau_{0}}} \left| W  - (W)_{Q_{\tau_{0}}} \right|^{2} \, dx \right)^{\frac{1}{2}} \\
& \qquad + 2 c_{1} \delta^{-n} \nu^{\frac{\sigma}{2(2+\sigma)} } \left( \frac{ \tau}{R} \right)^{\frac{\mu \sigma}{2+\sigma} }    \sum_{i=0}^{j} \left[ \delta^{ \frac{\mu \sigma i}{ 2+\sigma } }  \left(  \left| (W)_{Q_{ \tau_{i} }} \right| 
+  \| F \|_{L^{\infty}(Q_{\tau})} \right) \right].
\end{aligned}\end{equation*}
With the inductive assumption \eqref{PCC960}, we have that
\begin{equation*}
\sum_{i=0}^{j}  \delta^{ \frac{\mu \sigma i}{ 2+\sigma } }   \left| (W)_{Q_{ \tau_{i} }} \right| \leq 10\delta^{-n} \left(\sum_{i=0}^{j}  \delta^{ \frac{\mu \sigma i}{ 2+\sigma } } \right) \left[ \left( \mint_{Q_{\tau_{0}}} \left| W   \right|^{2} \, dx \right)^{\frac{1}{2}}  +  \| F \|_{L^{\infty}(Q_{\tau})} \right].
\end{equation*}
We also note that
\begin{equation}\label{PCC970}
\left( \mint_{Q_{\tau_{0}}} \left| W  - (W)_{Q_{\tau_{0}}} \right|^{2} \, dx \right)^{\frac{1}{2}} \leq 2 \left( \mint_{Q_{\tau_{0}}} \left| W  \right|^{2} \, dx \right)^{\frac{1}{2}}.
\end{equation}
Therefore, we have from \eqref{PCC930} that
\begin{equation*}\begin{aligned}
& \left| (W)_{Q_{ \tau_{j+1} }} - (W)_{Q_{ \tau_{0} }} \right| \\
& \qquad \leq 8 \delta^{-\frac{n}{2}} \left( \mint_{Q_{\tau_{0}}} \left| W  \right|^{2} \, dx \right)^{\frac{1}{2}}  + \delta^{-n}  \left[    \left( \mint_{Q_{\tau_{0}}} \left| W   \right|^{2} \, dx \right)^{\frac{1}{2}}+  \| F \|_{L^{\infty}(Q_{\tau})}  \right].
\end{aligned}\end{equation*}
We now remember that $\delta \in (0,1]$ to get the desired \eqref{PCC960}. 

\smallskip

Since $\tau_{i} = \delta^{i} \tau$, we find from \eqref{PCC930}, \eqref{PCC950}, \eqref{PCC960} and \eqref{PCC970} that
\begin{equation*}\begin{aligned}
  \mint_{Q_{\tau_{j}}} \left| W  - (W)_{Q_{\tau_{j}}} \right|^{2} \, dx 
\leq c \left[  \mint_{Q_{\tau_{0}}} \left| W   \right|^{2} \, dx    +  \| F \|^2_{L^{\infty}(Q_{\tau})} \right] 
\qquad ( j =0,1,2,\cdots),
\end{aligned}\end{equation*}
because $\delta \in (0,1)$ was chosen universal. So by \eqref{PCC960} and that $\tau_{i} =  \delta^{i} \tau$ $(i=0,1,\cdots)$,
\begin{equation*}\begin{aligned}
\mint_{Q_{\tau_{j}}} |W|^{2} \, dx
\leq c \left[  \mint_{Q_{\tau}} |W|^{2} \, dx +  \| F \|_{L^{\infty}(Q_{\tau})}^{2} \right] 
\qquad ( j =0,1,2,\cdots).
\end{aligned}\end{equation*}
Since $\delta \in (0,1)$ was a universal constant, we find that
\begin{equation*}\label{}
\mint_{Q_{\rho}} |W|^{2} \, dx \leq c \left[ \mint_{Q_{\tau}} |W|^{2} \, dx +  \| F \|_{L^{\infty}(Q_{\tau})}^{2} \right],
\end{equation*}
for any $ 0 < \rho \leq \tau$ and the lemma follows.
\end{proof}

\subsection{Linear coordinate transformation} 

To apply Lemma \ref{PCCS800} and Lemma \ref{PCCS900} for the general situation, we use a linear coordinate transformation. To use a coordinate transformation $\Psi$, we define `the derivative of the naturally induced flow respect to $\Psi$' in Definition \ref{new derivative of flow} corresponds to `the derivative of the naturally induced flow'  in Definition \ref{derivative of flow}. 

For this subsection, we employ the letter $c \geq 1$ to denote any constants that can be explicitly computed in terms $n$, $ \sup_{k \in K_{+}} \left\| D_{x'}\varphi_{k} \right\|_{L^{\infty}(Q_{r}')}$, $R^{\gamma} \sup_{k \in K_{+}} \left[ D_{x'}\varphi_{k} \right]_{C^{\gamma}(Q_{r}')}$ and the number of the element in the set $K$.

\begin{defn}\label{new derivative of flow}
Suppose that $\left( Q_{r}, \left\{ \varphi_{k} : k \in K_{+} \right\} \right)$ is a composite cube. For any $Q_{\tau}(z) \subset Q_{r}$ and $\zeta' \in \br^{n-1}$, let $\Psi : Q_{\tau}(z) \to \br^{n}$ be a linear coordinate transformation defined as 
\begin{equation*}
\Psi(x) = \left( x^{1} - z^{1} - \zeta' \cdot (x'-z') , x' - z' \right).
\end{equation*}
Let $y = \Psi(x)$ be the new coordinate system. Then we define the derivative of the naturally induced flow respect to $\Psi$ as follows.

Let the new graph $ \left \{ \left( \tilde{\varphi}_{k}(y'), y' \right) : y' \in Q_{\tau}' \right\}$ be the transformation of the graph $ \left \{ \left( \varphi_{k}(x'), x' \right) : x' \in Q_{\tau}'(z) \right\} $  under the coordinate transformation $\Psi$.  For any $ k \in K$, set $\tilde{T}_{k} : \Psi \left( Q_{\tau}(z) \right) \to [0,1]$ as
\begin{equation}\label{PCN330} 
\tilde{T}_{k}(y^{1},y') = \frac{y^{1} - \tilde{\varphi}_{k}( y')}{\tilde{\varphi}_{k+1}(y') - \tilde{\varphi}_{k}(y')} 
~ \text{ in } ~ \Psi \left( Q_{\tau}^{k}(z) \right).
\end{equation}
Then the following vector-valued function $\tilde{\pi} : \Psi \left( Q_{\tau}(z) \right) \to \br^{n}$ 
\begin{equation*}\label{}
\tilde{\pi} = \left( -1,\tilde{\pi}' \right) = \left( \tilde{\pi}_{1}, \tilde{\pi}_{2}, \cdots, \tilde{\pi}_{n} \right)
\end{equation*}
is called the derivative of the naturally induced flow respect to $\Psi$, where $\tilde{\pi}_{1}=-1$ and
\begin{equation}\begin{aligned}\label{PCN340}
\tilde{\pi}_{\alpha}(y) : = D_{y^{\alpha}}\tilde{\varphi}_{k+1}( y') \cdot \tilde{T}_{k}(y)
+ D_{y^{\alpha}}\tilde{\varphi}_{k}(y')  \cdot \left[ 1-\tilde{T}_{k}(y) \right]
~ \text{ in } ~ \Psi \left( Q_{\tau}^{k}(z) \right),
\end{aligned}\end{equation}
for any $k \in K$ and $\alpha \in \{2, \cdots,n \}$.
\end{defn}

To apply the coordinate transformation, we prove the following lemma. We remark that in the following lemma, $\Psi$ only depends on the point $z \in Q_{r}$ and is independent of the size $\tau$.

\begin{lem}\label{PCNS300}
Suppose that $\left( Q_{r}, \left\{ \varphi_{k} : k \in K_{+} \right\} \right)$ is a composite cube with the condition that \begin{equation}\label{PCN230}
|D_{x'}\varphi_{l}(x') - D_{x'}\varphi_{k}(x')|
\leq \kappa \left( \frac{ |\varphi_{l}(x') - \varphi_{k}(x')| }{ R } \right)^{2\mu} 
\quad  \left( x' \in Q_{r}', ~ k,l \in K_{+} \right),
\end{equation}
and
\begin{equation}\label{PCN235}
|D_{x'}\varphi_{k}(x') - D_{x'}\varphi_{k}(y')|
\leq \kappa  \left( \frac{ |x'-y'| }{ R } \right)^{2\mu}
\qqquad  \left( x',y' \in Q_{r}', ~ k \in K_{+} \right),
\end{equation}
for some constant $\kappa>0$. For $Q_{\tau}(z) \subset Q_{r}$ and the derivative of the naturally induced flow $\pi : Q_{r} \to \br^{n}$,  let $\Psi : Q_{\tau}(z) \to \br^{n}$ be a linear coordinate transformation defined as
\begin{equation*}
\Psi(x) = \left( x^{1} - z^{1} - \pi'(z) \cdot (x'-z') , x' - z' \right),
\end{equation*}
and  $\tilde{\pi} : \Psi \left( Q_{\tau}(z) \right) \to \br^{n}$  be the derivative of the naturally induced flow respect to $\Psi$. Then 
\begin{equation*}\label{}
|\tilde{\pi}'(0)|=0
\qquad \text{and} \qquad
|\tilde{\pi}'| \leq c \kappa \left(\frac{\rho}{R}\right)^{2\mu}
\quad \text{in} \quad Q_{\rho}
\end{equation*}
for any $Q_{\rho} \subset \Psi(Q_{\tau}(z))$.
\end{lem}

\begin{proof}
Let the new graph $ \left \{ \left( \tilde{\varphi}_{k}(y'), y' \right) : y' \in Q_{\tau}' \right\} $ be the transformation of the graph $ \left \{ \left( \varphi_{k}(x'), x' \right) : x' \in Q_{\tau}'(z) \right\} $ under the linear coordinate transformation  $\Psi$. Then one can check that
\begin{equation}\label{PCN360} 
\tilde{\varphi}_{k}(y') = \varphi_{k} \left( y'+z' \right) - z^{1} -  \pi'(z) \cdot y'.
\end{equation}
We claim that
\begin{equation}\label{PCN370}
\tilde{\pi}_{\alpha}(y) = \pi_{\alpha} \left( y^{1} +  z^{1} + \pi'(z) \cdot y', y' + z' \right) - \pi_{\alpha}(z)
\end{equation}
for any $y  \in \Psi \left( Q_{\tau}(z) \right)$ and $\alpha \in \{2, \cdots, n\} $.
Fix $y = (y^{1},y') \in \Psi \left( Q_{\tau}^{k}(z) \right)$ $(k \in K)$. In view of \eqref{PCN360}, we obtain that
\begin{equation*}\label{}
D_{y^{\alpha}} \tilde{\varphi}_{k}(y') 
= D_{x^{\alpha}} \varphi_{k}(y'+z') - \pi_{\alpha}(z)
\qquad 
\left( y' \in Q_{\tau}', \, k \in K_{+}, \, \alpha \in \{2, \cdots, n\}  \right).
\end{equation*}
So by \eqref{PCN330} and \eqref{PCN340},
\begin{equation*}\begin{aligned}\label{}
\tilde{\pi}_{\alpha}(y) 
& = \frac{[D_{x^{\alpha}}\varphi_{k+1}(y'+ z') - \pi_{\alpha}(z)][y^{1} - \tilde{\varphi}_{k}( y')]}{\tilde{\varphi}_{k+1}(y') - \tilde{\varphi}_{k}(y')} \\
& \qquad +  \frac{[D_{x^{\alpha}}\varphi_{k}(y'+z') - \pi_{\alpha}(z)][ \tilde{\varphi}_{k+1}( y') - y^{1} ]}{\tilde{\varphi}_{k+1}(y') - \tilde{\varphi}_{k}(y')} \\
& = \frac{ D_{x^{\alpha}}\varphi_{k+1}( y'+z') \, [y^{1}-\tilde{\varphi}_{k}( y')] }{\tilde{\varphi}_{k+1}(y') - \tilde{\varphi}_{k}(y')} 
+ \frac{  D_{x^{\alpha}}\varphi_{k}( y'+z') \, [\tilde{\varphi}_{k+1}( y' ) - y^{1}]}{\tilde{\varphi}_{k+1}(y') - \tilde{\varphi}_{k}(y')}  - \pi_{\alpha}(z)
\end{aligned}\end{equation*}
for any $\alpha \in \{ 2, \cdots, n \}$. It follows from \eqref{PCN360} that
\begin{equation}\begin{aligned}\label{PCN390}
\tilde{\pi}_{\alpha}(y) 
& = \frac{ D_{x^{\alpha}}\varphi_{k+1}(y'+z') \left[y^{1} + z^{1} + \pi'(z) \cdot y' - \varphi_{k}(y'+z') \right]   }{ \varphi_{k+1}(y'+z') -  \varphi_{k}(y'+z')} \\
& \quad + \frac{ D_{x^{\alpha}}\varphi_{k}(y'+z') \left[\varphi_{k+1}(y'+z')-y^{1}-z^{1} - \pi'(z) \cdot y' \right]  }{ \varphi_{k+1}(y'+z') -  \varphi_{k}(y'+z')} - \pi_{\alpha}(z)
\end{aligned}\end{equation}
for any $\alpha \in \{ 2, \cdots, n \}$. Since $y \in \Psi \left( Q_{\tau}^{k}(z) \right)$, we have that 
\begin{equation*}
\Psi^{-1}(y) = \left( y^{1} +  z^{1} + \pi'(z) \cdot y', y' + z' \right) \in  Q_{\tau}^{k}(z).
\end{equation*}
Then from \eqref{GS260}, we obtain that
\begin{equation*}\begin{aligned}\label{}
& \pi_{\alpha} \left( y^{1} +  z^{1} + \pi'(z) \cdot y', y' + z' \right) \\
& \quad = \frac{ D_{x^{\alpha}}\varphi_{k+1}(y'+z') \left[ y^{1} + z^{1} + \pi'(z) \cdot y' - \varphi_{k}(y'+z') \right]   }{ \varphi_{k+1}(y'+z') -  \varphi_{k}(y'+z')} \\
& \qquad + \frac{ D_{x^{\alpha}}\varphi_{k}(y'+z') \left[\varphi_{k+1}(y'+z')-y^{1}-z^{1} - \pi'(z) \cdot y' \right]  }{ \varphi_{k+1}(y'+z') -  \varphi_{k}(y'+z')},
\end{aligned}\end{equation*}
for any $\alpha \in \{ 2, \cdots, n \}$. So by \eqref{PCN390},
\begin{equation}\label{PCN394}
\tilde{\pi}_{\alpha}(y) = \pi_{\alpha} \left( y^{1} +  z^{1} + \pi'(z) \cdot y', y' + z' \right) - \pi_{\alpha}(z),
\end{equation}
for any $\alpha \in \{ 2, \cdots, n \}$. Thus 
\begin{equation}\label{PCN395}
\tilde{\pi}_{\alpha}(0)=0
\qquad 
(\alpha \in \{2, \cdots, n \}) .
\end{equation}
Since  $y = (y^{1},y') \in \Psi \left( Q_{\tau}^{k}(z) \right)$ $(k \in K)$ was arbitrary chosen, the claim \eqref{PCN370} holds. By comparing \eqref{PCN230} and \eqref{PCN235} with \eqref{GS400} and \eqref{GS405} respectively, we apply Lemma \ref{GSS600} to $\pi$ and \eqref{PCN370}. Then we obtain from \eqref{PCN394} that
\begin{equation}\begin{aligned}\label{PCN396}
\left| \tilde{\pi}'(y) \right| 
& = \left| \pi' \left( y^{1} +  z^{1} + \pi'(z) \cdot y', y' + z' \right) - \pi'(z) \right| \\
& \leq c \kappa \left( \frac{ \left| \left( y^{1} + \pi'(z) \cdot y', y' \right) \right| }{ R } \right)^{2\mu}
\end{aligned}\end{equation}
for any $y \in \Psi(Q_{\tau}(z))$. From \eqref{GS260}, we have that $|\pi'(z)| \leq 2 \sup_{k \in K_{+}} \left\| D_{x'}\varphi_{k} \right\|_{L^{\infty}(Q_{r}')} \leq c$. So by \eqref{PCN395} and \eqref{PCN396},
\begin{equation*}\label{}
|\tilde{\pi}'(0)|=0
\qquad \text{and} \qquad
|\tilde{\pi}'| 
\leq c \kappa \left( \frac{ \rho }{ R } \right)^{2\mu}
\quad \text{in} \quad Q_{\rho}
\end{equation*}
for any $Q_{\rho} \subset \Psi(Q_{\tau}(z))$.
\end{proof}

\subsection{Piecewise constant coefficients with no decay assumption} 

For the composite cube $\left( Q_{r}, \left\{ \varphi_{k} : k \in K_{+} \right\} \right)$, let $\pi$ be  the derivative of the naturally induced flow $\pi : Q_{r} \to \br^{n}$. Then
\begin{equation*}\label{}
\pi = (-1,\pi') = ( -1, \pi_{2}, \cdots, \pi_{n} )
\end{equation*}
where
\begin{equation*}\begin{aligned}\label{}
\pi_{\alpha}(x) = D_{\alpha}\varphi_{k+1}( x') \cdot T_{k}(x)
+ D_{\alpha}\varphi_{k}( x')  \cdot [1-T_{k}(x)]
\quad \text{ in } \quad Q_{r}^{k},
\end{aligned}\end{equation*}
for any $k \in K$ and  $\alpha \in \{ 2, \cdots, n \}$. We also assume that 
\begin{equation}\label{PCN130}
|D_{x'}\varphi_{l}(x') - D_{x'}\varphi_{k}(x')|
\leq \kappa \left( \frac{ |\varphi_{l}(x') - \varphi_{k}(x')| }{ R } \right)^{2\mu} 
\quad  \left( x' \in Q_{r}, ~ k,l \in K_{+} \right),
\end{equation}
and
\begin{equation}\label{PCN135}
|D_{x'}\varphi_{k}(x') - D_{x'}\varphi_{k}(y')|
\leq \kappa  \left( \frac{ |x'-y'| }{ R } \right)^{2\mu}
\qqquad  \left( x',y' \in Q_{r}', ~ k \in K_{+} \right),
\end{equation}
for some constant $\kappa>0$. 

For this subsection, we employ the letter $c \geq 1$ to denote any constants that can be explicitly computed in terms such as $n$, $N$, $\lambda$, $\Lambda$, $ \sup_{k \in K_{+}} \left\| D_{x'}\varphi_{k} \right\|_{L^{\infty}(Q_{r}')}$, $R^{\gamma} \sup_{k \in K_{+}} \left[ D_{x'}\varphi_{k} \right]_{C^{\gamma}(Q_{r}')}$ and the number of elements in the set $K$.

\sskip

As in Subsection \ref{PCCDA}, we handle the case when the coefficients are piecewise constant. For the constants $A_{ij,k}^{\alpha \beta}, F_{\alpha,k}^{i}$ $(1 \leq \alpha, \beta \leq n, \ 1 \leq i,j \leq N, ~ k \in K)$ satisfying that 
\begin{equation}\label{PCN240} 
\lambda |\xi|^{2} 
\leq A_{ij,k}^{\alpha \beta } \xi_{\alpha}^{i} \xi_{\beta}^{j} 
\qquad \text{and} \qquad
\left| A_{ij,k}^{\alpha \beta} \right| \leq \Lambda 
\qqquad
\left( \xi \in \br^{Nn} \right),
\end{equation}
we define $A_{ij}^{\alpha \beta}, F_{\alpha}^{i}$ $(1 \leq \alpha, \beta \leq n, \ 1 \leq i,j \leq N)$ as 
\begin{equation}\label{PCN250}
A_{ij}^{\alpha \beta}(x) 
= \sum_{k \in K} A_{ij,k}^{\alpha \beta} \chi_{Q_{r}^{k}}
\qquad \text{and} \qquad
F_{\alpha}^{i}(x) 
= \sum_{k \in K} F_{\alpha,k}^{i} \chi_{Q_{r}^{k}}
\quad \text{in} \quad Q_{r}.
\end{equation}
We remark that $A_{ij}^{\alpha \beta }$ and $F_{\alpha}^{i}$ $(1 \leq \alpha, \beta \leq n, \ 1 \leq i,j \leq N )$  are constant in each $Q_{r}^{k}$ $(k \in K)$. Then one can check from \eqref{PCN250} that 
\begin{equation}\label{PCN260} 
\lambda |\xi|^{2} 
\leq A_{ij}^{\alpha \beta }(x) \xi_{\alpha}^{i} \xi_{\beta}^{j} 
\quad \text{and} \quad
\big| A_{ij}^{\alpha \beta}(x) \big| \leq \Lambda 
\qquad \qquad
\left( x \in Q_{r}, ~ \xi \in \br^{Nn} \right).
\end{equation} 

Under the assumption \eqref{PCN130}, \eqref{PCN240} and \eqref{PCN250}, let $v$  be a weak solution of
\begin{equation}\label{PCN270}
D_{\alpha} \left[ A_{ij}^{\alpha \beta}(x) D_{\beta}^{j}v \right]  =  D_{\alpha} F_{\alpha}^{i} ~ \text{ in } ~ Q_{r}.
\end{equation}
Then we define  $V : Q_{r} \to \br^{Nn}$ as $V = \left( V^{1}, \cdots, V^{N} \right)^{T}$ where
\begin{equation}\label{PCN280}
V^{i} = \left( \sum_{1 \leq \alpha \leq n} \pi_{\alpha} \left[ \left( \sum_{1 \leq j \leq N} \sum_{ 1 \leq \beta \leq N } A_{ij}^{\alpha \beta} D_{x^{\beta}} v^{j} \right) - F_{\alpha}^{i} \right], D_{x'}v^{i} + \pi' \, D_{x^{1}} v^{i} \right) 
\end{equation}
for any $1 \leq i \leq N$. In this subsection, we obtain Lipschitz estimate of $V$ in $Q_{\frac{r}{2}}$ and the excess decay estimate of $V$.

\begin{lem}\label{PCNS500}
Suppose that $Q_{\tau}(z) \subset Q_{r}$. There exists a small universal constant $\varepsilon \in (0,1]$ such that if $\kappa \left(\frac{r}{R}\right)^{2\mu} \leq \varepsilon$ and $z$ is a Lebesgue point  of $V$ then
\begin{equation*}\label{}
|V(z)|^{2} \leq c \left[ \mint_{Q_{\tau}(z)} |V|^{2} \, dx +  \| F \|_{L^{\infty}(Q_{\tau}(z))}^{2} \right].
\end{equation*}
Here, $\varepsilon \in (0,1]$ depends only on $n$, $N$, $\lambda$, $\Lambda$, the number of the element in the set $K$, $ \sup_{k \in K_{+}} \left\| D_{x'}\varphi_{k} \right\|_{L^{\infty}(Q_{r}')}$ and $R^{\gamma} \sup_{k \in K_{+}} \left[ D_{x'}\varphi_{k} \right]_{C^{\gamma}(Q_{r}')}$.
\end{lem}

\begin{proof}
Define a linear coordinate transformation $\Psi : Q_{\tau}(z) \to \br^{n}$ with  $\Phi= \Psi^{-1}$ as
\begin{equation}\label{PCN520} 
\Psi(x) = \left( x^{1} - z^{1}  - \pi'(z) \cdot  (x'-z'), x' - z' \right)
\end{equation}
Let $y = \Psi(x)$  be the new coordinate system. Then for any $\alpha, \beta \in \{ 2, \cdots, n \}$,
\begin{equation}\label{PCN530} 
\frac{ \partial y^{1} }{ \partial x^{1} } = 1, 
\qquad
\frac{ \partial y^{1} }{ \partial x^{\beta} } =  - \pi_{\beta}(z),
\qquad
\frac{ \partial y^{\alpha} }{ \partial x^{1} } 
= 0,
\qquad
\frac{ \partial y^{\alpha} }{ \partial x^{\beta} } = \delta_{\alpha \beta},
\end{equation}
and
\begin{equation}\label{PCN540} 
\frac{ \partial x^{1} }{ \partial y^{1} } = 1, 
\qquad
\frac{ \partial x^{1} }{ \partial y^{\beta} } =  \pi_{\beta}(z),
\qquad
\frac{ \partial x^{\alpha} }{ \partial y^{1} } 
= 0,
\qquad
\frac{ \partial x^{\alpha} }{ \partial y^{\beta} } = \delta_{\alpha \beta}.
\end{equation}
From \eqref{GS260}, we have that $|\pi'(z)| \leq 2 \sup_{k \in K_{+}} \left\| D_{x'}\varphi_{k} \right\|_{L^{\infty}(Q_{r}')} $. So for a sufficiently small constant $\delta = \delta \left( n,\sup_{k \in K_{+}} \left\| D_{x'}\varphi_{k} \right\|_{L^{\infty}(Q_{r}')} \right) \in (0,1]$, we have from \eqref{PCN520} that 
\begin{equation}\label{PCN550}
Q_{\delta \rho} \subset \Psi(Q_{\rho}(z))
\qquad \text{and} \qquad
Q_{\delta \rho}(z) \subset \Phi(Q_{\rho})
\qqquad 
\big( \rho \in (0,\tau] \big).
\end{equation}
One can check from \eqref{PCN270} and \eqref{PCN550} that 
\begin{equation}\label{PCN560} 
D_{y^{\alpha}} \left[ \tilde{A}_{ij}^{\alpha \beta}(y) D_{y^{\beta}} w \right] = D_{y^{\alpha}} G_{\alpha}^{i} 
\qquad \text{in} \qquad Q_{\delta \tau} \subset \Psi \left( Q_{\tau}(z) \right)   ,
\end{equation}
where 
\begin{equation}\label{PCN570} 
\tilde{A}_{ij}^{\alpha \beta}  = \sum_{s,t} \frac{ \partial y^{\alpha} }{ \partial x^{s} } \frac{ \partial y^{\beta} }{ \partial x^{t} } \cdot A_{ij}^{st},
\ \ 
D_{y^{\beta}}w = \sum_{t} \frac{ \partial x^{t} }{ \partial y^{\beta} } \cdot D_{x^{t}}v
\ \ \text{and} \ \
G_{\alpha}^{i} = \sum_{ s } \frac{ \partial y^{\alpha} }{ \partial x^{s} } \cdot F_{s}^{i},
\end{equation}
for any  $1 \leq \alpha, \beta \leq n$ and $1 \leq i,j \leq N$. In view of \eqref{PCN250} and  \eqref{PCN260}, 
\begin{equation}\label{PCN580}
\tilde{A}_{ij}^{\alpha \beta}(y) 
= \sum_{k \in K} \sum_{1 \leq s,t \leq n } \frac{ \partial y^{\alpha} }{ \partial x^{s} } \frac{ \partial y^{\beta} }{ \partial x^{t} } \cdot A_{ij,k}^{st} \chi_{\Psi(Q_{\tau}^{k}(z))}
\quad \text{in} \quad \Psi(Q_{\tau}(z)),
\end{equation}
and
\begin{equation}\label{PCN585}
G_{\alpha}^{i}(y) 
= \sum_{k \in K} \sum_{1 \leq s \leq n} \frac{ \partial y^{\alpha} }{ \partial x^{s} } \cdot F_{s,k}^{i} \chi_{\Psi(Q_{\tau}^{k}(z))}
\quad \text{in} \quad \Psi(Q_{\tau}(z)).
\end{equation}
Since $\Psi$ is a linear coordinate transformation, $\frac{\partial y}{\partial x}$ and $\frac{\partial x}{\partial y}$ are constant as in \eqref{PCN530} and \eqref{PCN540}. So by \eqref{PCN580} and \eqref{PCN585}, one can check that $\tilde{A}_{ij}^{\alpha \beta }$ and $\tilde{F}_{\alpha}^{i}$ $(1 \leq \alpha, \beta \leq n, \, 1 \leq i,j \leq N )$ become constants in each $Q_{\delta \tau} \cap \Psi \left( Q_{\tau}^{k}(z) \right)$ for any $k \in K$. Also one can check from \eqref{PCN240}, \eqref{PCN530}, \eqref{PCN540} and \eqref{PCN550} that 
\begin{equation}\label{PCN590} 
c^{-1} |\xi|^{2} 
\leq \tilde{A}_{ij}^{\alpha \beta }(y) \xi_{\alpha}^{i} \xi_{\beta}^{j} 
\qquad \text{and} \qquad
\big| \tilde{A}_{ij}^{\alpha \beta}(y) \big| \leq c
\qquad
\left( y \in Q_{\delta \tau}, ~ \xi \in \br^{Nn} \right),
\end{equation} 
and
\begin{equation}\label{PCN595} 
\big\| G \big\|_{L^{\infty}(Q_{\delta \tau})}
\leq  \big\| F \big\|_{L^{\infty}(Q_{\tau}(z))} .
\end{equation}
Let $\tilde{\pi} : \Psi \left( Q_{\tau}(z) \right) \to \br^{n-1}$  be the derivative of the naturally induced flow respect to $\Psi$ defined in Definition \ref{new derivative of flow}.  In view of Lemma \ref{PCNS300} and \eqref{PCN550}, we obtain that
\begin{equation}\label{PCN600}
|\tilde{\pi}'(0)|=0
\qquad \text{and} \qquad
|\tilde{\pi}'| \leq c \kappa \left( \frac{ \rho }{ R } \right)^{2\mu}
\quad \text{in} \quad Q_{\rho} 
\left( \subset \Psi(Q_{\tau}(z)) \right)
\end{equation}
for any $0 < \rho \leq \delta \tau$, which corresponds to the condition \eqref{PCC230} used for Lemma \ref{PCCS900}. 

\medskip

Set $W : \Psi( Q_{r}(z) ) \to \br^{Nn}$ as $W = \left( W^{1} , \cdots, W^{N} \right)^{T}$ where
\begin{equation}\label{PCN610} 
W^{i} = \left( \left[ -\sum_{ 1 \leq j \leq N} \sum_{1 \leq \beta \leq n} \tilde{A}_{ij}^{1\beta} D_{y^{\beta}} w^{j} \right] + G_{1}^{i}, D_{y'} w^{i} \right)
\qquad \left( i = 1, \cdots, N \right).
\end{equation}
By comparing  \eqref{PCN590} and \eqref{PCN600} with  \eqref{PCC240} and \eqref{PCC230} respectively, we apply Lemma \ref{PCCS900} to the size $\delta^{-1}\rho$ and $\delta \tau$. So there exists a small universal constant $\varepsilon \in (0,1]$ such that if $\kappa \left(\frac{r}{R}\right)^{2\mu} \leq \varepsilon$ then
\begin{equation}\label{PCN620}
\mint_{Q_{\delta^{-1}\rho}(z)} |W|^{2} \, dx 
\leq c \left[ \mint_{Q_{\delta \tau}(z)} |W|^{2} \, dx +  \| G \|_{L^{\infty}(Q_{\delta \tau}(z))}^{2} \right],
\end{equation}
for any $0 < \rho \leq \delta^{2} \tau \leq \delta^{2} r$. Set $\bar{V} : Q_{\tau}(z) \to \br^{Nn}$ as $\bar{V} = \left( \bar{V}^{1} , \cdots, \bar{V}^{N} \right)^{T}$ where
\begin{equation*}\label{}
\bar{V}^{i} = \left( \sum_{1 \leq \alpha \leq N}  \pi_{\alpha}(z) \left[ \left( \sum_{1 \leq j \leq N} \sum_{1 \leq \beta \leq n} A_{ij}^{\alpha \beta} D_{x^{\beta}} v^{j} \right) - F_{\alpha}^{i} \right], D_{x'}v^{i} + \pi'(z) \, D_{x^{1}} v^{i} \right)
\end{equation*}
for any $1 \leq i \leq N$. 

By comparing \eqref{PCN130} and \eqref{PCN135} with \eqref{GS400} and \eqref{GS405} respectively, we have from Lemma \ref{GSS600} that $\pi = (-1,\pi') \in C^{2\mu}(Q_{r})$. So for $V$ in  \eqref{PCN280}, we have from Lemma \ref{HODS600} that
\begin{equation*}\label{}
\left| V - \bar{V} \right|
\leq c |\pi-\pi(z)|  \Big[ |D_{x}v| + |F| \Big]
\leq c\kappa \left( \frac{ r }{ R } \right)^{2\mu}  \Big[ |V| + |F| \Big]
\quad \text{ in } Q_{r}.
\end{equation*}
So for a sufficiently small universal constant $\varepsilon \in (0,1]$, if $\kappa \left(\frac{r}{R}\right)^{2\mu} \leq \varepsilon$ then
\begin{equation}\label{PCN636}
\left| \bar{V} \right| + |F|
\leq c \left[ \left| V \right| + |F| \right]
\leq c \left[ \left| \bar{V} \right| + |F| \right] 
\quad \text{ in } \quad Q_{r}.
\end{equation} 
Since $\pi_{1}=-1$, one can check from \eqref{PCN530} and \eqref{PCN540} that 
\begin{equation*}\label{} 
\bar{V}_{1}^{i} = \left( - \sum_{1 \leq j \leq N} \sum_{1 \leq s,t \leq n } \frac{ \partial y^{1} }{ \partial x^{s} } \cdot A_{ij}^{ s t}  D_{x^{t}}v^{j} \right)
+ \sum_{1 \leq s \leq n} \frac{ \partial y^{1} }{ \partial x^{s} } \cdot F_{s}^{i},
\end{equation*}
and
\begin{equation*}\label{} 
\bar{V}_{\alpha}^{i} = 
\sum_{1 \leq s \leq n } \frac{ \partial x^{s} }{ \partial y^{\beta} } \cdot D_{x^{s}}v^{i},
\end{equation*}
for any $i = 1, \cdots, N$ and $\beta = 2, \cdots, n$.  So by comparing $\bar{V}$ with $W$ in \eqref{PCN610}, we find from \eqref{PCN570}  that 
\begin{equation}\label{PCN650}
W\left( \Psi(x) \right) = \bar{V}(x)
\qquad \qquad \left( x \in Q_{\tau}(z) \right).
\end{equation}
Since $\det \left( \frac{ \partial y}{\partial x} \right) =1$, we find from \eqref{PCN550}, \eqref{PCN595} and \eqref{PCN650} that
\begin{equation*}\begin{aligned}
\mint_{Q_{\rho}(z)} |\bar{V}|^{2} \, dx
\leq c \mint_{\Psi(Q_{\rho}(z))} |W|^{2} \, dy
\leq c \mint_{Q_{\delta^{-1}\rho}} |W|^{2} \, dy,
\end{aligned}\end{equation*}
and
\begin{equation*}\begin{aligned}
\mint_{Q_{\delta \tau}} |W|^{2} \, dy +  \| G \|_{L^{\infty}(Q_{\delta \tau}(z))}^{2} 
& \leq c \left[ \mint_{\Phi(Q_{\delta \tau})} \left| \bar{V} \right|^{2} \, dx +  \| F \|_{L^{\infty}(Q_{\tau}(z))}^{2}  \right] \\
& \leq c \left[  \mint_{Q_{\tau}(z)} \left| \bar{V} \right|^{2} \, dx +  \| F \|_{L^{\infty}(Q_{\tau}(z))}^{2} \right],
\end{aligned}\end{equation*}
for any $0 < \rho \leq \delta^{2} \tau$. So by combining \eqref{PCN620} and the above two estimates,
\begin{equation}\label{PCN670}
\mint_{Q_{\rho}(z)} |\bar{V}|^{2} \, dx
\leq c \left[  \mint_{Q_{\tau}(z)} \left| \bar{V} \right|^{2} \, dx +  \| F \|_{L^{\infty}(Q_{\tau}(z))}^{2} \right].
\end{equation}
for any $0 < \rho \leq \delta^{2} \tau$. It follows from \eqref{PCN636} that
\begin{equation*}\label{}
\mint_{Q_{\rho}(z)} \left| V \right|^{2} \, dx
\leq c \left[  \mint_{Q_{\tau}(z)} \left| V \right|^{2} \, dx +  \| F \|_{L^{\infty}(Q_{\tau}(z))}^{2} \right],
\end{equation*}
for any $0 < \rho \leq \delta^{2} \tau$. Since $0 < \rho \leq \delta^{2} \tau$ was arbitrary chosen and  $z$ is Lebesgue point of $W$, the lemma follows.
\end{proof}

\begin{lem}\label{PCNS700}
For the small universal constant $\varepsilon \in (0,1]$ chosen in Lemma \ref{PCNS500}, if $\kappa \left( \frac{ r }{ R } \right)^{2\mu} \leq \varepsilon$ then
\begin{equation*}\label{}
\| V \|_{L^{\infty}\left( Q_{\frac{r}{2}} \right)}  \leq c \left[ \mint_{Q_{r}} |V|^{2} \, dx +  \| F \|_{L^{\infty}(Q_{r})}^{2} \right].
\end{equation*}
\end{lem}

\begin{proof}
For any Lebesgue point $z \in Q_{\frac{r}{2}}$ of $V$, we have from Lemma \ref{PCNS500} that
\begin{equation*}\label{}
|V(z)|^{2} \leq c \left[ \mint_{Q_{\frac{r}{2}}(z)} |V|^{2}  dx +  \| F \|_{L^{\infty}(Q_{\frac{r}{2}}(z))}^{2} \right]
\leq c \left[ \mint_{Q_{r}(z)} |V|^{2} \, dx +  \| F \|_{L^{\infty}(Q_{r}(z))}^{2} \right].
\end{equation*}
Since the Lebesgue point $z \in Q_{\frac{r}{2}}$ of $V$ was chosen arbitrary, the lemma follows.
\end{proof}

\begin{lem}\label{PCNS800}
For the small universal constant $\varepsilon \in (0,1]$ chosen in Lemma \ref{PCNS500}, if $Q_{\tau}(z) \subset Q_{r}$ and $\kappa \left( \frac{ r }{ R } \right)^{2\mu} \leq \varepsilon$ then
\begin{equation*}\begin{aligned}\label{} 
\mint_{ Q_{\rho}(z) } \left| V  - \left( V \right)_{ Q_{\rho}(z) } \right|^{2} \, dx
& \leq c  \left( \frac{\rho}{\tau} \right) 
\mint_{ Q_{\tau }(z) } \left| V - \left( V \right)_{ Q_{\tau}(z) } \right|^{2} \, dx \\
& \quad + c \kappa  \left( \frac{\tau}{R}\right)^{2\mu}  \left( \frac{ \tau }{\rho} \right)^{n} \left( \mint_{ Q_{\tau}(z) } | V |^{2} \, dx 
+   \| F \|_{L^{\infty}( Q_{\tau}(z) )}^{2}  \right) ,
\end{aligned}\end{equation*}
for any $0 < \rho \leq \tau$. 
\end{lem}

\begin{proof}
Assume that $0 < 2\rho \leq \delta^{2} \tau$, other-wise the lemma can be easily proved by that $\delta^{2}\tau < 2\rho \leq 2\tau$.

By Lemma \ref{PCNS700}, if $\kappa \left( \frac{ r }{ R } \right)^{2\mu} \leq \varepsilon$ then
\begin{equation*}\label{}
|V| \in L^{\infty}\left( Q_{\frac{r}{2}} \right)
\end{equation*}
Set $W : \Psi( Q_{r}(z) ) \to \br^{Nn}$ as in \eqref{PCN610} where $W = \left( W^{1} , \cdots, W^{N} \right)^{T}$ and
\begin{equation}\label{PCN910} 
W^{i} = \left( \left[ -\sum_{ 1 \leq j \leq N} \sum_{1 \leq \beta \leq n} \tilde{A}_{ij}^{1\beta} D_{y^{\beta}} w^{j} \right] + G_{1}^{i}, D_{y'} w^{i} \right)
\qquad \left( i = 1, \cdots, N \right).
\end{equation}
With \eqref{PCN280}, we obtain from \eqref{PCN570} and Lemma \ref{HODS600} that
\begin{equation*}\begin{aligned}
|W(y)| 
& \leq c \left[  |Dw(y)| + |G(y)| \right] \\
& \leq c \left[ \left| Dv \left( \Phi(y) \right) \right|  + \left| F \left( \Phi(y) \right) \right|  \right]
\leq c \left[ \left| V \left( \Phi(y) \right) \right| + \left| F \left( \Phi(y) \right) \right|  \right]
\end{aligned}\end{equation*}
for any $y \in \Psi \left( Q_{r} \right)$. So by Lemma \ref{PCNS700} and \eqref{PCN550},
\begin{equation}\label{PCN920}
W \in L^{\infty} \left( Q_{ \frac{ \delta r}{2}} \right).
\end{equation}
So by comparing  \eqref{PCN590} and \eqref{PCN600} with  \eqref{PCC240} and \eqref{PCC230} respectively, we apply Lemma \ref{PCCS800} with \eqref{PCN920} (take $\sigma = \infty$ in Lemma \ref{PCCS800}) to the size $\delta^{-1}\rho$ and $\frac{ \delta \tau }{2}$ instead of $\rho$ and $\theta$. Then
\begin{equation}\begin{aligned}\label{PCN930} 
& \mint_{Q_{\delta^{-1}\rho}} \left| W  - (W)_{Q_{\delta^{-1}\rho}} \right|^{2} \, dy \\
& \quad \leq c  \left( \frac{\rho}{\tau} \right)
\mint_{Q_{\delta \tau}} \left| W - (W)_{Q_{\delta \tau}} \right|^{2} \, dy \\
&\qquad + c \kappa \left( \frac{\tau}{R}\right)^{2\mu} \left( \frac{\tau}{\rho} \right)^{n} \left( \mint_{Q_{\delta \tau}} | W |^{2} \, dy 
+ \| G \|_{L^{\infty}(Q_{\delta \tau})}^{2} \right) 
\end{aligned}\end{equation}
for any $0 < \rho \leq \frac{ \delta^{2} \tau }{2} \leq \frac{\delta^{2} r}{4}$. By repeating the proof of \eqref{PCN650} in Lemma \ref{PCNS500}, for $V : Q_{\tau}(z) \to \br^{Nn}$ defined as $V = \left( V^{1} , \cdots, V^{N} \right)^{T}$ and
\begin{equation*}\label{}
\bar{V}^{i} = \left( \sum_{1 \leq \alpha \leq N}  \pi_{\alpha}(z) \left[ \left( \sum_{1 \leq j \leq N} \sum_{1 \leq \beta \leq n} A_{ij}^{\alpha \beta} D_{x^{\beta}} v^{j} \right) - F_{\alpha}^{i} \right], D_{x'}v^{i} + \pi'(z) \, D_{x^{1}} v^{i} \right)
\end{equation*}
where $1 \leq i \leq N$, we have that  
\begin{equation}\label{PCN950}
W\left( \Psi(x) \right) = \bar{V}(x)
\qquad \qquad \left( x \in Q_{\tau}(z)  \right).
\end{equation}
Also by repeating the proof of \eqref{PCN595} and \eqref{PCN636} in Lemma \ref{PCNS500}, we obtain that
\begin{equation}\label{PCN955} 
\big\| G \big\|_{L^{\infty}(Q_{\delta \tau})}
\leq  \big\| F \big\|_{L^{\infty}(Q_{\tau}(z))},
\end{equation}
and
\begin{equation}\label{PCN960}
\left| \bar{V} \right| +|F|
\leq c \left[ \left| V \right| + |F| \right]
\leq c \left[ \left| \bar{V} \right| + |F| \right] 
\quad \text{ in } Q_{r}.
\end{equation}
Since $\det \left( \frac{ \partial y}{\partial x} \right) =1$ and $\delta \in (0,1]$ is universal, one can check from \eqref{PCN950} that 
\begin{equation*}\begin{aligned}
\mint_{Q_{\rho}(z) } \left| \bar{V} - \left( \bar{V} \right)_{Q_{\rho}(z) } \right|^{2}  \, dx
& \leq c\mint_{\Psi(Q_{\rho}(z) )} \left| W - \left( W \right)_{\Psi(Q_{\rho}(z) )} \right|^{2}  \, dy \\
& \leq 
c \mint_{Q_{\delta^{-1}\rho}} \left| W - \left( W \right)_{Q_{\delta^{-1}\rho}} \right|^{2}  \, dy
\end{aligned}\end{equation*}
and
\begin{equation*}\begin{aligned}
\mint_{Q_{\delta \tau}} \left| W - \left( W \right)_{Q_{\delta \tau}}  \right|^{2} \, dy
& \leq c \mint_{\Psi( Q_{\tau}(z)  )} \left| W - \left( W \right)_{ \Psi(Q_{\tau}(z) ) } \right|^{2} \, dy \\
& \leq c \mint_{ Q_{\tau}(z)  } \left| \bar{V} - \left( \bar{V} \right)_{ Q_{\tau}(z)  } \right|^{2} \, dx.
\end{aligned}\end{equation*}
So it follows from \eqref{PCN930}, \eqref{PCN955} and \eqref{PCN960} that 
\begin{equation*}\begin{aligned}\label{} 
\mint_{ Q_{\rho}(z) } \left| V  - \left( V \right)_{ Q_{\rho}(z) } \right|^{2} \, dx &  \leq c  \left( \frac{\rho}{\tau} \right)
\mint_{ Q_{\tau }(z) } \left| V - \left( V \right)_{ Q_{\tau}(z) } \right|^{2} \, dx \\
& \quad+ c\kappa  \left( \frac{\tau}{R}\right)^{2\mu}  \left( \frac{ \tau }{\rho} \right)^{n} \left( \mint_{ Q_{\tau}(z) } | V |^{2} \, dx 
+ \| F \|_{L^{\infty}( Q_{\tau}(z) )}^{2}  \right),
\end{aligned}\end{equation*}
for any $0 < \rho \leq \frac{ \delta^{2} \tau }{2} \leq \frac{\delta^{2} r}{4}$. One can easily extend this estimate to the case when $ \frac{r}{2} < \tau \leq r$.
\end{proof}

\subsection{Piecewise H\"{o}lder continuous coefficients with no decay assumptions}

In this subsection, we obtain the corresponding result to Lemma \ref{PCNS800} for general piecewise H\"{o}lder continuous coefficients.

For the composite cube $\left( Q_{r}, \left\{ \varphi_{k} : k \in K_{+} \right\} \right)$, let $\pi$ be  the derivative of the naturally induced flow $\pi : Q_{r} \to \br^{n}$. Then
\begin{equation*}\label{}
\pi = (-1,\pi') = ( -1, \pi_{2}, \cdots, \pi_{n} )
\end{equation*}
where
\begin{equation*}\begin{aligned}\label{}
\pi_{\alpha}(x) = D_{\alpha}\varphi_{k+1}( x') \cdot T_{k}(x)
+ D_{\alpha}\varphi_{k}( x')  \cdot [1-T_{k}(x)]
\quad \text{ in } \quad Q_{r}^{k},
\end{aligned}\end{equation*}
for any $k \in K$ and  $\alpha \in \{ 2, \cdots, n \}$. Assume \eqref{ell1}, \eqref{ell2} and that $A_{ij}^{\alpha \beta}, F_{\alpha}^{i} \in C^{\mu} \left( Q_{r}^{k} \right)$ for any $k \in K$. Also we further assume that 
\begin{equation}\label{PHN230}
|D_{x'}\varphi_{l}(x') - D_{x'}\varphi_{k}(x')|
\leq \kappa \left( \frac{ |\varphi_{l}(x') - \varphi_{k}(x')| }{ R } \right)^{2\mu} 
\quad    \left( x' \in Q_{r}, ~ k,l \in K_{+} \right),
\end{equation}
\begin{equation}\label{PHN235}
|D_{x'}\varphi_{k}(x') - D_{x'}\varphi_{k}(y')|
\leq \kappa  \left( \frac{ |x'-y'| }{ R } \right)^{2\mu}
\qqquad  \left( x',y' \in Q_{r}', ~ k \in K_{+} \right),
\end{equation}
and
\begin{equation}\label{PHN240}
R^\mu \left[ A_{ij}^{\alpha \beta} \right]_{C^{\mu}(Q^k_{r})}
\leq \kappa^{\frac{1}{2} }
\qquad (1 \leq \alpha, \beta \leq n, \ 1 \leq i,j \leq N, \ k \in K)
\end{equation}
for some constant $\kappa>0$.

For this subsection, we employ the letter $c \geq 1$ to denote any constants that can be explicitly computed in terms such as $n$, $N$, $\lambda$, $\Lambda$, $ \sup_{k \in K_{+}} \left\| D_{x'}\varphi_{k} \right\|_{L^{\infty}(Q_{r}')}$, $R^{\gamma} \sup_{k \in K_{+}} \left[ D_{x'}\varphi_{k} \right]_{C^{\gamma}(Q_{r}')}$  and the number of the element in the set $K$.

\smallskip

Let $u$ be a weak solution of
\begin{equation}\label{PHN300}
D_{\alpha} \left[ A_{ij}^{\alpha \beta} D_{\beta}u^{j} \right]  =  D_{\alpha} F_{\alpha}^{i} ~ \text{ in } ~ Q_{r}.
\end{equation}
Then we obtain the following lemma.

\begin{lem}\label{PHNS400}
For the small universal constant $\varepsilon \in (0,1]$ chosen in Lemma \ref{PCNS500}, if $\kappa \left( \frac{ r }{ R } \right)^{2\mu} \leq \varepsilon$ then for $U : Q_{r} \to \br^{Nn}$ defined as $U = \left( U^{1}, \cdots, U^{N} \right)^{T}$ and
\begin{equation}\label{PHN410} 
U^{i} = \left( \sum_{1 \leq \alpha \leq n } \pi_{\alpha} \left[ \left( \sum_{1 \leq j \leq N} \sum_{1 \leq \beta \leq n } A_{ij}^{\alpha \beta} D_{\beta} u^{j} \right) - F_{\alpha}^{i} \right], D_{x'}u^{i} + \pi' \, D_{1} u^{i} \right), 
\end{equation}
where $1 \leq i \leq N$, we have that
\begin{equation}\begin{aligned}\label{PHN415}
& \mint_{Q_{\rho} } \left| U  - (U)_{Q_{\rho} } \right|^{2} \, dx \\
& \quad \leq c  \left( \frac{\rho}{\tau} \right)
\mint_{Q_{\tau} } \left| U - (U)_{Q_{\tau} } \right|^{2} \, dx \\
& \qquad + c \left( \frac{ \tau }{ R } \right)^{ 2\mu } \left( \frac{\tau}{\rho} \right)^{n} 
\left( \kappa \mint_{Q_{\tau} } |U|^{2} \, dx 
+ \kappa \| F \|_{L^{\infty}(Q_{\tau})}^{2} + \sup_{ k \in K } [ F ]_{C^{\mu}(Q_{\tau}^{k})}^{2}    \right),
\end{aligned}\end{equation}
for any $0 < \rho \leq \tau \leq r$.
\end{lem}

\begin{proof}
Fix $ 0 < \rho \leq \tau \leq r$. For any $k\in K$ with $Q_{\tau}^{k} \not = \emptyset$, we choose the points $z_{k} \in Q_{\tau}^{k} $. On the other-hand, for any $k \in K$ with $Q_{\tau}^{k} = \emptyset$, let $z_{k}=0$ for the simplicity of notation. We remark that we will focus on the set $Q_{\tau}$ and we don't need to consider the set $Q_{\tau}^{k} $ when $Q_{\tau}^{k}  = \emptyset$. Set
\begin{equation}\label{PHN440}
\bar{A}_{ij}^{\alpha \beta}(x) 
= \sum_{k \in K} A_{ij}^{\alpha \beta}(z_{k}) \chi_{Q_{\tau}^{k} }
\quad \text{and} \quad
\bar{F}_{\alpha}^{i}(x) 
= \sum_{k \in K} F_{\alpha,k}^{i}(z_{k}) \chi_{Q_{\tau}^{k} }
\qquad \text{ in } ~ Q_{\tau} .
\end{equation}
Since $A_{ij}^{\alpha \beta}, F_{\alpha}^{i} \in C^{\mu} \left( Q_{\tau}^{k}  \right)$ for any $k \in K$, we find from \eqref{PHN240} and \eqref{PHN440} that
\begin{equation}\label{PHN450}
\left\|A_{ij}^{\alpha \beta} - \bar{A}_{ij}^{\alpha \beta} \right\|_{L^{\infty}(Q_{\tau} )}
\leq \kappa^{\frac{1}{2}} \left(\frac{\tau}{R}\right)^{\mu} 
\end{equation}
and
\begin{equation}\label{PHN455}
\left\|F - \bar{F} \right\|_{L^{\infty}(Q_{\tau} )}
\leq \tau^{\mu} \sup_{k \in K} [F]_{C^{\mu}(Q_{\tau} )}.
\end{equation}
Also we find from \eqref{PHN440} that 
\begin{equation}\label{PHN460}
\left\| \bar{A}_{ij}^{\alpha \beta} \right\|_{L^{\infty}(Q_{\tau} )}
\leq c \left\| A_{ij}^{\alpha \beta} \right\|_{L^{\infty}(Q_{\tau} )}
\qquad \text{and} \qquad
\left\| \bar{F} \right\|_{L^{\infty}(Q_{\tau} )}
\leq c  \| F \|_{L^{\infty}(Q_{\tau} )}.
\end{equation}
Let $v$ be the weak solution of 
\begin{equation}\label{PHN470}\left\{\begin{array}{rcll}
D_{\alpha} \left[ \bar{A}_{ij}^{\alpha \beta} D_{\beta}v^{j} \right] & = & D_{\alpha} \bar{F}_{\alpha}^{i} & \text{ in } ~ Q_{\tau} ,\\
v & = & u  & \text{ in } ~ \partial Q_{\tau} .
\end{array}\right.\end{equation}
We test \eqref{PHN300} and \eqref{PHN470} by $u^{i}-v^{i}$ $(i = 1, \cdots, N)$ to find that 
\begin{equation*}\begin{aligned}\label{}
& \mint_{Q_{\tau} } \left \langle \bar{A}_{ij}^{\alpha \beta} \left[ D_{\beta} u^{j} - D_{\beta} v^{j} \right] , D_{\alpha} u^{i} - D_{\alpha} v^{i} \right \rangle \, dx \\
& \quad = \mint_{Q_{\tau} } \left \langle \left[ \bar{A}_{ij}^{\alpha \beta} - A_{ij}^{\alpha \beta}  \right]  D_{\beta} u^{j} , D_{\alpha} u^{i} - D_{\alpha} v^{i} \right \rangle + \left \langle F_{\alpha}^{i} - \bar{F}_{\alpha}^{i} , D_{\alpha} u^{i} - D_{\alpha} v^{i} \right \rangle \, dx.
\end{aligned}\end{equation*}
So by using Young's inequality, we obtain from \eqref{PHN450} and \eqref{PHN455} that
\begin{equation}\label{PHN490}
\mint_{Q_{\tau} } |Du - Dv|^{2} \, dx 
\leq c \left( \frac{ \tau }{ R } \right)^{ 2\mu } \left[ \kappa \mint_{Q_{\tau} } |Du|^{2} \, dx
 + R^{2\mu} \sup_{ k \in K } [ F ]_{C^{\mu}(Q_{\tau}^{k})}^{2}    \right].
\end{equation}
To use Lemma \ref{PCNS800}, we set $V : Q_{\tau} \to \br^{Nn}$ as $V = \left( V^{1}, \cdots, V^{N} \right)^{T}$ where
\begin{equation}\label{PHN500}
V^{i} = \left( \sum_{ 1 \leq \alpha \leq n} \pi_{\alpha} \left[ \left( \sum_{1 \leq j \leq N} \sum_{ 1 \leq \beta \leq n} \bar{A}_{ij}^{\alpha \beta} D_{\beta} v^{i} \right) - \bar{F}_{\alpha}^{i} \right], D_{x'}v^{i} + \pi' \, D_{1} v^{i} \right) 
\end{equation}
for any $1 \leq i \leq N$. By the triangle inequality, \eqref{PHN450}, \eqref{PHN455} and \eqref{PHN490}, we compare $U$ in \eqref{PHN410} and $V$ in \eqref{PHN500} as follows:
\begin{equation*}\begin{aligned}\label{}
\mint_{Q_{\tau} } | U - V |^{2}  dx 
& \leq c  \mint_{Q_{\tau} } \left| \left[ A_{ij}^{\alpha \beta} - \bar{A}_{ij}^{\alpha \beta} \right] D_{\beta}u^{j }\right|^{2} \, dx \\
& \quad + c  \mint_{Q_{\tau} } \left|  \left( \bar{A}_{ij}^{\alpha \beta} \left[ D_{\beta}u^{j} - D_{\beta}v^{j} \right] - \left[  F_{1}^{i} - \bar{F}_{1}^{i} \right] , D_{x'}u^{i} - D_{x'}v^{i}  \right) \right|^{2}  dx \\
& \leq c \left[ \mint_{Q_{\tau} } |Du - Dv|^{2}  + \kappa \tau^{2\mu} |Du|^{2} \, dx \right] \\
& \leq c \left( \frac{ \tau }{ R } \right)^{ 2\mu } \left[ \kappa \mint_{Q_{\tau} } |Du|^{2} \, dx
 + R^{2\mu} \sup_{ k \in K } [ F ]_{C^{\mu}(Q_{\tau}^{k})}^{2}    \right].
\end{aligned}\end{equation*}
So it follows from Lemma \ref{HODS600} that
\begin{equation*}\label{}
\mint_{Q_{\tau} } | U - V |^{2} \, dx 
\leq c \left( \frac{ \tau }{ R } \right)^{ 2\mu } \left( \kappa \mint_{Q_{\tau} } |U|^{2} \, dx 
+ \kappa \| F \|_{L^{\infty}(Q_{\tau})}^{2} + R^{2\mu} \sup_{ k \in K } [ F ]_{C^{\mu}(Q_{\tau}^{k})}^{2}     \right).
\end{equation*}
Since $\bar{A}_{ij}^{\alpha \beta} $ $(1 \leq \alpha, \beta \leq n, \ 1 \leq i,j \leq N, ~ k \in K)$ are piecewise constant coefficients and $\kappa \left( \frac{ r }{ R } \right)^{2\mu} \leq \varepsilon$, we have from Lemma \ref{PCNS800} that
\begin{equation*}\begin{aligned}\label{}
\mint_{Q_{\rho} } \left| V  - (V)_{Q_{\rho} } \right|^{2} \, dx 
& \leq c  \left( \frac{\rho}{\tau} \right)
\mint_{Q_{\tau} } \left| V - (V)_{Q_{\tau} } \right|^{2} \, dx \\
& \quad + c \left( \frac{ \tau }{ R } \right)^{ 2\mu } \left( \frac{\tau}{\rho} \right)^{n} 
\left( \kappa \mint_{Q_{\tau} } |V|^{2} \, dx 
+ \kappa \| F \|_{L^{\infty}(Q_{\tau})}^{2}   \right).
\end{aligned}\end{equation*}
By combining the above two estimates, we have from Lemma \ref{HODS600} that
\begin{equation*}\begin{aligned}\label{}
& \mint_{Q_{\rho} } \left| U  - (U)_{Q_{\rho} } \right|^{2} \, dx \\
& \quad \leq c  \left( \frac{\rho}{\tau} \right)
\mint_{Q_{\tau} } \left| U - (U)_{Q_{\tau} } \right|^{2} \, dx \\
& \qquad + c \left( \frac{ \tau }{ R } \right)^{ 2\mu } \left( \frac{\tau}{\rho} \right)^{n} 
\left( \kappa \mint_{Q_{\tau} } |U|^{2} \, dx 
+ \kappa \| F \|_{L^{\infty}(Q_{\tau})}^{2} + R^{2\mu} \sup_{ k \in K } [ F ]_{C^{\mu}(Q_{\tau}^{k})}^{2}    \right).
\end{aligned}\end{equation*}
Since $0 < \rho \leq \tau \leq r$ was arbitrary chosen, the lemma follows.
\end{proof}

By repeating the proof of Lemma \ref{PCCS900}, we obtain the following lemma.

\begin{lem}\label{PHNS600}
For the small universal constant $\varepsilon \in (0,1]$ chosen in Lemma \ref{PCNS500}, if $\kappa \left( \frac{ r }{ R } \right)^{2\mu} \leq \varepsilon$ then for $U : Q_{r} \to \br^{Nn}$ defined as $U = \left( U^{1}, \cdots, U^{N} \right)^{T}$ and
\begin{equation}\label{PHN610} 
U^{i} = \left( \sum_{1 \leq \alpha \leq n } \pi_{\alpha} \left[ \left( \sum_{1 \leq j \leq N} \sum_{1 \leq \beta \leq n } A_{ij}^{\alpha \beta} D_{\beta} u^{j} \right) - F_{\alpha}^{i} \right], D_{x'}u^{i} + \pi' \, D_{1} u^{i} \right), 
\end{equation}
where $1 \leq i \leq N$, we have that
\begin{equation*}\begin{aligned}
\mint_{Q_{\rho} } \left| U  \right|^{2} \, dx
\leq c \left( \mint_{Q_{r} } \left| U  \right|^{2} \, dx
+   \| F \|_{L^{\infty}(Q_{r} )}^{2} 
+ R^{2\mu} \sup_{ k \in K } [ F ]_{C^{\mu}(Q_{r}^{k})}^{2}     \right),
\end{aligned}\end{equation*}
for any $0 < \rho < r$.
\end{lem}

\begin{proof}
The proof is same to that of Lemma \ref{PCCS900}. Instead of \eqref{PCC915} and $W$, use \eqref{PHN415} and $U$ respectively.
\end{proof}

We extend the technical result \cite[Lemma 3.4]{HQLF1} to the following lemma.

\begin{lem}\label{PHNS700}
Let $\phi(t)$ be a nonnegative  function on $[0,R]$. If
\begin{equation*}
\phi(\rho) \leq  A\left( \frac{\rho}{\tau} \right)^{\alpha} \phi(\tau) + B  \tau^{\beta} \left( \frac{\tau}{\rho} \right)^{n}
\end{equation*}
holds for any $0 < \rho \leq \tau \leq r$ with $A,B,\alpha,\beta,n$ nonnegative constants and $\beta < \alpha$. Then for any $\gamma \in [\beta,\alpha)$, there exist  a positive constant $c$ depending on $n,A,\alpha,\beta,\gamma$ such that 
\begin{equation*}
\phi(\rho) \leq  c \left[ \left( \frac{\rho}{\tau} \right)^{\gamma} \phi(\tau) + B  \rho^{\beta} \right],
\end{equation*}
for all $0<\rho \leq \tau \leq r$.
\end{lem}

\begin{proof}
Since $\gamma \in [\beta,\alpha)$, choose $\delta = \delta(A,\alpha,\gamma) \in (0,1)$ so that $2A\delta^{\alpha} \leq \delta^{\gamma}$. Fix $0 < \rho \leq \delta \tau  \leq \delta r$. Other-wise the lemma holds from that $\delta \tau \leq  \rho \leq \tau$ and that $\delta = \delta(A,\alpha,\gamma) \in (0,1)$. Let $\tau_{i} = \delta^{i} \tau$. Then
\begin{equation*}\label{} 
\phi( \tau_{i+1}) \leq \frac{ \delta^{\gamma} }{2} \cdot \phi( \tau_{i} ) + B \delta^{-n} \tau_{i}^{\beta},
\end{equation*}
for any $i = 0,1,\cdots$. Since $\delta \in (0,1)$ and $\gamma \in [\beta,\alpha)$, we find that 
\begin{equation*}\begin{aligned}
\phi(\tau_{i+1}) 
& \leq \left( \frac{\delta^{\gamma}}{2} \right)^{i+1} \phi(\tau) + B \delta^{-n}  \left[ \tau_{0}^{\beta} \left( \frac{\delta^{\gamma}}{2} \right)^{i}  + \tau_{1}^{\beta} \left( \frac{\delta^{\gamma}}{2} \right)^{i-1}  + \cdots + \tau_{i}^{\beta}   \right] \\
& \leq \left( \frac{\delta^{\gamma}}{2} \right)^{i+1} \phi(\tau) + B \delta^{-n}  \left[ \tau_{0}^{\beta} \left( \frac{\delta^{\beta}}{2} \right)^{i}  + \tau_{1}^{\beta} \left( \frac{\delta^{\beta}}{2} \right)^{i-1}  + \cdots + \tau_{i}^{\beta}   \right].
\end{aligned}\end{equation*}
Since $\tau_{0}^{\beta} \delta^{\beta i} = \tau_{1}^{\beta} \delta^{ \beta(i-1)} = \cdots = \tau_{i}^\beta$, we obtain that
\begin{equation*}\label{} 
\phi(\tau_{i+1}) 
\leq \left( \frac{\delta^{\gamma}}{2} \right)^{i+1} \phi(\tau) 
+  2 B \delta^{-n} \tau_{i}^{\beta},
\end{equation*}
for any $i = 0,1,\cdots$. Choose $i \in \{ 0,1,2,\cdots\}$ satisfying  that $\tau_{i+2} \leq \rho \leq \tau_{i+1}$. Then
\begin{equation*}\begin{aligned}
\phi(\rho) 
& \leq A \left( \frac{ \tau_{i+1} }{\rho} \right)^{n} \phi(\tau_{i+1}) + B \delta^{-n}\tau_{i+1}^\beta \\
& \leq A \delta^{-n} \left( \frac{\delta^{\gamma}}{2} \right)^{i} \phi(\tau) 
+  (2 A+1) B \delta^{-2n} \tau_{i}^{\beta} \\
& \leq A \delta^{-n - \gamma} \left( \frac{ \rho }{\tau} \right)^{\gamma} \phi(\tau) 
+  (2 A+1) B \delta^{-2n} \tau_{i}^{\beta},
\end{aligned}\end{equation*}
and the lemma follows from that $\tau_{i}^{\beta} \leq \delta^{-2\beta} \tau_{i+2}^{\beta} 
\leq \delta^{-2\beta } \rho^{\beta}$.
\end{proof}

\begin{lem}\label{PHNS800}
For the small universal constant $\varepsilon \in (0,1]$ chosen in Lemma \ref{PCNS500}, if $\kappa \left( \frac{ r }{ R } \right)^{2\mu} \leq \varepsilon$ then for $U : Q_{r} \to \br^{Nn}$ defined as $U = \left( U^{1}, \cdots, U^{N} \right)^{T}$ and
\begin{equation*}\label{} 
U^{i} = \left( \sum_{1 \leq \alpha \leq n } \pi_{\alpha} \left[ \left( \sum_{1 \leq j \leq N} \sum_{1 \leq \beta \leq n } A_{ij}^{\alpha \beta} D_{\beta} u^{j} \right) - F_{\alpha}^{i} \right], D_{x'}u^{i} + \pi' \, D_{1} u^{i} \right), 
\end{equation*}
where $1 \leq i \leq N$, we have that
\begin{equation*}\begin{aligned}
& \frac{1}{\rho^{2\mu}}  \mint_{Q_{\rho} } \left| U  - (U)_{Q_{\rho} } \right|^{2} \, dx \\
& \quad \leq \frac{c}{ r^{ 2\mu }}
\mint_{Q_{r} } \left| U - (U)_{Q_{r} } \right|^{2} \, dx \\
& \qquad + \frac{c}{ R^{2\mu} } \left( \kappa \mint_{Q_{r} } |U|^{2} \, dx 
+ \kappa \| F \|_{L^{\infty}(Q_{r})}^{2} + R^{2\mu} \sup_{ k \in K } [ F ]_{C^{\mu}(Q_{\tau}^{k})}^{2}    \right),
\end{aligned}\end{equation*}
for any $0 < \rho \leq r$.
\end{lem}

\begin{proof}
In view of Lemma \ref{PHNS600}, we have that
\begin{equation}\label{PHN820}
\mint_{Q_{\tau} } \left| U  \right|^{2} \, dx
\leq \left( \mint_{Q_{r} } \left| U  \right|^{2} \, dx
+   \| F \|_{L^{\infty}(Q_{r} )}^{2} 
+ R^{2\mu} \sup_{ k \in K } [ F ]_{C^{\mu}(Q_{r}^{k})}^{2} \right),
\end{equation}
for any $0 < \tau \leq r$. In view of Lemma \ref{PHNS400}, we have that 
\begin{equation*}\begin{aligned}
& \mint_{Q_{\rho} } \left| U  - (U)_{Q_{\rho} } \right|^{2} \, dx \\
& \quad \leq c \left( \frac{\rho}{\tau} \right)
\mint_{Q_{\tau} } \left| U - (U)_{Q_{\tau} } \right|^{2} \, dx \\
& \qquad + c \left( \frac{ \tau }{ R } \right)^{ 2\mu }  \left( \frac{\tau}{\rho} \right)^{n} 
\left( \kappa \mint_{Q_{\tau} } |U|^{2} \, dx 
+ \kappa \| F \|_{L^{\infty}(Q_{\tau})}^{2} + R^{2\mu} \sup_{ k \in K } [ F ]_{C^{\gamma}(Q_{\tau}^{k}    )}^{2}  \right),
\end{aligned}\end{equation*}
for any $0 < \rho \leq \tau \leq r$. So it follows from \eqref{PHN820} that
\begin{equation*}\begin{aligned}
& \mint_{Q_{\rho} } \left| U  - (U)_{Q_{\rho} } \right|^{2} \, dx \\
& \quad \leq c  \left( \frac{\rho}{\tau} \right)
\mint_{Q_{\tau} } \left| U - (U)_{Q_{\tau} } \right|^{2} \, dx \\
& \qquad + c \left( \frac{ \tau }{ R } \right)^{ 2\mu } \left( \frac{\tau}{\rho} \right)^{n}  \left( \kappa \mint_{Q_{r} } |U|^{2} \, dx 
+ \kappa \| F \|_{L^{\infty}(Q_{r})}^{2} + R^{2\mu}  \sup_{ k \in K } [ F ]_{C^{\gamma}(Q_{r}^{k}    )}^{2}  \right),
\end{aligned}\end{equation*}
for any $0 < \rho \leq \tau \leq r$. So by taking $\alpha = 1$, $\gamma=\beta=2\mu \in (0,1)$ in Lemma \ref{PHNS700},
\begin{equation*}\begin{aligned}
& \mint_{Q_{\rho} } \left| U  - (U)_{Q_{\rho} } \right|^{2} \, dx \\
& \quad \leq c  \left( \frac{\rho}{\tau} \right)^{2\mu}
\mint_{Q_{\tau} } \left| U - (U)_{Q_{\tau} } \right|^{2} \, dx \\
& \qquad + c \left( \frac{ \rho }{ R } \right)^{ 2\mu }  \left( \kappa \mint_{Q_{r} } |U|^{2} \, dx 
+ \kappa \| F \|_{L^{\infty}(Q_{r})}^{2} + R^{2\mu} \sup_{ k \in K } [ F ]_{C^{\gamma}(Q_{r}^{k}    )}^{2}  \right),
\end{aligned}\end{equation*}
for any $0 < \rho \leq r$. So the lemma follows.
\end{proof}

\section{Proof of the main theorem}

With the assumption in the main theorems, has scaling invariance. By taking $r = 2R$, $\rho=R$, $c_{1}= \sup_{k \in K_{+}} [D_{x'}\varphi_{k}]_{C^{\gamma}(Q_{3R}')}$ and $c_{2} = \sup_{k \in K_{+}} \| \varphi_{k} \|_{L^{\infty}(Q_{3R}')} \leq 4R + 2nR  \sup_{k \in K_{+}} \| D_{x'}\varphi_{k} \|_{L^{\infty}(Q_{3R}')}$ in  Lemma \ref{GSS300}, we use the condition \eqref{minimum} to find that 
\begin{equation*}\begin{aligned}\label{}
& \left| D_{x'}\varphi_{l}(x') - D_{x'}\varphi_{k}(x') \right| \\
& \ \leq 6 \left[ R^{\gamma}  \sup_{k \in K_{+}} [D_{x'}\varphi_{k}]_{C^{\gamma}(Q_{3R}')} + 1 + n  \sup_{k \in K_{+}} \| D_{x'}\varphi_{k} \|_{L^{\infty}(Q_{3R}')} \right]^{\frac{1}{\gamma+1}} \left[ \frac{ \varphi_{l}(x') - \varphi_{k}(x') }{R} \right]^{\frac{\gamma}{\gamma+1}} 
\end{aligned}\end{equation*}
for any $x' \in Q_{2R}'$ and $k,l \in K_{+}$. Since $\mu = \frac{\gamma}{2(\gamma+1)}$, we have that
\begin{equation*}\begin{aligned}
\left[ D_{x'}\varphi_{k} \right]_{C^{2\mu}(Q_{3R}')}
& = \sup_{ x', y' \in Q_{3R}' } \frac{ \left| D_{x'}\varphi_{k}(x') - D_{x'}\varphi_{k}(y') \right| }{ |x'-y'|^{\frac{\gamma}{\gamma+1}} } \\
& \leq (3nR)^{\frac{ \gamma^{2} }{\gamma+1}}  \sup_{ x', y' \in Q_{3R}' } \frac{  \left| D_{x'}\varphi_{k}(x') - D_{x'}\varphi_{k}(y') \right| }{ |x'-y'|^{\gamma} } \\
& =  (3nR)^{\frac{ \gamma^{2} }{\gamma+1}} [D_{x'}\varphi_{k}]_{C^{\gamma}(Q_{3R}')}.
\end{aligned}\end{equation*}
So for the following universal constant
\begin{equation}\label{kappa}
\kappa =  18n \left( 1 + R^{\gamma}  \sup_{k \in K_{+}} [D_{x'}\varphi_{k}]_{C^{\gamma}(Q_{3R}')} +   \sup_{k \in K_{+}} \| D_{x'}\varphi_{k} \|_{L^{\infty}(Q_{3R}')} \right),
\end{equation}
one can show that
\begin{equation}\label{PMT160}
|D_{x'}\varphi_{l}(x') - D_{x'}\varphi_{k}(x')|
\leq \kappa \left( \frac{ \varphi_{l}(x') - \varphi_{k}(x') }{R} \right)^{2\mu} 
\ \ \quad  \left( x' \in Q_{2r}, ~ k,l \in K_{+} \right),
\end{equation}
and
\begin{equation}\label{PMT170}
|D_{x'}\varphi_{k}(x') - D_{x'}\varphi_{k}(y')|
\leq \kappa  \left( \frac{ |x'-y'| }{R} \right)^{2\mu}
\qquad  \left( x',y' \in Q_{2r}', ~ k \in K_{+} \right).
\end{equation}

To prove H\"{o}lder continuity of $U$ in Theorem \ref{Theorem_of_U}, we use the following Campanato type embedding.

\begin{prop}\label{PMTS300}
Suppose that $h \in L^{2}(Q_{2R}(z))$ satisfies 
\begin{equation*}
\int_{Q_{R}(y)} | h - (h)_{Q_{r}(y)}|^{2} \, dx 
\leq M^{2} r^{n+2\gamma} 
\qquad 
\left( y \in Q_{R}(z), \ r \in (0,R] \right)
\end{equation*}
for some $\gamma \in (0,1)$. Then
\begin{equation*}
[h]_{C^{\gamma}(Q_{R}(z))}
\leq c M.
\end{equation*}
\end{prop}

\begin{proof}[Proof of Theorem \ref{Theorem_of_U}]
Let $u$ be a weak solution of 
\begin{equation*}\label{}
D_{\alpha} \left[ A_{ij}^{\alpha \beta} D_{\beta}u^{j} \right] = D_{\alpha} F^{i}_{\alpha} \quad \text{ in } \quad Q_{r}(z).
\end{equation*}
We first prove  \eqref{Theorem_of_U_Estimate2}. By comparing \eqref{PMT160} and \eqref{PMT170} with \eqref{PHN230} and \eqref{PHN235} respectively, we apply Lemma \ref{PHNS600} and Lemma \ref{PHNS800} respect to the point $z \in Q_{R}$ instead of the origin to find that for a sufficiently small universal constant $\varepsilon \in (0,1]$, if $ \kappa \left( \frac{ r }{ R } \right)^{2\mu} \leq \varepsilon $ then 
\begin{equation}\begin{aligned}\label{PMT220}
& \frac{1}{\rho^{2\mu}}  \mint_{ Q_{\rho}(z) } \left| U  - (U)_{ Q_{\rho}(z) } \right|^{2} \, dx \\
& \quad \leq  \frac{c}{ r^{ 2\mu }}
\mint_{ Q_{r}(z) } \left| U - (U)_{ Q_{r}(z) } \right|^{2} \, dx \\
& \qquad +  \frac{c}{ R^{2\mu} } \left( \kappa \mint_{ Q_{r}(z) } |U|^{2} \, dx 
+ \kappa \| F \|_{L^{\infty}(Q_{r}(z))}^{2} + R^{2\mu} \sup_{ k \in K } [ F ]_{C^{\mu}(Q_{r}^{k}(z))}^{2}    \right),
\end{aligned}\end{equation}
and
\begin{equation}\begin{aligned}\label{PMT230}
\mint_{Q_{\rho}(z) } \left| U  \right|^{2} \, dx
\leq c \left[  \mint_{Q_{r}(z) } \left| U  \right|^{2} \, dx
+   \| F \|_{L^{\infty}(Q_{r}^{}(z) )}^{2} 
+ R^{2\mu} \sup_{ k \in K } [ F ]_{C^{\mu}(Q_{r}^{k}(z))}^{2}     \right],
\end{aligned}\end{equation}
for any $z \in Q_{R}$ and $0 < \rho \leq r$. For the simplicity, set $\bar{\varepsilon} = \left( \kappa^{-1} \varepsilon \right)^{\frac{1}{2\mu}}$ which is a univeral constant.

If $  \bar{\varepsilon} R  < \rho \leq R$ then we have that $ \bar{\varepsilon} R  < \rho \leq r \leq R$. So \eqref{Theorem_of_U_Estimate1} and \eqref{Theorem_of_U_Estimate2} holds when $  \bar{\varepsilon} R  < \rho \leq R$. So suppose that $\rho \leq \bar{\varepsilon} R $.

If $0 < \rho \leq r \leq \bar{\varepsilon} R $ then  \eqref{Theorem_of_U_Estimate1} and   \eqref{Theorem_of_U_Estimate2} holds from \eqref{PMT220} and \eqref{PMT230}. So assume that $0 < \rho \leq \bar{\varepsilon} R  \leq r \leq R$. Then  by \eqref{PMT220} and \eqref{PMT230},
\begin{equation*}\begin{aligned}\label{}
& \frac{1}{\rho^{2\mu}}  \mint_{ Q_{\rho}(z) } \left| U  - (U)_{ Q_{\rho}(z) } \right|^{2} \, dx \\
& \quad \leq  \frac{c}{ R^{ 2\mu }}
\mint_{ Q_{ \bar{\varepsilon}R }(z) } \left| U - (U)_{ Q_{ \bar{\varepsilon}R }(z) } \right|^{2} \, dx \\
& \qquad +  \frac{c}{ R^{2\mu} } \left( \kappa \mint_{ Q_{ \bar{\varepsilon}R }(z) } |U|^{2} \, dx 
+ \kappa \| F \|_{L^{\infty}(Q_{ \bar{\varepsilon}R }(z))}^{2} + R^{2\mu} \sup_{ k \in K } [ F ]_{C^{\mu}(Q_{ \bar{\varepsilon}R }^{k}(z))}^{2}    \right),
\end{aligned}\end{equation*}
and
\begin{equation*}\begin{aligned}\label{}
\mint_{Q_{\rho}(z) } \left| U  \right|^{2} \, dx
& \leq c \left(  \mint_{Q_{ \bar{\varepsilon}R }(z) } \left| U  \right|^{2} \, dx
+   \| F \|_{L^{\infty}(Q_{ \bar{\varepsilon}R }^{}(z) )}^{2} 
+ R^{2\mu} \sup_{ k \in K } [ F ]_{C^{\mu}(Q_{ \bar{\varepsilon}R }^{k}(z))}^{2}     \right).
\end{aligned}\end{equation*}
So from that $\kappa$ and $\bar{\varepsilon}$ is universal, \eqref{Theorem_of_U_Estimate1} and \eqref{Theorem_of_U_Estimate2} holds when $0 < \rho \leq \bar{\varepsilon} R  \leq r \leq R$.
\end{proof}

\begin{proof}[Proof of Corollary \ref{Corollary_of_Du}]
Let $u$ be a weak solution of 
\begin{equation*}\label{}
D_{\alpha} \left[ A_{ij}^{\alpha \beta} D_{\beta}u^{j} \right] = D_{\alpha} F^{i}_{\alpha} \quad \text{ in } \quad Q_{2R}.
\end{equation*}
Since $z \in Q_{R}$ and $0 < \rho \leq r \leq R$ in \eqref{Theorem_of_U_Estimate1} and \eqref{Theorem_of_U_Estimate2} were arbitrary, by Proposition \ref{PMTS300},
\begin{equation}\label{PMT610}
[U]_{C^{\mu}(Q_{R})}^{2} 
 \leq  \frac{c}{R^{2\mu}} \left( \mint_{Q_{2R} } |U|^{2} \, dx 
+ \| F \|_{L^{\infty}(Q_{2R})}^{2} + R^{2\mu} \sup_{ k \in K } [ F ]_{C^{\mu}(Q_{2R}^{k}    )}^{2}  \right),
\end{equation}
and
\begin{equation}\label{PMT615} 
\| U \|_{L^{\infty}(Q_{R})}^{2}
\leq c \left( \mint_{ Q_{2R} } \left| U  \right|^{2} \, dx
+\| F \|_{L^{\infty}(Q_{2R}^{})}^{2} 
+ R^{2\mu} \sup_{ k \in K } [ F ]_{C^{\mu}(Q_{2R}^{k})}^{2}     \right).
\end{equation}

By taking $\zeta = Du$, from Lemma \ref{HODS600} and Lemma \ref{HODS700}, we have that
\begin{equation}\label{PMT620} 
|Du(x)| \leq c \Big[ |U(x)|  + |F(x)| \Big]  
\end{equation}
for any $x \in Q_{R}$ and
\begin{equation}\begin{aligned}\label{PMT630}
& |Du(x) - Du(y)| \\
& \quad \leq c \Big[ |U(y) - U(x)| + |F(y) - F(x)| \Big] \\
& \qquad +  c  \Big[ |U(x)| + |F(x)| \Big] \left[  \left| \pi(x) - \pi(y)\right| +  \left| \pi(x) - \pi(y)\right|^{2}  \right] \\
& \qquad +  c  \Big[ |U(x)| + |F(x)| \Big]  \sum_{1 \leq i,j \leq N}   \sum_{1 \leq \alpha, \beta \leq n  }  \left| A_{ij}^{\alpha \beta}(y)  - A_{ij}^{\alpha \beta}(x) \right|,
\end{aligned}\end{equation}
for any $x, y \in Q_{R}^{l}$ and $l \in K$. We have from Lemma \ref{GSS600} that $\pi \in C^{\mu}(Q_{2R})$ with the estimate that
\begin{equation}\label{PMT640}
\left[ \pi' \right]_{C^{\mu}(Q_{2R})} \leq c R^{- \mu}.
\end{equation}
Also we have that
\begin{equation}\label{PMT650}
R^{\mu} \left[ A_{ij}^{\alpha \beta} \right]_{C^{\mu}(Q_{2R}^{l})} \leq c
\qquad (1 \leq \alpha, \beta \leq n, \ 1 \leq i,j \leq N)
\end{equation}
for any $l \in K$. With \eqref{PMT640} and \eqref{PMT650}, we find from \eqref{PMT630} that
\begin{equation*}\begin{aligned}
\label{}
|Du(x) - Du(y)| &\leq c \bigg[ |x-y|^{\mu}  \big( [U]_{C^{\mu}(Q_{R})} + [F]_{C^{\mu}(Q_{R})} \big)\\ 
&\quad +\left( \frac{ |x-y| }{ R } \right)^{\mu} \big( \| U \|_{L^{\infty}(Q_{R})} + \| F \|_{L^{\infty}(Q_{2R}^{})} \big) \bigg],
\end{aligned}\end{equation*}
for any $x, y \in Q_{R}^{l}$ and $l \in K$. So with \eqref{PMT620}, we find from \eqref{PMT610} and \eqref{PMT615} that 
\begin{equation*}\label{} 
\| Du \|_{L^{\infty}(Q_{R})}^{2}
\leq c \left( \mint_{ Q_{2R} } \left| U  \right|^{2} \, dx
+\| F \|_{L^{\infty}(Q_{2R})}^{2} 
+ R^{2\mu} \sup_{ k \in K } [ F ]_{C^{\mu}(Q_{2R}^{k})}^{2}     \right).
\end{equation*}
and
\begin{equation*}\label{}
[Du]_{ C^{\mu} \left(Q_{R}^{l} \right) }^{2} 
 \leq  \frac{c}{R^{2\mu}} \left( \mint_{Q_{2R} } |U|^{2} \, dx 
+ \| F \|_{L^{\infty}(Q_{2R})}^{2} + R^{2\mu} \sup_{ k \in K } [ F ]_{C^{\mu}(Q_{2R}^{k}    )}^{2}  \right),
\end{equation*}
for any $l \in K$. Since $|U| \leq c \big( |Du| + |F| \big)$ in $Q_{2R}$, Corollary  \ref{Corollary_of_Du} holds.
\end{proof}

\subsection*{Acknowledgement}
Y. Kim was supported by the National Research Foundation of Korea (NRF) grant funded by the Korea Government NRF-2020R1C1C1A01013363.
P. Shin was supported by Basic Science Research Program through the National Research Foundation of Korea(NRF) funded by the Ministry of Education(No. NRF-2020R1I1A1A01066850).

\bibliographystyle{amsplain}

\end{document}